\documentclass[11pt]{article}
\usepackage[top = 1in,tmargin = 1in]{geometry}
\usepackage{graphicx} % Required for inserting images
\usepackage{subcaption}
\usepackage{amsfonts,amsmath,amssymb,amsthm,mathtools}
\usepackage[colorlinks=true,citecolor=black,linkcolor=black,urlcolor=black]{hyperref}
\usepackage{cleveref}
\usepackage[backend=biber,style=numeric]{biblatex}
\usepackage{float}
\usepackage{xcolor}
\usepackage[normalem]{ulem}
\usepackage{titling}
\usepackage{xfrac}
\usepackage{todonotes}
\usepackage{dsfont}
\usepackage{tikz}
\usepackage{multicol}
\usepackage[affil-sl]{authblk}

\usepackage[title]{appendix}
\bibliography{proof.bib}
\usepackage[left, mathlines]{lineno}
\usepackage{titling}
\numberwithin{equation}{section}
\providecommand{\keywords}[1]
{
  \small	
  \textit{Keywords:} #1
}
\providecommand{\MSCcodes}[1]
{
  \small	
  \textbf{MSC Codes:} #1
}

\newcommand\blfootnote[1]{
    \begingroup
    \renewcommand\thefootnote{}\footnote{#1}
    \addtocounter{footnote}{-1}
    \endgroup
}

\title{A Parallel, Energy-Stable Low-Rank Integrator for Nonlinear Multi-Scale Thermal Radiative Transfer}
\author[1]{Chinmay Patwardhan\thanks{Corresponding author}}
\author[2]{Jonas Kusch}
\affil[1]{Karlsruhe Institute of Technology, Germany}
\affil[2]{Norwegian University of Life Sciences, Norway}
\date{\vspace{-2em}}
\newcommand{\R}{\mathbb{R}} % Real numbers
\newcommand{\N}{\mathbb{N}} % Natural numbers

\newcommand{\norm}[1]{\left\lVert #1 \right\rVert} % general norm taking one input
\newcommand{\abs}[1]{\left\lvert #1 \right\rvert} % absolute value of the input

\newcommand{\vart}{t} % Temporal variable
\newcommand{\Ident}{\mathcal{I}} % Identitiy operator Ix=x
\newcommand{\Nx}{N_{x}} % Number of spatial cells/ intervals
\newtheorem{lemma}{Lemma}

\newtheorem{theorem}{Theorem}
\newtheorem{property}{Property}
\theoremstyle{remark}
\newtheorem{remark}{Remark}
\newcommand{\hindexp}[2]{{#1}^{+}_{\sfrac{#2}{2}}}
\newcommand{\hindexm}[2]{{#1}^{-}_{\sfrac{#2}{2}}}
\newcommand{\Deltaxy}{\Delta \zeta}
\begin{document}

\maketitle
%%%%%%%%%%%%%%%%%%%%%%%%%%%%%%%%%%%%%%%%%%%%%%%%%%%%%%%%%%%%%%%%%%%%%%%%%%%%%%%%%%%%%%%%%%%%%%%%%%%%
\begin{abstract}
    Thermal radiative transfer models physical phenomena ranging from supernovas in astrophysics to radiation from a hohlraum striking a fusion target in plasma physics. Transport and absorption of particles in radiative transfer at different rates lead to a complex interaction between the material and particles that involves highly varying time scales. Resolving these effects can require prohibitively small step sizes, which, combined with nonlinear effects and the particle density's high-dimensional phase space, render conventional numerical methods computationally expensive. This work presents an asymptotic--preserving, mass conservative, rank-adaptive, and parallel integrator for a macro--micro decomposition-based dynamical low-rank approximation of the thermal radiative transfer equations. The proposed integrator efficiently incorporates reflection-transmission type boundary conditions in the low-rank factors. It captures the nonlinear effects of thermal radiation and is energy stable with the step size restriction capturing both hyperbolic and parabolic CFL conditions. The efficacy of the proposed integrator is demonstrated with numerical experiments.
\end{abstract}

\keywords{thermal radiative transfer, nonlinear energy stability, asymptotic--preserving, dynamical low-rank approximation, parallel BUG}

\MSCcodes{ 35L65, 35Q49, 65M12, 35B40}

\blfootnote{Chinmay Patwardhan: chinmay.patwardhan@kit.edu}
\blfootnote{Jonas Kusch: jonas.kusch@nmbu.no}
%%%%%%%%%%%%%%%%%%%%%%%%%%%%%%%%%% Section %%%%%%%%%%%%%%%%%%%%%%%%%%%%%%%%%%%
  \section{Introduction}\label{section:Introduction}
%%%%%%%%%%%%%%%%%%%%%%%%%%%%%%%%%%%%%%%%%%%%%%%%%%%%%%%%%%%%%%%%%%%%%%%%%%%%%%
The field of thermal radiative transfer models the interaction of particles traveling through and interacting with a background material. Physical phenomena governed by the thermal radiative transfer equations include star formation, supernova explosions, radiation from a hohlraum striking a fusion target, and laser wakefield acceleration driven by pressure waves. To numerically simulate such problems, particles are commonly described by a phase space density $f(t,\textbf{\textit{x}},\boldsymbol{\Omega})$ where $t$ is time, $\textbf{\textit{x}}\in\R^{3}$ is the spatial position, and $\boldsymbol{\Omega}\in\mathbb{S}^{2}$ is the direction of flight. Then, the number of particles at time $t$ with spatial position in $\mathrm{d}\textbf{\textit{x}}$ around $\textbf{\textit{x}}$ and direction of travel in $\mathrm{d}\boldsymbol{\Omega}$ around $\boldsymbol{\Omega}$ is given as $f(t,\textbf{\textit{x}},\boldsymbol{\Omega})\mathrm{d}\textbf{\textit{x}}\mathrm{d}\boldsymbol{\Omega}$.

Two central challenges exist in determining the phase space density: First, the phase space is six-dimensional, which poses a challenge to store and evolve the phase space density on a finely resolved computational grid. Second, the underlying dynamics are commonly governed on strongly varying time scales and numerical methods must be designed to accurately capture all essential solution characteristics while not having to resolve prohibitively small scales. In particular, if many particles are absorbed on a small time scale, the dynamics of the thermal radiative transfer equations asymptotically converge to a diffusive nonlinear partial differential equation called the Rosseland equation~\cite{Rosseland_approx_ref}. Numerical methods that capture a discrete analog of this behavior while not requiring the resolution of prohibitively small time scales are often called asymptotic--preserving (AP). For instance, AP unified gas kinetic scheme (UGKS)~\cite{MR3306333,MR3404516,MR4673300,MR3713434}, the high-order/low-order (HOLO) scheme~\cite{MR4316005}, and high-order IMEX schemes~\cite{MR4439880,MR3248034} are examples of AP schemes for the thermal radiative transfer equations.

To address the high-dimensionality of thermal radiative transfer problems, recently developed numerical methods employ dynamical low-rank approximations (DLRA) introduced in \cite{doi:10.1137/050639703}. The main idea of DLRA is to represent the solution as a low-rank factorization and then derive evolution equations for the low-rank factors such that the full-rank dynamics is captured as closely as possible. While DLRA can significantly reduce memory and computational requirements, it has several additional challenges. Most importantly, the evolution equations for low-rank factors are ill-conditioned, and a large amount of research has been devoted to constructing numerical time integrators that are robust to stiffness. The most frequently used robust integrators are the projector--splitting integrator \cite{doi:10.1007/s10543-013-0454-0}, and basis-update \& Galerkin (BUG) integrators \cite{doi:10.1007/s10543-021-00873-0, robustBUG,ceruti2023parallel}, which allow for an extension to higher order \cite{ceruti2024robust,kusch2024secondorder}. Besides the construction of robust integrators, numerical methods for DLRA are commonly required to preserve certain characteristics of the original problem, e.g., the preservation of local conservation laws \cite{EINKEMMER2023112060,einkemmer2023conservation,koellermeier2023macromicro,coughlin2024robust,baumann2023energy,doi:10.1137/24M1646303}, asymptotic limits \cite{ding2021dynamical,MR4253315,einkemmer2022asymptoticpreserving,doi:10.1137/24M1646303}, or stability regions \cite{MR4535339,einkemmer2022asymptoticpreserving,baumann2023energy,doi:10.1137/24M1646303}. Such structure--preserving properties are often problem-dependent and require a careful modification of the standard integrators for individual applications. Here, the augmented BUG integrator \cite{doi:10.1007/s10543-021-00873-0} has proven to be beneficial since it allows for increased flexibility to incorporate solution structures while, unlike the projector--splitting integrators, not requiring steps backward in time. The augmented BUG integrator simplifies the construction of structure--preserving low-rank integrators, however, it requires additional costs compared to the parallel BUG integrator \cite{ceruti2023parallel}, which evolves all low-rank factors in parallel while not requiring a coefficient update at an increased rank.

Several structure--preserving DLRA methods have been developed for the thermal radiative transfer equations. A DLRA scheme to solve the one-dimensional thermal radiative transfer equation has been proposed in \cite{ceruti2022dynamical}; however, without stability guarantees, it requires a parameter study to determine a sufficiently small time step size. In \cite{baumann2023energy}, a provable energy-stable and locally conservative DLRA scheme has been derived in the one-dimensional setting for the Su-Olson closure~\cite{SU19971035}. Though not physically motivated, the Su-Olson closure significantly simplifies the evolution equation as it eliminates nonlinear effects in the coupling of the material and particles. Moreover, in \cite{doi:10.1137/24M1646303}, the authors propose a DLRA method that is asymptotic--preserving, locally conservative, and energy--stable in a one-dimensional and linearized setting.

It is important to note that the stability of previously derived methods is only understood in simplified settings, using linearizations or simplified closure relations. Moreover, while the methods proposed in \cite{baumann2023energy,doi:10.1137/24M1646303} allow for rank-adaptivity, their construction as augmented BUG integrators leads to increased computational costs since the time evolution of low-rank coefficients requires a sequential time update at an increased rank. An additional challenge for DLRA schemes not addressed in \cite{baumann2023energy,doi:10.1137/24M1646303} is the efficient implementation of boundary conditions for complex two- and three-dimensional geometries. Due to the nonlinear ansatz which separates the basis in each phase space variable, describing boundary conditions for low-rank schemes is not straightforward. A few techniques for incorporating boundary conditions in the DLRA scheme have been proposed in \cite{MR4728768,MR4466549,sapsis2009dynamically}. However, since the dynamics of the parallel BUG integrator are completely determined by the initial conditions the described methods are not applicable to the parallel BUG integrator. This work aims to develop efficient structure-preserving DLRA methods for thermal radiative transfer that can be used in higher-dimensional settings. The main novelties are
\begin{itemize}
	\item \textit{A new asymptotic--preserving, mass conservative, parallel, and rank-adaptive BUG integrator for thermal radiative transfer}: We construct a novel DLRA scheme based on a nodal micro-macro decomposition of the parallel integrator in two- and three-dimensional setting. The resulting integrator simultaneously evolves all low-rank factors while not requiring an expensive sub-step at twice the current rank. The integrator captures the correct asymptotic limit and is mass conservative.
	\item \textit{nonlinear stability analysis of the full-rank macro--micro scheme as well as the macro--micro parallel BUG scheme}: We propose numerical discretizations for the full-rank and low-rank evolution equations and show that both the proposed schemes are energy stable, under mixed hyperbolic and parabolic CFL conditions, for a nonlinear closure of the thermal radiative transfer equations given by the Stefan-Boltzmann law.
    \item \textit{Efficient implementation of reflection-transmission type boundary conditions for the parallel BUG integrator}: We propose an efficient method of incorporating reflection-transmission type boundary conditions in the parallel BUG integrator for the thermal radiative transfer equations.
\end{itemize}

The rest of the paper is organized into four sections. Section~\ref{section:Background} introduces the thermal radiative transfer equations, the macro--micro decomposition, and the fundamentals of dynamical low-rank approximation and the parallel BUG integrator. Section~\ref{section:EnergyStableNMM} presents an asymptotic--preserving discretization for the macro--micro equations using discrete ordinates, a first-order upwind discretization on staggered grids, and an implicit-explicit (IMEX) time-stepping scheme. We show that the proposed full-rank scheme for the nonlinear closure given by the Stefan-Boltzmann law is energy stable under a mixed hyperbolic and parabolic CFL condition. Section~\ref{section:ParallelDLRAforNMM} proposes a computationally inexpensive and memory-efficient low-rank scheme based on the parallel BUG integrator for the macro--micro equations. It also describes an efficient algorithm for incorporating reflection-transmission type boundary conditions in the parallel BUG integrator. This scheme is shown to be asymptotic preserving and energy stable for the nonlinear closure, given by the Stefan-Boltzmann law, under the same CFL condition as the full-rank scheme. Finally, Section~\ref{section:Numericalresults} presents numerical experiments for Gaussian, Marshak wave, and hohlraum test cases. 

% %%%%%%%%%%%%%%%%%%%%%%%%%%%%%%%%%% Section %%%%%%%%%%%%%%%%%%%%%%%%%%%%%%%%%%%
\section{Background}\label{section:Background}
% %%%%%%%%%%%%%%%%%%%%%%%%%%%%%%%%%%%%%%%%%%%%%%%%%%%%%%%%%%%%%%%%%%%%%%%%%%%%%%
% This section briefly overviews the thermal radiative transfer equations and the dynamical low-rank approximation \cite{doi:10.1137/050639703} for time-dependent problems.

%%%%%%%%%%%%%%%%%%%%%%%%%%%%%%% Sub-section %%%%%%%%%%%%%%%%%%%%%%%%%%%%%%%%%%
\subsection{Thermal radiative transfer equations}
%%%%%%%%%%%%%%%%%%%%%%%%%%%%%%%%%%%%%%%%%%%%%%%%%%%%%%%%%%%%%%%%%%%%%%%%%%%%%%

The gray thermal radiative transfer equations model the interaction of radiation particles with the background material through the interplay of radiation and material heat. In dimensionless form, they read
\begin{subequations}\label{eq:RTE}
\begin{align}
    \frac{\varepsilon^{2}}{c}\partial_{t}f + \varepsilon~\boldsymbol{\Omega}\cdot\nabla_{\textit{\textbf{x}}}f &= \sigma^{a}(B(T) - f) + \sigma^{s}(\phi - f),\label{eq:RTEKin}\\
        \varepsilon^{2}c_{\nu}\partial_{t}T &= \int_{\mathbb{S}^{2}}\sigma^{a}(f- B(T))~\mathrm{d}\boldsymbol{\Omega}.\label{eq:RTETemp} %\\
\end{align}
\end{subequations}
Here, $f(t,\textbf{\textit{x}},\boldsymbol{\Omega})$ describes the particle density at time $t\in\R_{\geq0}$, position $\textbf{\textit{x}}= (x,y,z)\in \mathcal{D}\subset\R^{3}$ and direction of flight $\boldsymbol{\Omega}=(\Omega_{x},\Omega_{y},\Omega_{v})\in\mathbb{S}^{2}$. The material temperature, denoted by $T(t,\textbf{\textit{x}})$, varies in time and space and the specific heat of the material is denoted by $c_{\nu}$. $\varepsilon$, known as the Knudsen number, specifies the ratio of the mean free path of particles between collisions to the relevant spatial scale. Since most particles fly through the material without any interaction we specify the probability of the different types of interaction. The two main particle-material interactions are absorption and scattering, and their probabilities are given by $\sigma^{a}(\textbf{\textit{x}})$ and $\sigma^{s}(\textbf{\textit{x}})$, respectively. We define the total cross-section as $\sigma^{t}(\textbf{\textit{x}}) = \sigma^{a}(\textbf{\textit{x}}) + \sigma^{s}(\textbf{\textit{x}})$ and assume that $\sigma^{a}(\textbf{\textit{x}})\geq\sigma^{a}_{0}>0$ and $\sigma^{t}(\textbf{\textit{x}}) \geq \sigma^{t}_{0} > 0 $.  We introduce the short-hand notation $ \langle \cdot \rangle = \int_{\mathbb{S}^{2}}\cdot\hspace{1mm}\mathrm{d}\boldsymbol{\Omega} $ to denote the integration over the unit sphere $\mathbb{S}^{2} $. Then, the scalar flux, $\phi$, is defined as $\phi(t,\textbf{\textit{x}}) \coloneqq  \frac{1}{4\pi}\langle f(t,\textbf{\textit{x}},\boldsymbol{\Omega})\rangle $. 

The absorption of particles by the material increases its temperature and, due to blackbody radiation, the material emits particles proportional to the fourth power of its current temperature. To be precise, the rate at which particles are emitted by the material, represented by $B(T)$, is given by the Stefan-Boltzmann law which reads 
\begin{displaymath}
    B(T) = \frac{ac}{4\pi}T^{4}
\end{displaymath}
where $a$ is the radiation constant and $c$ is the speed of light. 

To complete the description of the thermal radiative transfer equations we must specify $f$ and $T$ at the initial time $t_{0}\in\R_{\geq 0}$ and the boundaries of the domain. The initial conditions are given by
\begin{displaymath}
    f(t_{0},\textbf{\textit{x}},\boldsymbol{\Omega}) = f_{I}(\textbf{\textit{x}},\boldsymbol{\Omega}),
         \quad T(t_{0},\textbf{\textit{x}}) = T_{I}(\textbf{\textit{x}}), \quad \textbf{\textit{x}}\in\mathcal{D},\boldsymbol{\Omega}\in\mathbb{S}^{2}.
\end{displaymath}
Let $\partial\mathcal{D}$ denote the boundary of $\mathcal{D}$ and $\widehat{\textbf{\textit{x}}}\in\partial\mathcal{D}$ be a point on the boundary. If $\textbf{\textit{n}}(\widehat{\textbf{\textit{x}}})$ denotes the outward unit normal to the boundary at $\widehat{\textbf{\textit{x}}}$, the reflection-transmission type boundary conditions for the thermal radiative transfer equations are given by
\begin{subequations}
    \begin{equation}\label{eq:RTEBC_parden}
         f(t,\widehat{\textbf{\textit{x}}},\boldsymbol{\Omega}) = 
         \begin{cases}
             \rho(\textbf{\textit{n}}\cdot\boldsymbol{\Omega})f(t,\widehat{\textbf{\textit{x}}},\boldsymbol{\Omega}) + (1-\rho(\textbf{\textit{n}}\cdot\boldsymbol{\Omega}))f_{B}(\widehat{\textbf{\textit{x}}},\boldsymbol{\Omega}),  &\mathrm{if}~\textbf{\textit{n}}\cdot\boldsymbol{\Omega}<0\\
             f(t,\widehat{\textbf{\textit{x}}},\boldsymbol{\Omega}), &\mathrm{if}~\textbf{\textit{n}}\cdot\boldsymbol{\Omega}\geq0
         \end{cases},
    \end{equation}
    \begin{equation}
        T(t,\widehat{\textbf{\textit{x}}}) = T_{B}(\widehat{\textbf{\textit{x}}}),
    \end{equation}
\end{subequations}
where $f_{B}$ denotes the particle density transmitted into the domain from outside, $0\leq \rho \leq 1$ specifies reflectivity and $\boldsymbol{\Omega}' = \boldsymbol{\Omega} - 2\textbf{\textit{n}}(\textbf{\textit{n}}\cdot\boldsymbol{\Omega})$. Note that $\rho = 1$ corresponds to a purely reflective boundary while $\rho = 0$ corresponds to a transmission or inflow boundary condition.

Absorption of a large number of particles at small time scales is equivalent to $\varepsilon \to 0$. In this limiting case, $\varepsilon \to 0$, the particle density $f = B(T)$ and the evolution of the material temperature $T$ is given by the nonlinear diffusion-type equation known as the Rosseland equation \cite{Rosseland_approx_ref} which reads
\begin{equation}\label{eq:RosselandApprox}
    c_{\nu}\partial_{t}T + \frac{4\pi}{c}\partial_{t}B(T) = \frac{4\pi}{3}\nabla_{\textbf{\textit{x}}}\cdot\left( \frac{1}{\sigma^{t}}\nabla_{\textbf{\textit{x}}}B(T) \right).
\end{equation}
% For further details on the derivation of the Rosseland equation from \eqref{eq:RTE}, we refer the reader to Appendix~\ref{appendix:AsymptoticAnalysis}.

 \begin{remark}\label{remark:EquibBC}
      The initial and boundary conditions for the Rosseland equation \eqref{eq:RosselandApprox} are found by solving an initial and boundary layer problem. A detailed analysis can be found in \cite{MR1842224,KLAR_SIEDOW_1998}. However, treating boundary layer problems is out of the scope of this work and thus we adopt the so-called equilibrium boundary conditions proposed in \cite{MR1842224}. That is, we assume $(1-\rho)(f_{B}-B(T_{B})) = 0$ almost everywhere on $\partial\mathcal{D}\times\mathbb{S}^{2}_{-} = \{(\widehat{\textbf{\textit{x}}},\boldsymbol{\Omega}) \text{ s.t. } \widehat{\textbf{\textit{x}}}\in\partial\mathcal{D},~ \textbf{\textit{n}}(\widehat{\textbf{\textit{x}}})\cdot\boldsymbol{\Omega}<0 \} $. Then, if $\mathcal{D}\subset\R^{3}$ has a smooth boundary $\partial\mathcal{D}$ and we assume that the initial and the boundary data is bounded and sufficiently smooth on the domain, the solution pair $(f,T)$ converges as $\varepsilon \to 0$ to the solution of the Rosseland equation \cite{MR1842224}. The initial and boundary conditions for the Rosseland equation are then given by $T_{I}(\textbf{\textit{x}})$ and $T_{B}(\widehat{\textbf{\textit{x}}})$, respectively. The reader is referred to \cite{MR1842224} for further details. 
 \end{remark}

%%%%%%%%%%%%%%%%%%%%%%%%%%%%% Sub-section %%%%%%%%%%%%%%%%%%%%%%%%%%%%%%%%
\subsubsection{Macro--micro decomposition}
%%%%%%%%%%%%%%%%%%%%%%%%%%%%%%%%%%%%%%%%%%%%%%%%%%%%%%%%%%%%%%%%%%%%%%%%%%%%%%
The thermal radiative transfer equations involve effects varying at different time scales. Thus, to avoid mixing scales we use a macro-micro decomposition \cite{MR2460781} which decomposes the particle density into unscaled equilibrium variables and scaled non-equilibrium variables. Specifically, we use the macro--micro ansatz described in \cite{MR1842224} which has the form
\begin{equation}\label{eq:macromicrodecomp}
    f(t,\textbf{\textit{x}},\boldsymbol{\Omega}) = B(T(t,\textbf{\textit{x}})) + \varepsilon g(t,\textbf{\textit{x}},\boldsymbol{\Omega}) + \varepsilon^{2}h(t,\textbf{\textit{x}}),
\end{equation}
where $\langle g \rangle = 0$. Thus, the particle density is decomposed into its angular mean $\langle f\rangle = B(T) + \varepsilon^{2}h$, and a correction term $\varepsilon g$.

Substituting this macro--micro ansatz \eqref{eq:macromicrodecomp} in the thermal radiative transfer equations \eqref{eq:RTE} and using the condition $\langle g \rangle = 0$ yields the macro--micro equations
\begin{subequations}\label{eq:MacMicSys}
     \begin{equation}\label{eq:MacMicSysmic}
        \frac{\varepsilon^{2}}{c}\partial_{t}g + \varepsilon\left( \Ident - \frac{1}{4\pi}\langle{\cdot}\rangle \right)(\boldsymbol{\Omega}\cdot\nabla_{\textbf{\textit{x}}}g) + \boldsymbol{\Omega}\cdot\nabla_{\textbf{\textit{x}}}(B(T) + \varepsilon^{2}h) = -\sigma^{t}g,
    \end{equation}
    \begin{equation}\label{eq:MacMicSysmes}
         \frac{\varepsilon^{2}}{c}\partial_{t}h + \frac{1}{c}\partial_{t}B(T) + \frac{1}{4\pi}\langle \boldsymbol{\Omega}\cdot\nabla_{\textbf{\textit{x}}}g \rangle = -\sigma^{a}h,
    \end{equation}
    \begin{equation}\label{eq:MacMicSysmac}
        c_{\nu}\partial_{t}T = 4\pi\sigma^{a} h.
    \end{equation}
\end{subequations}
The macro--micro equations are equivalent to the thermal radiative transfer equation and have the same Rosseland diffusion limit. This can be seen by comparing $\mathcal{O}(1)$ terms in \eqref{eq:MacMicSys} which yield the Rosseland equation \eqref{eq:RosselandApprox} in the limit $\varepsilon\to 0$.

%%%%%%%%%%%%%%%%%%%%%%%%%%%%% Sub-sub-section %%%%%%%%%%%%%%%%%%%%%%%%%%%%%%%%
\paragraph{Initial and boundary conditions}
%%%%%%%%%%%%%%%%%%%%%%%%%%%%%%%%%%%%%%%%%%%%%%%%%%%%%%%%%%%%%%%%%%%%%%%%%%%%%%

The initial and boundary conditions for the macro--micro equations have been derived and analyzed in \cite{MR1842224}. We present the relevant details here. The initial conditions for $g$ and $h$ can be derived from the following relations:
 \begin{displaymath} 
        g(t_{0},\textbf{\textit{x}},\boldsymbol{\Omega}) = \frac{1}{\varepsilon}\left( f_{I}(\textbf{\textit{x}},\boldsymbol{\Omega}) - \frac{1}{4\pi}\langle f_{I}(\textbf{\textit{x}},\boldsymbol{\Omega})\rangle \right), \quad h(t_{0},\textbf{\textit{x}}) = \frac{1}{\varepsilon^{2}}\left( \frac{1}{4\pi}\langle f_{I}(\textbf{\textit{x}},\boldsymbol{\Omega})\rangle - B(T_{I})(\textbf{\textit{x}}) \right).
\end{displaymath}
The boundary condition for $g$ is obtained by substituting the macro--micro ansatz \eqref{eq:macromicrodecomp} into the boundary conditions for $f$ given in~\eqref{eq:RTEBC_parden}. Thus, for $\widehat{\textbf{\textit{x}}}\in\partial\mathcal{D}$, the boundary condition for $g$ is given by:
\begin{displaymath}
    g(t,\widehat{\textbf{\textit{x}}},\boldsymbol{\Omega}) = \begin{cases}
         \rho g(t,\widehat{\textbf{\textit{x}}},\boldsymbol{\Omega}') + (1-\rho)\left( \frac{f_{B}(t,\widehat{\textbf{\textit{x}}},\boldsymbol{\Omega})- B(T_{B})(t,\widehat{\textbf{\textit{x}}})}{\varepsilon} - \epsilon h(t,\widehat{\textbf{\textit{x}}}) \right), &\mathrm{if}~\textbf{\textit{n}}\cdot\boldsymbol{\Omega}<0,\\
         g(t,\widehat{\textbf{\textit{x}}},\boldsymbol{\Omega}), &\mathrm{if}~\textbf{\textit{n}}\cdot\boldsymbol{\Omega}\geq 0.
    \end{cases}
\end{displaymath}
The boundary condition for $h$ is set such that the macro--micro decomposition is consistent at the boundaries, i.e., we compute $h$ such the $\langle g\rangle = 0$ is satisfied at the boundaries. Thus, $h$ must satisfy the following condition at the boundary:
\begin{displaymath}
    \int_{\textbf{\textit{n}}\cdot\boldsymbol{\Omega}>0}(1+\rho(-\textbf{\textit{n}}\cdot\boldsymbol{\Omega}))g~\mathrm{d}\boldsymbol{\Omega} + \int_{\textbf{\textit{n}}\cdot\boldsymbol{\Omega}<0}(1-\rho)\left( \frac{f_{B}- B(T_{B})}{\varepsilon} - \epsilon h \right)~\mathrm{d}\boldsymbol{\Omega} = 0, \quad \widehat{\textbf{\textit{x}}}\in\partial\mathcal{D},
\end{displaymath}
when the boundary is not purely reflective, i.e. $0\leq\rho<1$. For a purely reflective boundary, i.e., $\rho = 1$, the boundary conditions for $h$ are not required \cite{MR1842224}. Since we assume equilibrium boundaries the above boundary conditions for $g$ and $h$ simplify to
\begin{displaymath}
    \begin{aligned}
        g(t,\widehat{\textbf{\textit{x}}},\boldsymbol{\Omega}) &= \begin{cases}
            \rho g(t,\widehat{\textbf{\textit{x}}},\boldsymbol{\Omega}') - (1-\rho)\varepsilon h(t,\widehat{\textbf{\textit{x}}}),& \mathrm{if}~\textbf{\textit{n}}\cdot\boldsymbol{\Omega}<0,\\
            g(t,\widehat{\textbf{\textit{x}}},\boldsymbol{\Omega}'), &\mathrm{if}~\textbf{\textit{n}}\cdot\boldsymbol{\Omega}\geq0,
        \end{cases}\\
        \varepsilon\int_{\textbf{\textit{n}}\cdot\boldsymbol{\Omega}<0}(1-\rho)h ~\mathrm{d}\boldsymbol{\Omega} &= \int_{\textbf{\textit{n}}\cdot\boldsymbol{\Omega}>0}(1+\rho(-\textbf{\textit{n}}\cdot\boldsymbol{\Omega}))g~\mathrm{d}\boldsymbol{\Omega}.
    \end{aligned}
\end{displaymath}

%%%%%%%%%%%%%%%%%%%%%%%%%%%%% Sub-section %%%%%%%%%%%%%%%%%%%%%%%%%%%%%%%%%%%%
\subsection{Dynamical low-rank approximation} \label{section:DLRA}
%%%%%%%%%%%%%%%%%%%%%%%%%%%%%%%%%%%%%%%%%%%%%%%%%%%%%%%%%%%%%%%%%%%%%%%%%%%%%%
In this section, we present a summary of DLRA~\cite{doi:10.1137/050639703} for solving time-dependent problems on the manifold of low-rank matrices. The central idea is to evolve the solution on the low-rank manifold by projecting the dynamics onto the tangent space. Let $\textbf{g}(t) \in \R^{N_{\textbf{\textit{x}}}\times N_{q}} $ denote the matrix form of $g$ discretized in $\textbf{\textit{x}}$ and $\boldsymbol{\Omega}$ with $N_{\textbf{\textit{x}}}$ spatial cells and $N_{q}$ discrete directions. That is, $ (\textbf{g})_{ik} = g(t,\textbf{\textit{x}}_{i},\boldsymbol{\Omega}_{k})$. Then, \eqref{eq:MacMicSysmic} can be written as the matrix-valued differential equation
\begin{displaymath}
    \dot{\textbf{g}}(t) = \textbf{F}(t,\textbf{g}(t)).
\end{displaymath}
Let $\mathcal{M}_{r}$ denote the manifold of $N_{\textbf{\textit{x}}}\times N_{q}$ rank-$r$ matrices. Then, a low-rank approximation $\textbf{Y}(t)\in\mathcal{M}_{r}$ of $\textbf{g}(t)$ admits the factorization 
\begin{displaymath}
    \textbf{Y}(t) = \textbf{X}(t)\textbf{S}(t)\textbf{V}(t)^{\top},
\end{displaymath}
where $\textbf{X}(t)\in\R^{N_{\textbf{\textit{x}}}\times r}$, $\textbf{V}(t)\in\R^{N_{q}\times r}$ are orthonormal basis matrices and $\textbf{S}(t)\in\R^{r\times r}$ is the invertible coefficient matrix. Let $ \mathcal{T}_{\textbf{Y}}\mathcal{M}_{r} $ denote the tangent space to $\mathcal{M}_{r}$ at $\textbf{Y}$. Then, to ensure that $\textbf{Y}(t)\in\mathcal{M}_{r}$ approximates $\textbf{g}(t)$ we find $\dot{\textbf{Y}}(t)\in\mathcal{T}_{\textbf{Y}(t)}\mathcal{M}_{r}$ such that for all times $t$ the minimization problem 
\begin{displaymath}
    \underset{\dot{\textbf{Y}}(t)\in\mathcal{T}_{\textbf{Y}(t)}\mathcal{M}_{r}}{\mathrm{min}}\left\lVert \dot{\textbf{Y}}(t) - \textbf{F}(t,\textbf{Y}(t)) \right\rVert_{F}
\end{displaymath}
is satisfied with respect to the Frobenius norm $ \lVert \cdot \rVert_{F} $. Reformulating the minimization problem as a Galerkin condition on the tangent space, we see that the minimization problem is satisfied by
\begin{displaymath}\label{eq:Peq}
    \dot{\textbf{Y}}(t) = \textbf{P}(\textbf{Y}(t))\textbf{F}(t,\textbf{Y}(t)),
\end{displaymath}
where $\textbf{P}(\textbf{Y}(t))$ denotes the orthogonal projection onto the tangent space and has the form $\textbf{P}(\textbf{Y})\textbf{Z} = \textbf{X}\textbf{X}^{\top}\textbf{Z} - \textbf{X}\textbf{X}^{\top}\textbf{Z}\textbf{V}\textbf{V}^{\top} + \textbf{Z}\textbf{V}\textbf{V}^{\top}$~\cite[Lemma~4.1]{doi:10.1137/050639703}, for any $\textbf{Z}\in\R^{N_{\textbf{\textit{x}}}\times N_{q}}$. This allows us to derive differential equations for $\textbf{X}(t)$, $\textbf{S}(t)$, and $\textbf{V}(t)$. However, these equations are stiff in the case of rank over-approximation due to the presence of near-zero singular values \cite{doi:10.1137/050639703}.

In recent years, several structure--preserving integrators~\cite{doi:10.1007/s10543-013-0454-0,doi:10.1007/s10543-021-00873-0,ceruti2023parallel,ceruti2024robust,kusch2024secondorder} have been developed that are robust to this stiffness when solving the projected equation \eqref{eq:Peq}. In this work, we use the parallel BUG integrator~\cite{ceruti2023parallel} which evolves all the factors $\textbf{X}$, $\textbf{S}$, and $\textbf{V}$ in parallel resulting in a reduced number of potentially expensive projection operations while also allowing for rank-adaptivity. One step of the parallel BUG integrator for an initial rank-$r$ approximation $\textbf{Y}^{0} = \textbf{X}^{0}\textbf{S}^{0}\textbf{V}^{0,\top}$ at time $t_{0}$ updates the factors to $\textbf{X}^{1}$, $\textbf{S}^{1}$, $\textbf{V}^{1}$ of rank $r_{1}$ at $t_{1} = t_{0} + \Delta t$ in three steps:
\begin{enumerate}
    \item \textbf{Parallel update}: Update $\textbf{X}$, $\textbf{S}$, and $\textbf{V}$ in parallel and construct augmented basis matrices $\widehat{\textbf{X}} \in \R^{N_{\textbf{\textit{x}}}\times 2r}$ and $\widehat{\textbf{V}} \in \R^{N_{q}\times 2r}$:\\
    
    \textbf{K-step}: For $\textbf{K}(t_{0}) = \textbf{X}^{0}\textbf{S}^{0}$ solve from $t_{0}$ to $t_{1}$
    \begin{displaymath}
        \dot{\textbf{K}}(t) = \textbf{F}(t,\textbf{K}(t)\textbf{V}^{0,\top})\textbf{V}^{0}. 
    \end{displaymath}
    Determine $\widehat{\textbf{X}} = (\textbf{X}^{0}, \widetilde{\textbf{X}}^{1})\in \R^{N_{\textbf{\textit{x}}}\times 2r} $ as an orthonormal basis of $ (\textbf{X}^{0},\textbf{K}(t_{1})) $ (e.g., by QR decomposition). The $N_{\textbf{\textit{x}}}\times r$ matrix $\widetilde{\textbf{X}}^{1}$ is filled with zero columns if $\mathrm{rank}(\textbf{X}^{0},\textbf{K}(t_{1}))< r$. Compute and store the matrix $\widetilde{\textbf{S}}^{K}_{1} = \widetilde{\textbf{X}}^{1,\top}\textbf{K}(t_{1}) $.\\
    \textbf{L-step}: For $\textbf{L}(t_{0}) = \textbf{V}^{0}\textbf{S}^{0,\top}$ solve from $t_{0}$ to $t_{1}$
    \begin{displaymath}
        \dot{\textbf{L}}(t) = \textbf{F}(t,\textbf{X}^{0}\textbf{L}(t)^{\top})^{\top}\textbf{X}^{0}. 
    \end{displaymath}
    Determine $\widehat{\textbf{V}} = (\textbf{V}^{0}, \widetilde{\textbf{V}}^{1})\in \R^{N_{q}\times 2r} $ as an orthonormal basis of $ (\textbf{V}^{0},\textbf{L}(t_{1})) $ (e.g., by QR decomposition). The $N_{q}\times r$ matrix $\widetilde{\textbf{V}}^{1}$ is filled with zero columns if $\mathrm{rank}(\textbf{V}^{0},\textbf{L}(t_{1}))< r$. Compute and store the matrix $\widetilde{\textbf{S}}^{L}_{1} = \textbf{L}(t_{1})^{\top}\widetilde{\textbf{V}}^{1} $.\\
    \textbf{S-step}: For $ \overline{\textbf{S}}(t_{0}) = \textbf{S}^{0} $ solve from $t_{0}$ to $t_{1}$
    \begin{displaymath}
        \dot{\overline{\textbf{S}}}(t) = \textbf{X}^{0,\top}\textbf{F}(t,\textbf{X}^{0}\overline{\textbf{S}}(t)\textbf{V}^{0,\top})\textbf{V}^{0}.
    \end{displaymath}
    \item \textbf{Augmentation}: Construct the augmented coefficient matrix $\widehat{\textbf{S}}\in\R^{2r\times 2r}$
    \begin{displaymath}
            \widehat{\textbf{S}} = \begin{bmatrix}
                \overline{\textbf{S}}(t_{1}) & \widetilde{\textbf{S}}^{L}_{1}\\
                \widetilde{\textbf{S}}^{K}_{1} & \textbf{0}
            \end{bmatrix}.
    \end{displaymath}
    \item \textbf{Truncate}: Compute the singular value decomposition of the coefficient matrix $\widehat{\textbf{S}} = \widehat{\textbf{P}}\widehat{\boldsymbol{\Sigma}}\widehat{\textbf{Q}}^{\top}$ where $\widehat{\boldsymbol{\Sigma}}$ has the singular values, $\widehat{\sigma}_{j}$, of $\widehat{\textbf{S}}$ on its diagonal. The new rank $r_{1}$ is chosen as the minimal $r_{1}<2r$ such that, for a given tolerance $\vartheta$, the following inequality is satisfied
    \begin{displaymath}
        \left( \sum_{j = r_{1}+1}^{2r} \widehat{\sigma}_{j}^{2} \right)^{1/2} \leq \vartheta.
    \end{displaymath}
    The updated coefficient matrix $\textbf{S}_{1}$ is set as the diagonal matrix containing the first $r_{1}$ singular values of $\widehat{\boldsymbol{\Sigma}}$. To set the updated basis, we define $\textbf{P}_{r_{1}}$ and $\textbf{Q}_{r_{1}}$ to be the matrices containing the first $r_{1}$ columns of $\widehat{\textbf{P}}$ and $\widehat{\textbf{Q}}$, respectively. Then, the updated factors are set as $\textbf{X}^{1} = \widehat{\textbf{X}}\textbf{P}_{r_{1}}$ and $\textbf{V}^{1} = \widehat{\textbf{V}}\textbf{Q}_{r_{1}}$. 
\end{enumerate}
The approximation at time $\vart_{1}$ is then $\textbf{Y}^{1} = \textbf{X}^{1}\textbf{S}^{1}\textbf{V}^{1,\top}$.

%%%%%%%%%%%%%%%%%%%%%%%%%%%%%%%%%% Section %%%%%%%%%%%%%%%%%%%%%%%%%%%%%%%%%%%
\section{An AP scheme for the macro--micro equations}\label{section:EnergyStableNMM}
To simulate the thermal radiative transfer equations, we discretize the macro--micro equations \eqref{eq:MacMicSys} in its phase space variables $(t,\textbf{\textit{x}},\boldsymbol{\Omega})$. There are three important considerations when constructing a numerical scheme for the macro-micro equations. First, the numerical scheme should consistently discretize the Rosseland equation in the diffusive limit $\varepsilon\to 0$, i.e., the scheme should be AP. However, in the diffusive limit, the right-hand side of the macro--micro equations become stiff, and naive schemes require prohibitively small $\mathcal{O}(\varepsilon)$ step sizes for stability. Thus, the proposed numerical scheme should not require prohibitively small step sizes to capture the correct dynamics of the system in the diffusive limit. Finally, the numerical scheme should reduce the computational costs and memory requirements arising from the high-dimensional phase space of $g$.

In this section, we propose a numerical scheme based on the discrete ordinates method~\cite{Lewis1984Sn} for $\boldsymbol{\Omega}$, a first-order upwind discretization on staggered grids~\cite {MR3459981,MR1925043} for $\textbf{\textit{x}}$, and a first-order IMEX scheme~\cite{doi:10.1137/24M1646303} for $t$. In Section~\ref{section:NMMAPproperty} and \ref{section:EnergyStabNMM} the proposed scheme, dubbed the full-rank macro--micro scheme, is shown to be AP and energy stable for the nonlinear Stefan-Boltzmann closure with mixed hyperbolic and parabolic CFL conditions. 
%%%%%%%%%%%%%%%%%%%%%%%%%%%%%%%%%%%%%%%%%%%%%%%%%%%%%%%%%%%%%%%%%%%%%%%%%%%%%%

%%%%%%%%%%%%%%%%%%%%%%%%%%%%% Sub-section %%%%%%%%%%%%%%%%%%%%%%%%%%%%%%%%
\subsection{Angular discretisation}\label{section:NodalMM}
%%%%%%%%%%%%%%%%%%%%%%%%%%%%%%%%%%%%%%%%%%%%%%%%%%%%%%%%%%%%%%%%%%%%%%%%%%%%%%
The discrete ordinates, or S$_{N}$, method \cite{Lewis1984Sn} uses a quadrature rule on the unit sphere $\mathbb{S}^{2}$ and solves the macro--micro equations in these discrete directions. The $\mathrm{S}_{N}$ method has been a popular choice for angular discretization of kinetic equations due to its ease of implementation and handling of boundary conditions~\cite{MR4315485,MR4253315} compared to other methods like the method of moments \cite{case1967linear} or the minimal entropy method \cite{10.1063/1.869849}. 

In this work, we use the tensorized Gauss-Legendre quadrature rule, also called the product quadrature, of order $q$ on the unit sphere. Let $\{\boldsymbol{\Omega}^{\ell}\}_{\ell = 1,\ldots,N_{q}}$ be $N_{q}$ quadrature points with the associated (quadrature) weights $\left\{ w^{\ell} \right\}_{\ell = 1,\ldots,N_{q}}$ and let $g_{\ell}(\cdot,\cdot) \coloneqq g(\cdot,\cdot,\boldsymbol{\Omega}^{\ell})$. Then, the macro--micro equations in the $\ell$th discrete direction are given by
\begin{subequations}
    \begin{equation*}
        \frac{\varepsilon^{2}}{c}\partial_{t}g_{\ell} + \varepsilon\left( \boldsymbol{\Omega}^{\ell}\cdot\nabla_{\textbf{\textit{x}}}g_{\ell} - \frac{1}{4\pi}\sum_{\ell'=1}^{N_{q}}w^{\ell'}\boldsymbol{\Omega}^{\ell'}\cdot\nabla_{\textbf{\textit{x}}}g_{\ell'} \right) + \boldsymbol{\Omega}^{\ell}\cdot\nabla_{\textbf{\textit{x}}}(B(T) + \varepsilon^{2}h) = -\sigma^{t}g_{\ell},
    \end{equation*}
    \begin{equation*}
        \frac{\varepsilon^{2}}{c}\partial_{t}h + \frac{1}{c}\partial_{t}B(T) + \frac{1}{4\pi}\sum_{\ell'=1}^{N_{q}}w^{\ell'}\boldsymbol{\Omega}^{\ell}\cdot\nabla_{\textbf{\textit{x}}}g_{\ell'} = -\sigma^{a}h,
    \end{equation*}
    \begin{equation*}
         c_{\nu}\partial_{t}T = 4\pi\sigma^{a}h.
    \end{equation*}
\end{subequations}
Let $\boldsymbol{\Omega}^{\ell} \coloneqq (\Omega_{x}^{\ell},\Omega_{y}^{\ell},\Omega_{z}^{\ell})^{\top}$ and $\mathcal{J}_{\textbf{\textit{x}}}\coloneqq\{ x,y,z \}$. If $\textbf{\textit{g}} \coloneqq (g_{1},\ldots,g_{N_{q}})^{\top}\in\R^{N_{q}}$, then we can write the above system of $N_{q}$ equations as
\begin{subequations}\label{eq:NodalMacMicSys}
    \begin{equation}
    \begin{split}
         \frac{\varepsilon^{2}}{c}\partial_{t}\textbf{\textit{g}} + \varepsilon\left( \textbf{I} - \frac{1}{4\pi}\mathds{1}\textbf{\textit{w}}^{\top} \right)\sum_{v\in\mathcal{J}_{\textbf{\textit{x}}}}\textbf{Q}_{v}\partial_{v}\textbf{\textit{g}} + \sum_{v\in\mathcal{J}_{\textbf{\textit{x}}}}\textbf{Q}_{v}\mathds{1}\partial_{v}(B(T) + \varepsilon^{2}h) = -\sigma^{t}\textbf{\textit{g}},
    \end{split}
    \end{equation}
    \begin{equation}
        \frac{\varepsilon^{2}}{c}\partial_{t}h + \frac{1}{c}\partial_{t}B(T) + \frac{1}{4\pi}\textbf{\textit{w}}^{\top}\sum_{v\in\mathcal{J}_{\textbf{\textit{x}}}}\textbf{Q}_{v}\partial_{v}\textbf{\textit{g}} = -\sigma^{a}h,
    \end{equation}
    \begin{equation}
        c_{\nu}\partial_{t}T = 4\pi\sigma^{a}h,
    \end{equation}
\end{subequations}
where $\textbf{\textit{w}} = (w_{1},\ldots,w_{N_{q}})^{\top}$ and $ \mathds{1} = (1,\ldots,1)^{\top}$ are vectors in $\R^{N_{q}}$. In the above equations, $\textbf{I}\in\R^{N_{q}\times N_{q}}$ denotes the identity matrix and, for $v \in \mathcal{J}_{\textbf{\textit{x}}}$,
\begin{displaymath}
    \textbf{Q}_{v} \coloneqq\begin{bmatrix} 
        \Omega_{v}^{1} & &\\
        & \ddots &\\
        &&\Omega_{v}^{N_{q}}
    \end{bmatrix} \in\R^{N_{q}\times N_{q}}.
\end{displaymath}  

%%%%%%%%%%%%%%%%%%%%%%%%%%%%% Sub-section %%%%%%%%%%%%%%%%%%%%%%%%%%%%%%%%
\subsection{Spatio-temporal discretization}
%%%%%%%%%%%%%%%%%%%%%%%%%%%%%%%%%%%%%%%%%%%%%%%%%%%%%%%%%%%%%%%%%%%%%%%%%%%%%%
To discretize \eqref{eq:NodalMacMicSys} in $\textbf{\textit{x}}$, we use first-order upwind discretization on staggered grids~\cite{MR3459981,MR4253315} and a first-order IMEX scheme in $t$.  This staggered-grid approach for constructing AP schemes was first described in \cite{MR3459981} for the multi-scale transport equation and in \cite{MR4253315} for the macro--micro decomposition of the radiation transport equation. For brevity, we restrict the presentation of the discretization to two spatial dimensions on a rectangular domain, i.e., $\textbf{\textit{x}} = (x,y)\in[x_{L},x_{R}]\times[y_{B},y_{T}]\subset\R^{2}$. Note that in this projected geometry, $\boldsymbol{\Omega}\in P(\mathbb{S}^{2})$ where $P(\mathbb{S}^{2})$ is the projection of the unit sphere in $\R^{3}$ onto the two dimensional plane. This is defined as 
\begin{displaymath}
    P(\mathbb{S}^{2}) \coloneqq \left\{ (\sqrt{1-\mu^{2}}\sin(\theta),\sqrt{1-\mu^{2}}\cos(\theta))~|~ 0\leq\mu\leq 1, \theta\in[0,2\pi] \right\}.
\end{displaymath}
Note that the normalization constant $4\pi$ in \eqref{eq:RTE} is replaced by $2\pi$ in this projected geometry and $\mathcal{J}_{\textbf{\textit{x}}} = \{x,y\}$.

We construct two one-dimensional equidistant staggered grids, in $ x $ and $ y $ domains, with size $\Delta x = (x_{R}-x_{L})/N_{x}$ and $\Delta y = (y_{T}-y_{B})/N_{y}$, for some $N_{x}, N_{y}\in\N$. The interfaces between intervals are given by $\{x_{i + 1/2}\}_{i = 0}^{N_{x}}$ and $\{y_{j + 1/2}\}_{j = 0}^{N_{y}}$ while their mid-points are denoted by $\{x_{i}\}_{i = 1}^{N_{x}}$ and $\{y_{j}\}_{j = 1}^{N_{y}}$. Note that the discretization is set up such that $x_{1}$, $x_{N_{x}}$, $y_{1}$, and $y_{N_{y}}$ lie on the boundaries of the domain. Following \cite{MR3459981,MR4253315} the discretization is constructed such that 
\begin{itemize}
    \item $T$ and $h$ are evaluated at $(x_{i},y_{j})$ and $(x_{i+1/2},y_{j+1/2})$ (red colored circles in \Cref{fig:2D staggered grid}),
    \item $\textbf{\textit{g}}$ is evaluated at $(x_{i+1/2},y_{j})$ and $(x_{i},y_{j+1/2})$ (blue colored diamonds in \Cref{fig:2D staggered grid}).
\end{itemize}
For an illustration of the grid, see \Cref{fig:2D staggered grid}.

We introduce the following notation to simplify the presentation of the numerical scheme. Let $k^{\pm}_{\frac{1}{2}} \coloneqq k \pm 1/2$, such that the cell interfaces are denoted by $x_{i^{\pm}_{\sfrac{1}{2}}}$ and $y_{j^{\pm}_{\sfrac{1}{2}}}$. To further simplify notation, we define $B^{n}_{ij} \coloneqq \frac{ac}{2\pi}(T^{n}_{ij})^{4} $ at $ (x_{i},y_{j}) $ and analogously $B^{n}_{\hindexp{i}{1}\hindexp{j}{1}} \coloneqq \frac{ac}{2\pi}(T^{n}_{\hindexp{i}{1}\hindexp{j}{1}})^{4} $ at $(x_{\hindexp{i}{1}},y_{\hindexp{j}{1}}) $.
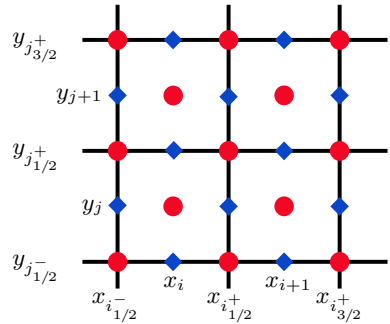
\begin{figure}[h]
    \centering
    \tikzset{every picture/.style={line width=0.75pt}} %set default line width to 0.75pt    
    % \begin{subfigure}[b]{0.39\linewidth}
    % \centering
    \begin{tikzpicture}[x=0.75pt,y=0.75pt,yscale=-1,xscale=1]
        
        %Shape: Grid [id:dp6086037867382692] 
        \draw  [draw opacity=0][line width=1.5]  (286.78,55.22) -- (440.67,55.22) -- (440.67,198.45) -- (286.78,198.45) -- cycle ; \draw  [line width=1.5]  (304.51,55.22) -- (304.51,198.45)(360.51,55.22) -- (360.51,198.45)(416.51,55.22) -- (416.51,198.45) ; \draw  [line width=1.5]  (286.78,72.95) -- (440.67,72.95)(286.78,128.95) -- (440.67,128.95)(286.78,184.95) -- (440.67,184.95) ; \draw  [line width=1.5]   ;
        %Shape: Diamond [id:dp2704383770206539] 
        \draw  [draw opacity=0][fill={rgb, 255:red, 10; green, 69; blue, 199 }  ,fill opacity=1 ] (304.51,151.7) -- (309.17,156.45) -- (304.51,161.2) -- (299.84,156.45) -- cycle ;
        %Shape: Diamond [id:dp05191167319843715] 
        \draw  [draw opacity=0][fill={rgb, 255:red, 10; green, 69; blue, 199 }  ,fill opacity=1 ] (332.51,180.2) -- (337.17,184.95) -- (332.51,189.7) -- (327.84,184.95) -- cycle ;
        %Shape: Diamond [id:dp8396711776809587] 
        \draw  [draw opacity=0][fill={rgb, 255:red, 10; green, 69; blue, 199 }  ,fill opacity=1 ] (360.51,152.2) -- (365.17,156.95) -- (360.51,161.7) -- (355.84,156.95) -- cycle ;
        %Shape: Diamond [id:dp43560726895372204] 
        \draw  [draw opacity=0][fill={rgb, 255:red, 10; green, 69; blue, 199 }  ,fill opacity=1 ] (332.89,123.94) -- (337.55,128.69) -- (332.89,133.44) -- (328.22,128.69) -- cycle ;
        %Shape: Diamond [id:dp49146238603682535] 
        \draw  [draw opacity=0][fill={rgb, 255:red, 10; green, 69; blue, 199 }  ,fill opacity=1 ] (388.51,180.2) -- (393.17,184.95) -- (388.51,189.7) -- (383.84,184.95) -- cycle ;
        %Shape: Diamond [id:dp625980383948383] 
        \draw  [draw opacity=0][fill={rgb, 255:red, 10; green, 69; blue, 199 }  ,fill opacity=1 ] (416.51,152.2) -- (421.17,156.95) -- (416.51,161.7) -- (411.84,156.95) -- cycle ;
        %Shape: Diamond [id:dp6603613912995167] 
        \draw  [draw opacity=0][fill={rgb, 255:red, 10; green, 69; blue, 199 }  ,fill opacity=1 ] (304.51,96.2) -- (309.17,100.95) -- (304.51,105.7) -- (299.84,100.95) -- cycle ;
        %Shape: Diamond [id:dp7532338499105224] 
        \draw  [draw opacity=0][fill={rgb, 255:red, 10; green, 69; blue, 199 }  ,fill opacity=1 ] (360.51,96.86) -- (365.17,101.61) -- (360.51,106.36) -- (355.84,101.61) -- cycle ;
        %Shape: Diamond [id:dp08242433686903328] 
        \draw  [draw opacity=0][fill={rgb, 255:red, 10; green, 69; blue, 199 }  ,fill opacity=1 ] (416.51,96.2) -- (421.17,100.95) -- (416.51,105.7) -- (411.84,100.95) -- cycle ;
        %Shape: Diamond [id:dp18820567957368406] 
        \draw  [draw opacity=0][fill={rgb, 255:red, 10; green, 69; blue, 199 }  ,fill opacity=1 ] (332.51,68.2) -- (337.17,72.95) -- (332.51,77.7) -- (327.84,72.95) -- cycle ;
        %Shape: Diamond [id:dp7617641125065957] 
        \draw  [draw opacity=0][fill={rgb, 255:red, 10; green, 69; blue, 199 }  ,fill opacity=1 ] (388.51,68.2) -- (393.17,72.95) -- (388.51,77.7) -- (383.84,72.95) -- cycle ;
        %Shape: Ellipse [id:dp12162460285712395] 
        \draw  [draw opacity=0][fill={rgb, 255:red, 242; green, 7; blue, 37 }  ,fill opacity=1 ] (299.6,185) .. controls (299.6,182.23) and (301.8,180) .. (304.51,180) .. controls (307.22,180) and (309.42,182.23) .. (309.42,185) .. controls (309.42,187.76) and (307.22,190) .. (304.51,190) .. controls (301.8,190) and (299.6,187.76) .. (299.6,185) -- cycle ;
        %Shape: Ellipse [id:dp6029546507689866] 
        \draw  [draw opacity=0][fill={rgb, 255:red, 242; green, 7; blue, 37 }  ,fill opacity=1 ] (299.6,128.97) .. controls (299.6,126.21) and (301.8,123.97) .. (304.51,123.97) .. controls (307.22,123.97) and (309.42,126.21) .. (309.42,128.97) .. controls (309.42,131.73) and (307.22,133.97) .. (304.51,133.97) .. controls (301.8,133.97) and (299.6,131.73) .. (299.6,128.97) -- cycle ;
        %Shape: Ellipse [id:dp8090189307248269] 
        \draw  [draw opacity=0][fill={rgb, 255:red, 242; green, 7; blue, 37 }  ,fill opacity=1 ] (299.6,72.95) .. controls (299.6,70.19) and (301.8,67.95) .. (304.51,67.95) .. controls (307.22,67.95) and (309.42,70.19) .. (309.42,72.95) .. controls (309.42,75.71) and (307.22,77.95) .. (304.51,77.95) .. controls (301.8,77.95) and (299.6,75.71) .. (299.6,72.95) -- cycle ;
        %Shape: Ellipse [id:dp23527142386059752] 
        \draw  [draw opacity=0][fill={rgb, 255:red, 242; green, 7; blue, 37 }  ,fill opacity=1 ] (355.62,72.95) .. controls (355.62,70.19) and (357.82,67.95) .. (360.53,67.95) .. controls (363.24,67.95) and (365.44,70.19) .. (365.44,72.95) .. controls (365.44,75.71) and (363.24,77.95) .. (360.53,77.95) .. controls (357.82,77.95) and (355.62,75.71) .. (355.62,72.95) -- cycle ;
        %Shape: Ellipse [id:dp4464329470756502] 
        \draw  [draw opacity=0][fill={rgb, 255:red, 242; green, 7; blue, 37 }  ,fill opacity=1 ] (355.62,128.97) .. controls (355.62,126.21) and (357.82,123.97) .. (360.53,123.97) .. controls (363.24,123.97) and (365.44,126.21) .. (365.44,128.97) .. controls (365.44,131.73) and (363.24,133.97) .. (360.53,133.97) .. controls (357.82,133.97) and (355.62,131.73) .. (355.62,128.97) -- cycle ;
        %Shape: Ellipse [id:dp4381396312087227] 
        \draw  [draw opacity=0][fill={rgb, 255:red, 242; green, 7; blue, 37 }  ,fill opacity=1 ] (355.6,184.95) .. controls (355.6,182.19) and (357.8,179.95) .. (360.51,179.95) .. controls (363.22,179.95) and (365.42,182.19) .. (365.42,184.95) .. controls (365.42,187.71) and (363.22,189.95) .. (360.51,189.95) .. controls (357.8,189.95) and (355.6,187.71) .. (355.6,184.95) -- cycle ;
        %Shape: Ellipse [id:dp37631485662632347] 
        \draw  [draw opacity=0][fill={rgb, 255:red, 242; green, 7; blue, 37 }  ,fill opacity=1 ] (411.65,185) .. controls (411.65,182.23) and (413.84,180) .. (416.56,180) .. controls (419.27,180) and (421.46,182.23) .. (421.46,185) .. controls (421.46,187.76) and (419.27,190) .. (416.56,190) .. controls (413.84,190) and (411.65,187.76) .. (411.65,185) -- cycle ;
        %Shape: Ellipse [id:dp9228054619778103] 
        \draw  [draw opacity=0][fill={rgb, 255:red, 242; green, 7; blue, 37 }  ,fill opacity=1 ] (411.6,128.95) .. controls (411.6,126.19) and (413.8,123.95) .. (416.51,123.95) .. controls (419.22,123.95) and (421.42,126.19) .. (421.42,128.95) .. controls (421.42,131.71) and (419.22,133.95) .. (416.51,133.95) .. controls (413.8,133.95) and (411.6,131.71) .. (411.6,128.95) -- cycle ;
        %Shape: Ellipse [id:dp9761113886790156] 
        \draw  [draw opacity=0][fill={rgb, 255:red, 242; green, 7; blue, 37 }  ,fill opacity=1 ] (411.6,72.95) .. controls (411.6,70.19) and (413.8,67.95) .. (416.51,67.95) .. controls (419.22,67.95) and (421.42,70.19) .. (421.42,72.95) .. controls (421.42,75.71) and (419.22,77.95) .. (416.51,77.95) .. controls (413.8,77.95) and (411.6,75.71) .. (411.6,72.95) -- cycle ;
        %Shape: Ellipse [id:dp24235284312952232] 
        \draw  [draw opacity=0][fill={rgb, 255:red, 242; green, 7; blue, 37 }  ,fill opacity=1 ] (327.68,101.02) .. controls (327.68,98.26) and (329.88,96.02) .. (332.59,96.02) .. controls (335.3,96.02) and (337.5,98.26) .. (337.5,101.02) .. controls (337.5,103.78) and (335.3,106.02) .. (332.59,106.02) .. controls (329.88,106.02) and (327.68,103.78) .. (327.68,101.02) -- cycle ;
        %Shape: Ellipse [id:dp1826735098607173] 
        \draw  [draw opacity=0][fill={rgb, 255:red, 242; green, 7; blue, 37 }  ,fill opacity=1 ] (383.6,100.95) .. controls (383.6,98.19) and (385.8,95.95) .. (388.51,95.95) .. controls (391.22,95.95) and (393.42,98.19) .. (393.42,100.95) .. controls (393.42,103.71) and (391.22,105.95) .. (388.51,105.95) .. controls (385.8,105.95) and (383.6,103.71) .. (383.6,100.95) -- cycle ;
        %Shape: Ellipse [id:dp255823116982711] 
        \draw  [draw opacity=0][fill={rgb, 255:red, 242; green, 7; blue, 37 }  ,fill opacity=1 ] (327.6,156.95) .. controls (327.6,154.19) and (329.8,151.95) .. (332.51,151.95) .. controls (335.22,151.95) and (337.42,154.19) .. (337.42,156.95) .. controls (337.42,159.71) and (335.22,161.95) .. (332.51,161.95) .. controls (329.8,161.95) and (327.6,159.71) .. (327.6,156.95) -- cycle ;
        %Shape: Ellipse [id:dp9494151404914138] 
        \draw  [draw opacity=0][fill={rgb, 255:red, 242; green, 7; blue, 37 }  ,fill opacity=1 ] (383.6,156.95) .. controls (383.6,154.19) and (385.8,151.95) .. (388.51,151.95) .. controls (391.22,151.95) and (393.42,154.19) .. (393.42,156.95) .. controls (393.42,159.71) and (391.22,161.95) .. (388.51,161.95) .. controls (385.8,161.95) and (383.6,159.71) .. (383.6,156.95) -- cycle ;
        %Shape: Diamond [id:dp9243702003505755] 
        \draw  [draw opacity=0][fill={rgb, 255:red, 10; green, 69; blue, 199 }  ,fill opacity=1 ] (388.51,124.2) -- (393.17,128.95) -- (388.51,133.7) -- (383.84,128.95) -- cycle ;
        
        % Text Node
        \draw (326.16,191) node [anchor=north west][inner sep=0.75pt]  [font=\footnotesize]  {$x_{i}$};
        % Text Node
        \draw (377.4,191) node [anchor=north west][inner sep=0.75pt]  [font=\footnotesize]  {$x_{i+1}$};
        % Text Node
        \draw (285,152) node [anchor=north west][inner sep=0.75pt]  [font=\footnotesize]  {$y_{j}$};
        % Text Node
        \draw (272,96) node [anchor=north west][inner sep=0.75pt]  [font=\footnotesize]  {$y_{j+1}$};
        % Text Node
        \draw (290,199) node [anchor=north west][inner sep=0.75pt]  [font=\footnotesize]  {$x_{\hindexm{i}{1}}$};
        % Text Node
        \draw (348,199) node [anchor=north west][inner sep=0.75pt]  [font=\footnotesize]  {$x_{\hindexp{i}{1}}$};
        % Text Node
        \draw (403,199) node [anchor=north west][inner sep=0.75pt]  [font=\footnotesize]  {$x_{\hindexp{i}{3}}$};
        % Text Node
        \draw (250,180) node [anchor=north west][inner sep=0.75pt]  [font=\footnotesize]  {$y_{\hindexm{j}{1}}$};
        % Text Node
        \draw (250,124) node [anchor=north west][inner sep=0.75pt]  [font=\footnotesize]  {$y_{\hindexp{j}{1}}$};
        % Text Node
        \draw (250,68) node [anchor=north west][inner sep=0.75pt]  [font=\footnotesize]  {$y_{\hindexp{j}{3}}$};

            \end{tikzpicture}
     \caption{Two-dimensional staggered grid as described in \cite{MR3459981}. $h$ and $T$ are evaluated at the red circles whereas $g$ is evaluated at the blue diamonds.}
    \label{fig:2D staggered grid}
\end{figure}

Let $\abs{\textbf{Q}_{v}}\in\R^{N_{q}\times N_{q}}$, $v \in \mathcal{J}_{\textbf{\textit{x}}}$, be the diagonal matrix with entries $\left(\abs{\textbf{Q}_{v}}\right)_{\ell\ell} = |\Omega_{v}^{\ell}|$, $\ell = 1,\ldots,N_{q}$. Then, we define the matrices $\textbf{Q}^{\pm}_{v} \coloneqq (\textbf{Q}_{v} \pm |\textbf{Q}_{v}|)/2$, $v \in  \mathcal{J}_{\textbf{\textit{x}}}$. For $(\alpha,\beta)=(\hindexp{i}{1},j)$ or $(\alpha,\beta)=(i,\hindexp{j}{1})$, the first-order upwind differential operators at the interfaces are defined as
\begin{align*}
     \mathcal{D}^{+}_{x}\textbf{\textit{g}}_{\alpha\beta} &\coloneqq \frac{1}{{\Delta x}}(\textbf{\textit{g}}_{\hindexp{\alpha}{1}\beta} - \textbf{\textit{g}}_{\alpha\beta}),& \mathcal{D}^{-}_{x}\textbf{\textit{g}}_{\alpha\beta} &\coloneqq \frac{1}{{\Delta x}}(\textbf{\textit{g}}_{\alpha\beta} - \textbf{\textit{g}}_{\hindexm{\alpha}{1}\beta}),\\  
     \mathcal{D}^{+}_{y}\textbf{\textit{g}}_{\alpha\beta} &\coloneqq \frac{1}{{\Delta y}}(\textbf{\textit{g}}_{\alpha\hindexp{\beta}{1}} - \textbf{\textit{g}}_{\alpha\beta}), & \mathcal{D}^{-}_{y}\textbf{\textit{g}}_{\alpha\beta} &\coloneqq \frac{1}{{\Delta y}}(\textbf{\textit{g}}_{\alpha\beta} - \textbf{\textit{g}}_{\alpha\hindexm{\beta}{1}}),\\
     \delta^{0}_{x}h_{\alpha\beta} &\coloneqq \frac{1}{{\Delta x}}(h_{\hindexp{\alpha}{1}\beta} - h_{\hindexm{\alpha}{1}\beta}), & \delta^{0}_{y}h_{\alpha\beta} &\coloneqq \frac{1}{\Delta y}(h_{\alpha\hindexp{\beta}{1}} - h_{\alpha\hindexm{\beta}{1}}).
\end{align*}
Additionally, for $(\alpha,\beta)=(i,j)$ or $(\alpha,\beta)=(\hindexp{i}{1},\hindexp{j}{1})$, we define the following centered difference operator at the cell centers and corners
 \begin{align*}
       \mathcal{D}^{0}_{x}\textbf{\textit{g}}_{\alpha\beta} &\coloneqq \frac{1}{\Delta x}(\textbf{\textit{g}}_{\hindexp{\alpha}{1}\beta} - \textbf{\textit{g}}_{\hindexm{\alpha}{1}\beta}), & \mathcal{D}^{0}_{y}\textbf{\textit{g}}_{\alpha\beta} &\coloneqq \frac{1}{\Delta y}(\textbf{\textit{g}}_{\alpha\hindexp{\beta}{1}} - \textbf{\textit{g}}_{\alpha\hindexm{\beta}{1}}).
 \end{align*}
Note that though $\mathcal{D}^{0}$ and $\delta^{0}$ are both central difference operators, $\mathcal{D}^{0}$ approximates the derivative of $g$ at the cell centers and corners while $\delta^{0}$ approximates the derivatives of $h$ and $T$ at the interfaces.  
Finally, using an IMEX scheme similar to \cite{doi:10.1137/24M1646303} for discretizing in time, the full-rank macro--micro scheme reads
\begin{subequations}\label{eq:FDNMM}
    \begin{equation}\label{eq:FDNMM3}
    \begin{aligned}
        \frac{\varepsilon^{2}}{c\Delta t}\left( \textbf{\textit{g}}^{n+1}_{i_{\sfrac{1}{2}}^+,j} - \textbf{\textit{g}}^{n}_{i_{\sfrac{1}{2}}^+,j}\right) + \varepsilon\left( \textbf{I} - \frac{1}{2\pi}\mathds{1}\textbf{\textit{w}}^{\top} \right)\mathcal{L}_{x}\textbf{\textit{g}}^{n}_{\hindexp{i}{1}j} + \varepsilon\left( \textbf{I} - \frac{1}{2\pi}\mathds{1}\textbf{\textit{w}}^{\top} \right)\mathcal{L}_{y}\textbf{\textit{g}}^{n}_{\hindexp{i}{1}j}&\\
         + \textbf{Q}_{x}\mathds{1}\delta^{0}_{x}\left(B^{n}_{\hindexp{i}{1}j} + \varepsilon^{2}h^{n}_{\hindexp{i}{1}j} \right) + \textbf{Q}_{y}\mathds{1}\delta^{0}_{y}\left( B^{n}_{\hindexp{i}{1}j}+ \varepsilon^{2}h^{n}_{\hindexp{i}{1}j} \right) = -&\sigma^{t}_{\hindexp{i}{1}j}\textbf{\textit{g}}^{n+1}_{\hindexp{i}{1}j},
    \end{aligned}
    \end{equation}
     \begin{equation}\label{eq:FDNMM1}
        \frac{\varepsilon^{2}}{c\Delta t}\left( h^{n+1}_{ij} - h^{n}_{ij} \right) + \frac{1}{c\Delta t}\left(B^{n+1}_{ij} - B^{n}_{ij} \right) + \frac{1}{2\pi}\textbf{\textit{w}}^{\top}\textbf{Q}_{x}\mathcal{D}^{0}_{x}\textbf{\textit{g}}^{n+1}_{ij} + \frac{1}{2\pi}\textbf{\textit{w}}^{\top}\textbf{Q}_{y}\mathcal{D}^{0}_{y}\textbf{\textit{g}}^{n+1}_{ij} = -\sigma^{a}_{ij}h^{n+1}_{ij},
    \end{equation}
    \begin{equation}\label{eq:FDNMM5}
        \frac{c_{\nu}}{\Delta t}\left( T^{n+1}_{ij}-T^{n}_{ij} \right) = 2\pi\sigma^{a}_{ij}h^{n+1}_{ij},
    \end{equation}

\end{subequations} 
where, for $v\in\mathcal{J}_{\textbf{\textit{x}}}$,
\begin{align*}
    \mathcal{L}_{v}\textbf{\textit{g}}^{n}_{\hindexp{i}{1}j} &= \left(\textbf{Q}_{v}^{-}\mathcal{D}^{+}_{v} + \textbf{Q}_{v}^{+}\mathcal{D}^{-}_{v}\right)\textbf{\textit{g}}^{n}_{\hindexp{i}{1}j}.
\end{align*}
The scheme is analogously defined at $(x_{i},y_{\hindexp{j}{1}} )$ for $g$ and at $( x_{\hindexp{i}{1}},y_{\hindexp{j}{1}} )$ for $h$ and $T$.

\paragraph{Local mass conservation} Let $\phi^{n}_{ij} = B^{n}_{ij} + \varepsilon^{2}h^{n}_{ij} $ (analogously, $\phi^{n}_{\hindexp{i}{1}\hindexp{j}{1}}$) denote the scalar flux at time $t_{n}$. Then, we define the mass of the system at time $t_{n}$ as 
\begin{equation}\label{eq:massFull}
    m^{n} \coloneqq \sum_{ij} \left(\frac{2\pi}{c}\phi^{n}_{ij} + c_{\nu}T^{n}_{ij}\right)\Deltaxy + \sum_{ij} \left(\frac{2\pi}{c}\phi^{n}_{\hindexp{i}{1}\hindexp{j}{1}} + c_{\nu}T^{n}_{\hindexp{i}{1}\hindexp{j}{1}}\right)\Deltaxy,
\end{equation}  
where $ \Deltaxy = \Delta x\Delta y $. Then, the following result holds for $m^{n}$:
\begin{theorem}[Local mass conservation]
    The full-rank macro--micro scheme \eqref{eq:FDNMM} is locally conservative for Dirichlet zero or periodic boundary conditions. That is, for $n \in\{0,1\}$, the scheme fulfills the discrete conservation law 
    \begin{subequations}
        \begin{displaymath}
             \frac{1}{c\Delta t}\left(\phi^{n+1}_{ij}-\phi^{n}_{ij} \right) + \frac{1}{2\pi}\textbf{\textit{w}}^{\top}\textbf{Q}_{x}\mathcal{D}^{0}_{x}\textbf{\textit{g}}^{n+1}_{ij} + \frac{1}{2\pi}\textbf{\textit{w}}^{\top}\textbf{Q}_{y}\mathcal{D}^{0}_{y}\textbf{\textit{g}}^{n+1}_{ij} = -\sigma^{a}_{ij}h^{n+1}_{ij},
        \end{displaymath}
        \begin{displaymath}
           \frac{ c_{\nu}}{\Delta t}\left( T^{n+1}_{ij}-T^{n}_{ij} \right) = 2\pi\sigma^{a}_{ij}h^{n+1}_{ij},
        \end{displaymath}
    \end{subequations}
    with an analogous definition of the conservation law at $(x_{\hindexp{i}{1}},y_{\hindexp{j}{1}})$.
\end{theorem}
\begin{proof}
    For zero or periodic boundary conditions, we have from the definition of $\mathcal{D}^{0}_{v}$, $v\in\mathcal{J}_{\textbf{\textit{x}}}$ that $\sum_{ij} \textbf{\textit{w}}^{\top}\textbf{Q}_{v}\mathcal{D}^{0}_{v}\textbf{\textit{g}}^{n+1}_{ij} = 0$, $ \sum_{ij}\textbf{\textit{w}}^{\top}\textbf{Q}_{v}\mathcal{D}^{0}_{v}\textbf{\textit{g}}^{n+1}_{\hindexp{i}{1}\hindexp{j}{1}} = 0$. Thus, the mass $m^{n}$, defined in \eqref{eq:massFull}, is conserved.  
\end{proof}
\begin{remark}
    For global mass conservation, we need the additional assumption that there is no absorption or emission of particles by the material, i.e. $\sigma^{a} = 0$.
\end{remark}

%%%%%%%%%%%%%%%%%%%%%%%%%%%%% Sub-section %%%%%%%%%%%%%%%%%%%%%%%%%%%%%%%%%%%%
\subsection{AP property}\label{section:NMMAPproperty}
%%%%%%%%%%%%%%%%%%%%%%%%%%%%%%%%%%%%%%%%%%%%%%%%%%%%%%%%%%%%%%%%%%%%%%%%%%%%%%

\begin{theorem}\label{theorem:APFull}
    The full-rank macro--micro scheme \eqref{eq:FDNMM} is asymptotic--preserving for the thermal radiative transfer equations. That is, it preserves the discrete Rosseland diffusion equation given by
    \begin{equation}\label{eq:ThmAPFulleq1}
        \frac{c_{\nu}}{\Delta t}\left( T^{n+1}_{ij} - T^{n}_{ij} \right) + \frac{2\pi}{c\Delta t}\left( B^{n+1}_{ij} - B^{n}_{ij} \right) = \frac{2\pi}{3}\left[ \mathcal{D}^{0}_{x}\left( \frac{1}{\sigma^{t}}\delta^{0}_{x}B^{n} \right)_{ij} + \mathcal{D}^{0}_{y}\left( \frac{1}{\sigma^{t}}\delta^{0}_{y}B^{n} \right)_{ij} \right].
    \end{equation}
    This is a 5-point centered difference discretization for the Rosseland diffusion equation \eqref{eq:RosselandApprox} on a staggered grid with an explicit time discretization. A similar relation holds at $(x_{\hindexp{i}{1}}, y_{\hindexp{j}{1}} )$.
\end{theorem}
\begin{proof}
    From \eqref{eq:FDNMM3} and the analogous definition at $(x_{i},y_{\hindexp{j}{1}})$, we see that as $\varepsilon\to 0$ we get
    \begin{displaymath}
        \textbf{\textit{g}}^{n+1}_{\alpha\beta} = -\frac{1}{\sigma^{t}_{\alpha\beta}}\left( \textbf{Q}_{x}\mathds{1}\delta^{0}_{x} + \textbf{Q}_{y}\mathds{1}\delta^{0}_{y} \right)B^{n}_{\alpha\beta},
    \end{displaymath}
    where $(\alpha,\beta) \in\{ (\hindexp{i}{1},j),(i,\hindexp{j}{1})\}$. Similarly, taking the limit $\varepsilon\to 0$ in \eqref{eq:FDNMM1} yields
    \begin{equation}\label{eq:APFulleq1}
        \sigma^{a}_{ij}h^{n+1}_{ij} = -\frac{1}{2\pi}\textbf{\textit{w}}^{\top}\textbf{Q}_{x}\mathcal{D}^{0}_{x}\textbf{\textit{g}}^{n+1}_{ij} -\frac{1}{2\pi}\textbf{\textit{w}}^{\top}\textbf{Q}_{y}\mathcal{D}^{0}_{y}\textbf{\textit{g}}^{n+1}_{ij} - \frac{1}{c\Delta t}\left(B^{n+1}_{ij} - B^{n}_{ij} \right).
    \end{equation}
    Note that $\mathcal{D}^{0}_{x}\textbf{\textit{g}}^{n+1}_{ij} = (\textbf{\textit{g}}_{\hindexp{i}{1}j} - \textbf{\textit{g}}_{\hindexm{i}{1}j})/\Delta x $ and $ \mathcal{D}^{0}_{y}\textbf{\textit{g}}^{n+1}_{ij} = (\textbf{\textit{g}}_{i\hindexp{j}{1}} - \textbf{\textit{g}}_{i\hindexm{j}{1}})/\Delta y $. Thus, substituting $\textbf{\textit{g}}^{n+1}_{\hindexp{i}{1}j}$ and $\textbf{\textit{g}}^{n+1}_{i\hindexp{j}{1}}$ in $\mathcal{D}^{0}_{x}\textbf{\textit{g}}^{n+1}_{ij}$ gives
    \begin{align*}
        \mathcal{D}^{0}_{x}\textbf{\textit{g}}^{n+1}_{ij} &= \frac{-1}{\Delta x}\left(\frac{1}{\sigma^{t}_{\hindexp{i}{1}j}}\left( \textbf{Q}_{x}\mathds{1}\delta^{0}_{x} + \textbf{Q}_{y}\mathds{1}\delta^{0}_{y} \right)B^{n}_{\hindexp{i}{1}j} - \frac{1}{\sigma^{t}_{\hindexm{i}{1}j}}\left( \textbf{Q}_{x}\mathds{1}\delta^{0}_{x} + \textbf{Q}_{y}\mathds{1}\delta^{0}_{y} \right)B^{n}_{\hindexm{i}{1}j} \right)\\
        &= \frac{-1}{\Delta x}\left(\frac{1}{\sigma^{t}_{\hindexp{i}{1}j}}\textbf{Q}_{x}\mathds{1}\delta^{0}_{x}B^{n}_{\hindexp{i}{1}j} - \frac{1}{\sigma^{t}_{\hindexm{i}{1}j}}\textbf{Q}_{x}\mathds{1}\delta^{0}_{x}B^{n}_{\hindexm{i}{1}j} \right) \\& \ \ \ \ -\frac{1}{\Delta x}\left(\frac{1}{\sigma^{t}_{\hindexp{i}{1}j}}\textbf{Q}_{y}\mathds{1}\delta^{0}_{y}B^{n}_{\hindexp{i}{1}j} - \frac{1}{\sigma^{t}_{\hindexm{i}{1}j}}\textbf{Q}_{y}\mathds{1}\delta^{0}_{y}B^{n}_{\hindexm{i}{1}j} \right)\\
        &= -\mathcal{D}^{0}_{x}\left( \frac{1}{\sigma^{t}}\textbf{Q}_{x}\mathds{1}\delta^{0}_{x}B^{n} \right)_{ij} -\mathcal{D}^{0}_{x}\left( \frac{1}{\sigma^{t}}\textbf{Q}_{y}\mathds{1}\delta^{0}_{y}B^{n} \right)_{ij}.
    \end{align*}
    A similar expression can be derived for $\mathcal{D}^{0}_{y}\textbf{\textit{g}}^{n+1}_{ij}$. Substituting $\mathcal{D}^{0}_{x}\textbf{\textit{g}}^{n+1}_{ij} $ and $\mathcal{D}^{0}_{y}\textbf{\textit{g}}^{n+1}_{ij} $ in \eqref{eq:APFulleq1} yields
    \begin{align*}
        \sigma^{a}_{ij}h^{n+1}_{ij} &= \frac{1}{2\pi}\textbf{\textit{w}}^{\top}\textbf{Q}_{x}\left[ \textbf{Q}_{x}\mathds{1}\mathcal{D}^{0}_{x}\left(\frac{1}{\sigma^{a}}\delta^{0}_{x}B^{n}  \right)_{ij} + \textbf{Q}_{y}\mathds{1}\mathcal{D}^{0}_{x}\left(\frac{1}{\sigma^{t}}\delta^{0}_{y}B^{n}  \right)_{ij} \right]\\
        & \ \ \ \  + \frac{1}{2\pi}\textbf{\textit{w}}^{\top}\textbf{Q}_{y}\left[ \textbf{Q}_{x}\mathds{1}\mathcal{D}^{0}_{y}\left(\frac{1}{\sigma^{t}}\delta^{0}_{x}B^{n}  \right)_{ij} + \textbf{Q}_{y}\mathds{1}\mathcal{D}^{0}_{y}\left(\frac{1}{\sigma^{t}}\delta^{0}_{y}B^{n}  \right)_{ij} \right]\\
        & \ \ \ \ -\frac{1}{c\Delta t}\left(B^{n+1}_{ij} - B^{n}_{ij} \right).
    \end{align*}
    Since by the choice of our quadrature $\textbf{\textit{w}}^{\top}\textbf{Q}_{x}\textbf{Q}_{x}\mathds{1} = \int_{\mathbb{P}(\mathbb{S}^{2})}(\Omega_{x})^{2} \hspace{1mm}\mathrm{d}\Omega = 2\pi/3$, $\textbf{\textit{w}}^{\top}\textbf{Q}_{y}\textbf{Q}_{y}\mathds{1} = \int_{\mathbb{P}(\mathbb{S}^{2})}(\Omega_{y})^{2} \hspace{1mm}\mathrm{d}\Omega = 2\pi/3$, and $\textbf{\textit{w}}^{\top}\textbf{Q}_{x}\textbf{Q}_{y}\mathds{1} = \int_{\mathbb{P}(\mathbb{S}^{2})}\Omega_{x}\Omega_{y} \hspace{1mm}\mathrm{d}\Omega = 0$, we get
    % \begin{align*}
    %     \textbf{\textit{w}}^{\top}\textbf{Q}_{x}\textbf{Q}_{x}\mathds{1} &\approx \int_{\mathbb{P}(\mathbb{S}^{2})}(\Omega_{x})^{2} \hspace{1mm}\mathrm{d}\Omega = 4\pi/3\\
    %     \textbf{\textit{w}}^{\top}\textbf{Q}_{y}\textbf{Q}_{y}\mathds{1} &\approx \int_{\mathbb{P}(\mathbb{S}^{2})}(\Omega_{y})^{2} \hspace{1mm}\mathrm{d}\Omega = 4\pi/3\\
    %     \textbf{\textit{w}}^{\top}\textbf{Q}_{x}\textbf{Q}_{y}\mathds{1} &\approx \int_{\mathbb{P}(\mathbb{S}^{2})}\Omega_{x}\Omega_{y} \hspace{1mm}\mathrm{d}\Omega = 0,
    % \end{align*}

    \begin{equation*}
         \sigma^{a}_{ij}h^{n+1}_{ij} = \frac{1}{3}\left[ \mathcal{D}^{0}_{x}\left( \frac{1}{\sigma^{t}}\delta^{0}_{x}B^{n} \right)_{ij} + \mathcal{D}^{0}_{y}\left( \frac{1}{\sigma^{t}}\delta^{0}_{y}B^{n} \right)_{ij} \right] - \frac{1}{c\Delta t}\left(B^{n+1}_{ij} - B^{n}_{ij} \right).
    \end{equation*}
    Finally, substituting the value of $h^{n+1}_{ij}$ in \eqref{eq:FDNMM5} yields \eqref{eq:ThmAPFulleq1} at $(x_{i},y_{j})$.
    % \begin{equation}\label{eq:Thm1eq1}
    %     c_{\nu}\left( \frac{T^{n+1}_{ij}-T^{n}_{ij}}{\Delta t} \right) = \frac{4\pi}{3}\left[ \mathcal{D}^{0}_{x}\left( \frac{1}{\sigma^{t}_{ij}}\delta^{0}_{x}B^{n}_{ij} \right) + \mathcal{D}^{0}_{y}\left( \frac{1}{\sigma^{t}_{ij}}\delta^{0}_{y}B^{n}_{ij} \right) \right]- \frac{4\pi}{c}\left(\frac{B^{n+1}_{ij} - B^{n}_{ij}}{\Delta t} \right).
    % \end{equation}
    A similar equality can be shown to hold at $(x_{\hindexp{i}{1}},y_{\hindexp{j}{1}})$ in the limit $\varepsilon\to 0$.
    % as $\varepsilon\to 0$ we get
    % \begin{equation}\label{eq:Thm1eq2}
    % \begin{aligned}
    %     c_{\nu}\left( \frac{T^{n+1}_{\hindexp{i}{1}\hindexp{j}{1}}-T^{n}_{\hindexp{i}{1}\hindexp{j}{1}}}{\Delta t} \right) &= \frac{4\pi}{3}\left[ \mathcal{D}^{0}_{x}\left( \frac{1}{\sigma^{t}_{\hindexp{i}{1}\hindexp{j}{1}}}\delta^{0}_{x}B^{n}_{\hindexp{i}{1}\hindexp{j}{1}} \right)  +\mathcal{D}^{0}_{y}\left( \frac{1}{\sigma^{t}_{\hindexp{i}{1}\hindexp{j}{1}}}\delta^{0}_{y}B^{n}_{\hindexp{i}{1}\hindexp{j}{1}} \right) \right]\\& \quad- \frac{4\pi}{c}\left(\frac{B^{n+1}_{\hindexp{i}{1}\hindexp{j}{1}} - B^{n}_{\hindexp{i}{1}\hindexp{j}{1}}}{\Delta t} \right).
    % \end{aligned}
    % \end{equation}
    % Expanding the differential operators on the right-hand side of \eqref{eq:Thm1eq1} yields
    % \begin{equation*}
    %     \mathcal{D}^{0}_{x}\left( \frac{1}{\sigma^{a}_{ij}}\delta^{0}_{x}B^{n}_{ij} \right) = \frac{1}{\Delta x^{2}}\left[ \frac{1}{\sigma^{t}_{\hindexp{i}{1}j}}\left( B^{n}_{i+1,j} - B^{n}_{ij} \right) - \frac{1}{\sigma^{t}_{\hindexm{i}{1}j}}\left( B^{n}_{i,j} - B^{n}_{i-1,j} \right) \right].
    % \end{equation*}
    % Thus, \eqref{eq:Thm1eq1} corresponds to a  
\end{proof}

%%%%%%%%%%%%%%%%%%%%%%%%%%%%% Sub-section %%%%%%%%%%%%%%%%%%%%%%%%%%%%%%%%%%%%
\subsection{Energy stability}\label{section:EnergyStabNMM}
%%%%%%%%%%%%%%%%%%%%%%%%%%%%%%%%%%%%%%%%%%%%%%%%%%%%%%%%%%%%%%%%%%%%%%%%%%%%%%
In this section, we show that the full-rank macro--micro scheme~\eqref{eq:FDNMM} dissipates energy under a mixed CFL condition, for a suitable definition of energy. Let $ N_{\textbf{\textit{x}}}^{I} $ denote the number of cell interface points and $ N_{\textbf{\textit{x}}}^{C} $ denote the combined number of cell centers and corners. Let $\mathcal{K}^{I}$ denote the set of all spatial indices $(\alpha,\beta)$ such that $(x_{\alpha},y_{\beta})$ lies on the interface of two cells, i.e. it is of the form $(x_{\hindexp{i}{1}},y_{j})$ or $(x_{i},y_{\hindexp{j}{1}})$. Then, to map $g_{\kappa,\ell}$, $\kappa\in\mathcal{K}^{I}$, $\ell = 1,\ldots,N_{q}$, to the matrix $\textbf{g}\in\R^{N_{\textbf{\textit{x}}}^{\mathrm{I}}\times N_{q}} $ we define the bijective index map $\varrho^{I}:\mathcal{K}^{I}\to \{1,\ldots,N_{\textbf{\textit{x}}}^{I}\}$ such that $ (\textbf{g})_{\varrho(\kappa),\ell}=g_{\kappa,\ell}$. Similarly, let $\mathcal{K}^{C}$ denote the set of all spatial indices $(\alpha,\beta)$ such that $(x_{\alpha},y_{\beta})$ is either the corner of a cell or its center and $\varrho^{C}:\mathcal{K}^{C}\to \{1,\ldots, N_{\textbf{\textit{x}}}^{C}\}$ denote the corresponding bijective index map. Then, we define $\textbf{\textit{h}}\in\R^{N_{\textbf{\textit{x}}}^{C}}$ and $\textbf{\textit{T}}\in\R^{N_{\textbf{\textit{x}}}^{C}}$ with elements $(\textbf{\textit{h}})_{\varrho^{C}(\kappa)} = h_{\kappa}$ and $(\textbf{\textit{T}})_{\varrho^{C}(\kappa)} = T_{\kappa}$, respectively. The discrete-$ L^{2} $ norms of $\textbf{\textit{h}}$, $\textbf{\textit{T}}$, and $\textbf{g}$ are given as
\begin{equation*}
    \begin{aligned}
        \norm{\textbf{\textit{h}}}^{2} \coloneqq \sum_{\kappa\in\mathcal{K}^{C}}h_{\kappa}^{2}\Deltaxy, \quad
        \norm{\textbf{\textit{T}}}^{2} \coloneqq \sum_{\kappa\in\mathcal{K}^{C}}T_{\kappa}^{2}\Deltaxy, \quad
        \norm{\textbf{g}}^{2}_{\Omega} \coloneqq \sum_{\kappa\in\mathcal{K}^{I}}\textbf{\textit{g}}_{\kappa}^{\top}\textbf{M}^{2}\textbf{\textit{g}}_{\kappa}\Deltaxy 
        &= \sum_{\kappa\in\mathcal{K}^{I}}\sum_{\ell}w_{\ell}g_{\kappa,\ell}^{2}\Deltaxy,
    \end{aligned}
\end{equation*}
where $ \textbf{\textit{g}}_{\kappa} = (g_{\kappa,1},\ldots,g_{\kappa,N_{q}})^{\top}$ and $(\textbf{M})_{ij} = \sqrt{w_{i}}\delta_{ij} $, for $\delta_{ij} = 1$, if $i=j$ and $0$ otherwise. The energy at time $t_{n}$ is then defined as
\begin{equation}\label{eq:energydiscrete}
    e^{n} \coloneqq \norm{\frac{1}{c}\textbf{\textit{B}}^{n} + \frac{\varepsilon^{2}}{c}\textbf{\textit{h}}^{n}}^{2} + \frac{1}{2\pi}\norm{\frac{\varepsilon}{c}\textbf{g}^{n}}^{2}_{\Omega} + \frac{2}{5}\norm{\frac{\sqrt{ac_{\nu}}}{2\pi}(\textbf{\textit{T}}^{n})^{5/2}}^{2}.
\end{equation}
Note that in the remainder of the paper, we drop the subscript $\Omega$ from $\norm{\cdot}_{\Omega}$ and denote it by $\norm{\cdot}$ when it is clear from the context that $\norm{\cdot}_{\Omega}$ is chosen.
\begin{theorem}\label{thm:StabNMM}
    For a given spatial grid size $\Delta x, \Delta y$, let the step size $\Delta t$ satisfy the following CFL condition for all $\Omega_{x}^{\ell},\Omega_{y}^{\ell}\neq 0$,
    \begin{equation}\label{eq:CFLcond}
        \Delta t \leq \frac{1}{3c}\mathrm{min}\left\{ \varepsilon\Delta x + \frac{\sigma^{t}_{0}\Delta x^{2}}{4\abs{\Omega_{x}^{\ell}}}, \varepsilon\Delta y + \frac{\sigma^{t}_{0}\Delta y^{2}}{4\abs{\Omega_{y}^{\ell}}} \right\}.
    \end{equation}
    Then, the full-rank macro--micro scheme given by \eqref{eq:FDNMM} is energy stable, i.e. $e^{n+1} \leq e^{n},$ where the energy is defined in \eqref{eq:energydiscrete}.
\end{theorem}
\begin{proof}
	To keep the main part of the paper short, the proof of the theorem has been shifted to \Cref{appendix:ProofNMMstab}.
\end{proof}

% We know from Appendix~\ref{appendix:AsymptoticAnalysis} and \cite{doi:10.1137/24M1646303,einkemmer2022asymptoticpreserving,MR4253315} that in the limit $\varepsilon\to 0$, $g$ has a low-rank representation. Hence, in the next section, we investigate using DLRA~\cite{doi:10.1137/050639703} to update $g$ in time. 

%%%%%%%%%%%%%%%%%%%%%%%%%%%%%%%%% Section %%%%%%%%%%%%%%%%%%%%%%%%%%%%%%%%%%%%
\section{An AP parallel low-rank scheme for macro--micro equations}\label{section:ParallelDLRAforNMM}
%%%%%%%%%%%%%%%%%%%%%%%%%%%%%%%%%%%%%%%%%%%%%%%%%%%%%%%%%%%%%%%%%%%%%%%%%%%%%%
From Theorem~\ref{theorem:APFull} we see that in the diffusive limit, $\varepsilon\to 0$, the full-rank macro--micro scheme is a consistent discretization of the Rosseland equations. Additionally, Theorem~\ref{thm:StabNMM} shows that in the diffusive limit, the full-rank macro--micro scheme doesn't need to resolve small time scales to capture the correct dynamics of the system. Thus, the scheme~\eqref{eq:FDNMM} addresses the first two challenges outlined at the beginning of Section~\ref{section:EnergyStableNMM}, namely the AP property and energy stability without restrictive CFL conditions. 

However, the full-rank macro--micro scheme requires storing and updating $g$ at each pair $(\textbf{\textit{x}}_{\kappa},\boldsymbol{\Omega}_{\ell})$, for $\kappa\in\mathcal{K}^{I}$ and $\ell=1,\ldots,N_{q}$, and hence has a high computational cost and memory footprint. Thus, further modifications must be made to address the high-dimensional phase space of $g$ and reduce the computational costs of the scheme. In~\cite{doi:10.1137/24M1646303} the authors propose to use DLRA~\cite{doi:10.1137/050639703} to reduce the computational costs of the thermal radiative transfer equations using the augmented BUG integrator~\cite{robustBUG}. However, the proposed macro--micro augmented BUG scheme requires a serial S-step at twice the current rank. This increases the overall computational cost and memory requirement for the scheme. Moreover, only far-field boundaries have been considered in~\cite{doi:10.1137/24M1646303}, and thus efficient methods of implementing boundary conditions for the macro--micro low-rank schemes remain an open topic. In this section, we outline the use of the parallel BUG integrator~\cite{ceruti2023parallel} to update all the factors in parallel and outline a procedure to efficiently incorporate boundary conditions.

% In the limit $\varepsilon\to 0$, $g$ has a low-rank structure~\cite{doi:10.1137/24M1646303,einkemmer2022asymptoticpreserving,MR4253315} which can be exploited to reduce the computational cost and memory requirements of updating $g$.

For $v\in\mathcal{J}_{\textbf{\textit{x}}}$, let $ \mathbf{D}^{\pm}_{v}\in\R^{N_{\textbf{\textit{x}}}^{I}\times N_{\textbf{\textit{x}}}^{I}} $, $ \mathbf{d}^{0}_{v}\in\R^{N_{\textbf{\textit{x}}}^{C}\times N_{\textbf{\textit{x}}}^{I}} $, and $ \mathbf{D}^{0}_{v}\in\R^{N_{\textbf{\textit{x}}}^{I}\times N_{\textbf{\textit{x}}}^{C}} $ denote the matrix form of the operators $\mathcal{D}^{\pm}_{v}$, $ \delta^{0}_{v} $, and  $\mathcal{D}^{0}_{v}$, respectively. To begin, we consider the time-continuous evolution of $\textbf{g}(t)\in\R^{N_{\textbf{\textit{x}}}^{\mathrm{I}}\times N_{q}} $ which reads
\begin{equation}\label{eq:MicTimeEvol}
     \frac{\varepsilon^{2}}{c}\dot{\textbf{g}} = -\varepsilon\mathbf{L}_{\textbf{\textit{x}}}\textbf{g}\left( \textbf{I} - \frac{1}{2\pi}\mathds{1}\textbf{\textit{w}}^{\top} \right)^{\top} - \mathbf{d}^{0}_{\textbf{\textit{x}}}(B(\textbf{\textit{T}}) + \varepsilon^{2}\textbf{\textit{h}})\mathds{1}^{\top}\textbf{Q} -\boldsymbol{\Sigma}^{t}\textbf{g},
\end{equation}
where we use the short-hand notation $\mathbf{L}_{\textbf{\textit{x}}}\textbf{g} = (\mathbf{D}_{x}^{+}\textbf{g}\textbf{Q}_{x}^{-} + \mathbf{D}_{x}^{-}\textbf{g}\textbf{Q}_{x}^{+}) + (\mathbf{D}_{y}^{+}\textbf{g}\textbf{Q}_{y}^{-} + \mathbf{D}_{y}^{-}\textbf{g}\textbf{Q}_{y}^{+})$, and $\mathbf{d}^{0}_{\textbf{\textit{x}}}(B(\textbf{\textit{T}}) + \varepsilon^{2}\textbf{\textit{h}})\mathds{1}^{\top}\textbf{Q}_{\textbf{\textit{x}}} = \mathbf{d}^{0}_{x}(B(\textbf{\textit{T}}) + \varepsilon^{2}\textbf{\textit{h}})\mathds{1}^{\top}\textbf{Q}_{x} + \mathbf{d}^{0}_{y}(B(\textbf{\textit{T}}) + \varepsilon^{2}\textbf{\textit{h}})\mathds{1}^{\top}\textbf{Q}_{y}$. If $\boldsymbol{\Sigma}^{a}, \boldsymbol{\Sigma}^{s}\in\R^{N_{\textbf{\textit{x}}}^{I}\times N_{\textbf{\textit{x}}}^{I}}$ such that $(\boldsymbol{\Sigma}^{a})_{(\rho(\kappa),\rho(\kappa))} = \sigma^{a}_{\kappa}$ and $(\boldsymbol{\Sigma}^{s})_{(\rho(\kappa),\rho(\kappa))} = \sigma^{s}_{\kappa}$, for $\kappa\in\mathcal{K}^{I}$, then the total cross section $\boldsymbol{\Sigma}^{t}$ is given by $\boldsymbol{\Sigma}^{t} = \boldsymbol{\Sigma}^{a} + \boldsymbol{\Sigma}^{s}$. 

We make the following low-rank ansatz for $\textbf{g}$ 
\begin{displaymath}
    \textbf{g}(t) \approx \textbf{X}(t)\textbf{S}(t)\textbf{V}(t)^{\top},
\end{displaymath}
where $\textbf{X}(t)\in\R^{N_{\textbf{\textit{x}}}^{\mathrm{I}}\times r} $, $\textbf{V}(t)\in\R^{N_{q}\times r} $ are orthonormal and $\textbf{S}(t)\in\R^{r\times r} $ is invertible. The parallel BUG update scheme for \eqref{eq:MicTimeEvol} can be derived by following the steps outlined in Section~\ref{section:DLRA} with $\textbf{F}(t,\textbf{Y}) = -\varepsilon\mathbf{L}_{\textbf{\textit{x}}}\textbf{X}\textbf{S}\textbf{V}^{\top}\left( \textbf{I} - \frac{1}{2\pi}\mathds{1}\textbf{\textit{w}}^{\top} \right)^{\top} - \mathbf{d}^{0}_{\textbf{\textit{x}}}(B(\textbf{\textit{T}}) + \varepsilon^{2}\textbf{\textit{h}})\mathds{1}^{\top}\textbf{Q} -\boldsymbol{\Sigma}^{t}\textbf{X}\textbf{S}\textbf{V}^{\top}$. From Theorem~\ref{theorem:APFull} we know that in the limit $\varepsilon\to 0$, 
\begin{displaymath}
    \textbf{g}^{n+1} = -(\boldsymbol{\Sigma}^{t})^{-1}\mathbf{d}^{0}_{\textbf{\textit{x}}}B(\textbf{\textit{T}}^{n})\mathds{1}^{\top}\textbf{Q}_{\textbf{\textit{x}}}.
\end{displaymath}
Thus, for the low-rank parallel BUG scheme to be AP, $(\boldsymbol{\Sigma}^{t})^{-1}\mathbf{d}^{0}_{\textbf{\textit{x}}}B(\textbf{\textit{T}})$ must be in the range space of $\textbf{X}^{n+1}$ at time $t_{n+1}$ while $\mathds{1}^{\top}\textbf{Q}_{\textbf{\textit{x}}}$ must be in the range space of $\textbf{V}^{n+1}$. It can be shown that $(\boldsymbol{\Sigma}^{t})^{-1}\mathbf{d}^{0}_{\textbf{\textit{x}}}B(\textbf{\textit{T}})\in\mathrm{range}(\widehat{\textbf{X}})$ and $\mathds{1}^{\top}\textbf{Q}_{\textbf{\textit{x}}}\in\mathrm{range}(\widehat{\textbf{V}})$ (see Theorem~\ref{theorem:APparBUG}). However, due to the truncation step of the parallel BUG integrator, the range space of the augmented basis $\widehat{\textbf{X}}$ and $\widehat{\textbf{V}}$ need not be preserved in $\textbf{X}^{n+1}$, $\textbf{V}^{n+1}$. In general, this also holds for the class of rank-adaptive BUG integrators~\cite{robustBUG,ceruti2023parallel} which augment and truncate the basis up to a given tolerance. Based on the conservative truncation proposed in \cite{EINKEMMER2023112060}, a mitigation tactic was proposed in \cite{doi:10.1137/24M1646303}. The basis vectors $(\boldsymbol{\Sigma}^{t})^{-1}\mathbf{d}^{0}_{\textbf{\textit{x}}}B(\textbf{\textit{T}})$ and $\mathds{1}^{\top}\textbf{Q}_{\textbf{\textit{x}}}$ are augmented to the updated factors $\widehat{\textbf{X}}$ and $\widehat{\textbf{V}}$, then a conservative truncation strategy is used to preserve the range space of the augmented basis in the next time step. This guarantees that the resulting scheme is AP~\cite[Theorem~4.5]{doi:10.1137/24M1646303}. However, such a technique cannot be used for the parallel BUG integrator since the factors are updated in parallel and the dynamics are entirely determined by the solution at time $t_{n}$. Thus, we propose the following modifications to the parallel BUG integrator to derive a low-rank AP scheme for the thermal radiative transfer equations. 

For the initial rank-$r_{n}$ data $\textbf{\textit{T}}^{n}$, $\textbf{\textit{h}}^{n}$, $\textbf{X}^{n}$, $\textbf{S}^{n}$ and $\textbf{V}^{n}$ at time $t_{n}$ the low-rank factors are updated in four steps:
\begin{enumerate}
    \item \textbf{Pre-augmentation}: The spatially discretised gradient $\frac{1}{\sigma^{t}}\nabla_{\textbf{\textit{x}}}B(T)$ at time $t_{n}$ is augmented to the spatial basis $\textbf{X}^{n}$ and $\boldsymbol{\Omega}$ to the angular basis $\textbf{V}^{n}$. That is set 
\begin{displaymath}
    \begin{aligned}
      \textbf{X}^{n}_{\mathrm{aug}} &= \left[(\boldsymbol{\Sigma}^{t})^{-1}\mathbf{d}^{0}_{x}B(\textbf{\textit{T}}^{n}), (\boldsymbol{\Sigma}^{t})^{-1}\mathbf{d}^{0}_{y}B(\textbf{\textit{T}}^{n}), \textbf{X}^{n} \right],& \textbf{V}^{n}_{\mathrm{aug}} &= \left[\textbf{Q}_{x}\mathds{1}, \textbf{Q}_{y}\mathds{1}, \textbf{V}^{n} \right],
    \end{aligned}    
\end{displaymath}
and orthonormalize such that $ \left(\textbf{V}^{n}_{\mathrm{aug}}\right)^{\top}\textbf{\textit{w}} = 0 $ component-wise. Then, we project the coefficient matrix onto the augmented basis, i.e. we set $\textbf{S}^{n}_{\mathrm{aug}} = \textbf{X}^{n,\top}_{\mathrm{aug}}\textbf{X}^{n}\textbf{S}^{n}\textbf{V}^{n,\top}\textbf{V}^{n}_{\mathrm{aug}} $. The new initial rank is denoted by $\widetilde{r}_{n} = r_{n}+2$. \textit{Note that the subscript "$\mathrm{aug}$" is dropped in the rest of the paper and with abuse of notation we denote low-rank factors by} $\textbf{X}^{n}$, $\textbf{V}^{n}$, \textit{and} $\textbf{S}^{n}$.

\item \textbf{Parallel update}:\\
\textbf{K-step}: For $\textbf{K}^{n} = \textbf{X}^{n}\textbf{S}^{n} \in\R^{N_{\textbf{\textit{x}}}^{I}\times \widetilde{r}_{n}}$ update from $t_{n} $ to $t_{n+1} = t_{n} + \Delta t $
\begin{equation}\label{eq:NMMparK}
    \begin{split}
        \frac{\varepsilon^{2}}{c\Delta t}\left( \textbf{K}^{n+1} - \textbf{K}^{n} \right) = -\varepsilon\mathbf{L}_{\textbf{\textit{x}}}\textbf{K}^{n}\textbf{V}^{n,\top}\left( \textbf{I} - \frac{1}{2\pi}\mathds{1}\textbf{\textit{w}}^{\top} \right)^{\top}\textbf{V}^{n}\\ - \mathbf{d}^{0}_{\textbf{\textit{x}}}(B(\textbf{\textit{T}}^{n}) + \varepsilon^{2}\textbf{\textit{h}}^{n})\mathds{1}^{\top}\textbf{Q}_{\textbf{\textit{x}}}\textbf{V}^{n} -\boldsymbol{\Sigma}^{t}\textbf{K}^{n+1}.
    \end{split}
\end{equation}
Determine $\widehat{\textbf{X}} = (\textbf{X}^{n}, \widetilde{\textbf{X}}^{n+1})\in \R^{N_{\textbf{\textit{x}}}\times 2\widetilde{r}_{n}} $ and $\widetilde{\textbf{S}}^{K} = \widetilde{\textbf{X}}^{n+1,\top}\textbf{K}^{n+1} $ as described in Section~\ref{section:DLRA}.\\
\textbf{L-step}: For $\textbf{L}^{n} = \textbf{V}^{n}\textbf{S}^{n,\top} \in\R^{N_{q} \times \widetilde{r}_{n}} $ update from $t_{n} $ to $t_{n+1} = t_{n} + \Delta t $
\begin{equation}\label{eq:NMMparL}
    \begin{split}
        \frac{\varepsilon^{2}}{c\Delta t}\left( \textbf{L}^{n+1} - \textbf{L}^{n} \right) = -\varepsilon\left( \textbf{I} - \frac{1}{2\pi}\mathds{1}\textbf{\textit{w}}^{\top} \right)\textbf{L}^{n}\textbf{X}^{n,\top}\mathbf{L}_{\textbf{\textit{x}}}^{\top}\textbf{X}^{n} \\ - \textbf{Q}_{\textbf{\textit{x}}}\mathds{1}(B(\textbf{\textit{T}}^{n}) + \varepsilon^{2}\textbf{\textit{h}}^{n})^{\top}(\mathbf{d}^{0}_{\textbf{\textit{x}}})^{\top}\textbf{X}^{n} -\textbf{L}^{n+1}\textbf{X}^{n,\top}\boldsymbol{\Sigma}^{t}\textbf{X}^{n}.
    \end{split}
\end{equation}
Determine $\widehat{\textbf{V}} = (\textbf{V}^{n}, \widetilde{\textbf{V}}^{n+1})\in \R^{N_{\textbf{\textit{x}}}\times 2\widetilde{r}_{n}} $ such that $\widetilde{\textbf{V}}^{n+1,\top}\textbf{\textit{w}} = 0$ and store the matrix $\widetilde{\textbf{S}}^{L} = \textbf{L}^{n+1,\top}\widetilde{\textbf{V}}^{n+1} $ as described in Section~\ref{section:DLRA}\\
\textbf{S-step}: We update from $t_{n} $ to $t_{n+1} = t_{n} + \Delta t $
\begin{equation}\label{eq:NMMparS}
    \begin{split}
        \frac{\varepsilon^{2}}{c\Delta t}\left( \overline{\textbf{S}}^{n+1} - \textbf{S}^{n} \right) = -\varepsilon\textbf{X}^{n,\top}\mathbf{L}_{\textbf{\textit{x}}}\textbf{X}^{n}\textbf{S}^{n}\textbf{V}^{n,\top}\left( \textbf{I} - \frac{1}{2\pi}\mathds{1}\textbf{\textit{w}}^{\top} \right)^{\top}\textbf{V}^{n}\\ - \textbf{X}^{n,\top}\mathbf{d}^{0}_{\textbf{\textit{x}}}(B(\textbf{\textit{T}}^{n}) + \varepsilon^{2}\textbf{\textit{h}}^{n})\mathds{1}^{\top}\textbf{Q}_{\textbf{\textit{x}}}\textbf{V}^{n} -\textbf{X}^{n,\top}\boldsymbol{\Sigma}^{t}\textbf{X}^{n}\overline{\textbf{S}}^{n+1}.
    \end{split}
\end{equation}

\item \textbf{Augmentation}: Perform the augmentation of the coefficient matrix, i.e. set $\widehat{\textbf{S}}$ to be
    \begin{displaymath}
            \widehat{\textbf{S}} = \begin{bmatrix}
                \overline{\textbf{S}}(t_{n+1}) & \widetilde{\textbf{S}}^{L}\\
                \widetilde{\textbf{S}}^{K} & \mathbf{0}
            \end{bmatrix}.
    \end{displaymath}

\item \textbf{Conservative truncation}: A conservative truncation strategy, similar to the one used in \cite{doi:10.1137/24M1646303,EINKEMMER2023112060} is used to truncate the augmented basis while preserving the pre-augmented basis vectors at the next time step.
\end{enumerate}
% Note that we do not describe the entire scheme but only the changes to the update scheme for $h$, and $T$ described in~\Cref{section:EnergyStableNMM}.
After the truncation step, we set the solution at $t_{n+1}$ as $\mathbf{g}^{n+1}= \textbf{X}^{n+1}\textbf{S}^{n+1}\textbf{V}^{n+1,\top}$ with rank $1\leq r_{n+1} < 2\widetilde{r}_{n}$. The update for $h$ and $T$ remain the same as \eqref{eq:FDNMM} with the modification that $g_{\kappa,\ell}^{n+1} = (\textbf{X}^{n+1}\textbf{S}^{n+1}\textbf{V}^{n+1,\top})_{\varrho(\kappa),\ell}$, for $\kappa\in\mathcal{K}^{I}$, $\ell=1,\ldots,N_{q}$ .

\paragraph{Boundary conditions}
% Due to the nonlinear ansatz which separates the basis in each phase space variable, describing boundary conditions for low-rank schemes is not straightforward. This especially holds for kinetic problems, like thermal radiative transfer, in which the particle density at the boundary is only partially described through Dirichlet-type inflow conditions. This would require the construction of the full solution $\textbf{X}^{n+1}\textbf{S}^{n+1}\textbf{V}^{n+1,\top}$, setting the boundary conditions and then computing an SVD, which is highly undesirable in practice. 

An open question in dynamical low-rank approximation is the efficient implementation of boundary conditions. To impose boundary conditions, let us collect all spatial indices that lie on the boundary in the set $\mathcal{K}_{\mathrm{B}}^{I}:= \{(\alpha, \beta)\in\mathcal{K}^{I} \,\lvert\, (x_{\alpha},y_{\beta})\in\partial\mathcal{D}\}$ with $N_{\mathrm{B}}:=|\mathcal{K}_{\mathrm{B}}^{I}|$ boundary cells. Defining the bijective index map $\varrho^{I}_{B}:\mathcal{K}_{\mathrm{B}}^{I}\to\{1,\ldots, N_{\mathrm{B}}\} $, we can define the solution on the boundary as $\mathbf{\widetilde{g}}\in\mathbb{R}^{N_{\mathrm{B}} \times N_q}$ with elements $\widetilde{g}_{k\ell} = g_{\varrho^{I}_{\mathrm{B}}(\kappa),\ell}$ where $k\in\{1,\cdots, N_{\mathrm{B}}\}$, $\kappa\in \mathcal{K}_{\mathrm{B}}^{I}$. Vise-versa, we have $g_{\kappa\ell} = \widetilde{g}_{(\varrho^{I}_{\mathrm{B}})^{-1}(k),\ell}$. Now, to impose reflective-transmitive boundary conditions, we define
\begin{align}\label{eq:hatK}
    \widehat{g}_{k\ell} =
    \begin{cases}
        \rho  \widetilde{g}^{n+1}_{k\ell'} - (1-\rho)\varepsilon h^{n}_{(\varrho^{I}_{\mathrm{B}})^{-1}(k)},&  \mathrm{if }~\textbf{\textit{n}}_{k}\cdot\boldsymbol{\Omega}_{\ell}<0\\
         \widetilde{g}^{n+1}_{k\ell}, & \mathrm{if}~\textbf{\textit{n}}_{k}\cdot\boldsymbol{\Omega}_{\ell}\geq0
    \end{cases}\,.
\end{align}
where $\textbf{\textit{n}}_{k}$ is the outward-pointing normal at position $\textbf{\textit{x}}_{(\varrho^{I}_{\mathrm{B}})^{-1}(k)}$ and $\boldsymbol{\Omega}_{\ell'}$ is the reflection of $\boldsymbol{\Omega}_{\ell}$ along $\textbf{\textit{n}}_{k}$. The value of $h^{n}_{(\varrho^{I}_{\mathrm{B}})^{-1}(k)}$ is interpolated at the boundary point. We note that $\mathbf{\widetilde g}$ in \eqref{eq:hatK} can be computed efficiently by restricting $\textbf{X}^{n+1}$ to boundary points from ghost cells. Lastly, it remains to \emph{efficiently} impose $\mathbf{\widehat{g}}\in\mathbb{R}^{N_{\mathrm{B}} \times N_q}$ on the low-rank factorized solution, which we do after the time update of the parallel integrator. That is, we manipulate $\mathbf{g}^{n+1} = \textbf{X}^{n+1}\textbf{S}^{n+1}\textbf{V}^{n+1,\top}$ such that $g_{\kappa\ell}^{n+1} \equiv \widehat g_{\varrho^{I}_{\mathrm{B}}(\kappa),\ell}$ for all $\kappa \in\mathcal{K}_{\mathrm{B}}^{I}$ without having to compute and store $\mathbf{g}^{n+1}$. To apply boundary conditions efficiently, we define $\mathbf{\widehat K} = \mathbf{\widehat{g}}\textbf{V}^{n+1}\in\mathbb{R}^{N_{\mathrm{BC}} \times r}$ and $\mathbf{\bar K}\in\mathbb{R}^{N_{x} \times r}$ such that $\bar K_{\kappa\ell} = \widehat K_{\varrho^{I}_{\mathrm{B}}(\kappa),\ell}$ for $\kappa \in\mathcal{K}_{\mathrm{B}}^{I}$ and $\bar K_{j\ell} = \sum_i X_{ji}S_{i\ell}$ for $j\notin \mathcal{K}_{\mathrm{B}}^{I}$. Lastly, the basis $\textbf{X}^{n+1}$ and coefficient $\textbf{S}^{n+1}$ are recomputed as a QR decomposition of $\mathbf{\bar K}$, giving the factorized solution at time $n+1$ with reflective-transmitive boundary values imposed efficiently. 

Note that in an abuse of notation, we did not define an intermediate low-rank solution for the step in between the parallel integrator and imposing boundary conditions. Instead, to simplify notation, we recycle the notation of $\mathbf{g}^{n+1} = \textbf{X}^{n+1}\textbf{S}^{n+1}\textbf{V}^{n+1,\top}$. We also note that the above strategy does not impose boundary conditions exactly, similar to the projected boundary conditions, for example, used in \cite[Section~4]{MR4466549} or \cite{sapsis2009dynamically}. Alternatively, boundary conditions can be imposed exactly through an augmentation of the directional basis as is done for the projector--splitting~\cite{doi:10.1007/s10543-013-0454-0} and fixed-rank BUG\cite{doi:10.1007/s10543-021-00873-0} integrator in~\cite{MR4453266}. However, we found in our numerical experiment that an approximate imposition through a projection is sufficient to obtain accurate results.

% \textcolor{red}{@Chinmay: Can you check and adapt notation (I think you sometimes use textbf instead of mathbf and so on. I did not pay attention to this.) We might also want to change the names of variables, for example, the map $T$.}

%%%%%%%%%%%%%%%%%%%%%%%%%%%%% Sub-section %%%%%%%%%%%%%%%%%%%%%%%%%%%%%%%%%%%%
% \paragraph{Local mass conservation}
%%%%%%%%%%%%%%%%%%%%%%%%%%%%%%%%%%%%%%%%%%%%%%%%%%%%%%%%%%%%%%%%%%%%%%%%%%%%%%
With all the above modifications to the parallel BUG integrator we have the following results:

% \begin{lemma}[Consistency]
%      Assume that $ \textbf{g} $ satisfies the condition $\normalfont\textbf{X}^{n}\textbf{S}^{n}\textbf{V}^{n,\top}\textbf{\textit{w}} = 0$ at time $t_{n}$. After one step of the proposed AP nodal macro--micro parallel BUG scheme, the solution at time $t_{n+1}$ satisfies component-wise $\normalfont\textbf{X}^{n+1}\textbf{S}^{n+1}\textbf{V}^{n+1,\top}\textbf{\textit{w}} = 0$.
% \end{lemma}
% \begin{proof}
%     $\textbf{X}^{n}\textbf{S}^{n}\textbf{V}^{n,\top}\textbf{\textit{w}} = 0$ implies $\textbf{V}^{n,\top}\textbf{\textit{w}} = 0$ and the scheme is constructed such that $\widetilde{\textbf{V}}^{n,\top}\textbf{\textit{w}} = 0$. Thus we have $\widehat{\textbf{X}}\widehat{\textbf{S}}\widehat{\textbf{V}}^{\top}\textbf{\textit{w}} = 0$. After truncation of $\textbf{V}^{n+1}$ is also orthonormal to the weight vector $\textbf{\textit{w}}$, we get the result.
% \end{proof}

\begin{theorem}[Local mass conservation]
    The proposed macro--micro parallel BUG scheme is locally conservative for zero or periodic boundary conditions. Let the scalar flux at time $t_{n}$ be denoted by $\boldsymbol{\Phi}^{n} = \textbf{\textit{B}}^{n} + \varepsilon^{2}\textbf{\textit{h}}^{n} $, for $n \in\{0,1\}$. Then, the scheme fulfills the discrete conservation law 
    \begin{subequations}
        \begin{displaymath}
             \frac{1}{c\Delta t}\left(\boldsymbol{\Phi}^{n+1}-\boldsymbol{\Phi}^{n} \right) + \frac{1}{2\pi}\sum_{v\in\mathcal{J}_{\textbf{\textit{x}}}}\mathbf{D}^{0}_{v} \normalfont{\textbf{X}^{n+1}\mathrm{\textbf{S}}^{n+1}\mathrm{\textbf{V}}^{n+1,\top}}\textbf{Q}_{v}\textbf{\textit{w}}  = -\boldsymbol{\Sigma}^{a}\textbf{\textit{h}}^{n+1},
        \end{displaymath}
        \begin{displaymath}
            \frac{c_{\nu}}{\Delta t}\left( \textbf{\textit{T}}^{n+1}-\textbf{\textit{T}}^{n} \right) = 2\pi\boldsymbol{\Sigma}^{a}\textbf{\textit{h}}^{n+1},
        \end{displaymath}
    \end{subequations}
\end{theorem}

\begin{proof}
    Let $\textbf{g}^{n+1} = \textbf{X}^{n+1}\mathrm{\textbf{S}}^{n+1}\mathrm{\textbf{V}}^{n+1,\top} $ and, for $\kappa\in\mathcal{K}^{I}$, let $\textbf{\textit{g}}_{\kappa}^{n+1}$ denote the $\kappa$th row of $\textbf{g}^{n+1}$. Then, by the definition of $\mathcal{D}^{0}_{v}$, $v\in\mathcal{J}_{\textbf{\textit{x}}}$, for zero or periodic boundaries we have $\sum_{\kappa}\mathcal{D}^{0}_{v}\textbf{\textit{g}}_{\kappa}^{n+1,\top}\textbf{Q}_{v}\textbf{\textit{w}} = 0$. Thus $m^{n} $ defined in~\eqref{eq:massFull} is conserved.
\end{proof}
%%%%%%%%%%%%%%%%%%%%%%%%%%%%% Sub-section %%%%%%%%%%%%%%%%%%%%%%%%%%%%%%%%%%%%
% \subsubsection{Implementing boundary conditions}
%%%%%%%%%%%%%%%%%%%%%%%%%%%%%%%%%%%%%%%%%%%%%%%%%%%%%%%%%%%%%%%%%%%%%%%%%%%%%%sss

%%%%%%%%%%%%%%%%%%%%%%%%%%%%% Sub-section %%%%%%%%%%%%%%%%%%%%%%%%%%%%%%%%%%%%
\subsection{AP property}
%%%%%%%%%%%%%%%%%%%%%%%%%%%%%%%%%%%%%%%%%%%%%%%%%%%%%%%%%%%%%%%%%%%%%%%%%%%%%%

\begin{theorem}\label{theorem:APparBUG}
    The macro--micro parallel BUG scheme for the thermal radiative transfer equation is asymptotic--preserving. That is, in the limit $\varepsilon\to 0$ the scheme preserves the discrete Rosseland diffusion equation
    \begin{displaymath}
        \frac{c_{\nu}}{\Delta t}\left( \textbf{\textit{T}}^{n+1}-\textbf{\textit{T}}^{n} \right) + \frac{2\pi}{c\Delta t}\left( \textbf{\textit{B}}^{n+1} - \textbf{\textit{B}}^{n} \right) = \frac{2\pi}{3}\left[ \mathbf{D}^{0}_{x}\left( (\boldsymbol{\Sigma}^{t})^{-1}\mathbf{d}^{0}_{x}\textbf{\textit{B}}^{n} \right) + \mathbf{D}^{0}_{y}\left( (\boldsymbol{\Sigma}^{t})^{-1}\mathbf{d}^{0}_{y}\textbf{\textit{B}}^{n} \right) \right].
    \end{displaymath}
    This is a 5-point centered difference discretization for updating the diffusion equation \eqref{eq:RosselandApprox} on a staggered grid with an explicit time discretization. 
\end{theorem}
\begin{proof}
    As $\varepsilon\to 0$, since $\widetilde{\textbf{S}}^{K} = \widetilde{\textbf{X}}^{n+1,\top}\textbf{K}^{n+1} $, we get from the K-step~\eqref{eq:NMMparK}
    \begin{displaymath}
         \widetilde{\textbf{S}}^{K} = -\widetilde{\textbf{X}}^{n+1,\top}(\boldsymbol{\Sigma}^{t})^{-1}\left( \mathbf{d}^{0}_{x}B(\textbf{\textit{T}})\mathds{1}^{\top}\textbf{Q}_{x} + \mathbf{d}^{0}_{y}B(\textbf{\textit{T}})\mathds{1}^{\top}\textbf{Q}_{y} \right)\textbf{V}^{n}.
    \end{displaymath}
   Similarly, as $\varepsilon \to 0$ in the L-step~\eqref{eq:NMMparL} we get
   \begin{displaymath}
       \left(\textbf{X}^{n,\top}\boldsymbol{\Sigma}^{t}\textbf{X}^{n} \right)\widetilde{\textbf{S}}^{L} = -\textbf{X}^{n,\top}\left( \mathbf{d}^{0}_{x}B(\textbf{\textit{T}})\mathds{1}^{\top}\textbf{Q}_{x} + \mathbf{d}^{0}_{y}B(\textbf{\textit{T}})\mathds{1}^{\top}\textbf{Q}_{y} \right)\widetilde{\textbf{V}}^{n+1}.
   \end{displaymath}
   By construction, $(\boldsymbol{\Sigma}^{t})^{-1}\mathbf{d}^{0}_{x}B(\textbf{\textit{T}}^{n}),(\boldsymbol{\Sigma}^{t})^{-1}\mathbf{d}^{0}_{y}B(\textbf{\textit{T}}^{n}) \in \mathrm{range}(\textbf{X}^{n})$ which implies for $v\in\mathcal{J}_{\textbf{\textit{x}}}$
   \begin{displaymath}
       \textbf{X}^{n,\top}\mathbf{d}^{0}_{v}B(\textbf{\textit{T}}^{n}) = \textbf{X}^{n,\top}\boldsymbol{\Sigma}^{t}\textbf{X}^{n}\textbf{X}^{n,\top}(\boldsymbol{\Sigma}^{t})^{-1}\mathbf{d}^{0}_{v}B(\textbf{\textit{T}}^{n}),
   \end{displaymath}
   and thus if we assume that $\left(\textbf{X}^{n,\top}\boldsymbol{\Sigma}^{t}\textbf{X}^{n} \right)$ has full rank
   \begin{displaymath}
       \widetilde{\textbf{S}}^{L} = -\textbf{X}^{n,\top}(\boldsymbol{\Sigma}^{t})^{-1}\left( \mathbf{d}^{0}_{x}B(\textbf{\textit{T}})\mathds{1}^{\top}\textbf{Q}_{x} + \mathbf{d}^{0}_{y}B(\textbf{\textit{T}})\mathds{1}^{\top}\textbf{Q}_{y} \right)\widetilde{\textbf{V}}^{n+1}.
   \end{displaymath}
   From the S-step~\eqref{eq:NMMparS} we get, as $\varepsilon \to 0$,
   \begin{displaymath}
       \overline{\textbf{S}}^{n+1} = -\textbf{X}^{n,\top}(\boldsymbol{\Sigma}^{t})^{-1}\left( \mathbf{d}^{0}_{x}B(\textbf{\textit{T}})\mathds{1}^{\top}\textbf{Q}_{x} + \mathbf{d}^{0}_{y}B(\textbf{\textit{T}})\mathds{1}^{\top}\textbf{Q}_{y} \right)\textbf{V}^{n}.
   \end{displaymath}
   Then, 
   \begin{align*}
       \widehat{\textbf{X}}\widehat{\textbf{S}}\widehat{\textbf{V}}^{\top} &= \begin{bmatrix}
           \textbf{X}^{n}& \widetilde{\textbf{X}}^{n}
       \end{bmatrix}\begin{bmatrix}
           \overline{\textbf{S}}^{n+1} & \widetilde{\textbf{S}}^{L}\\
           \widetilde{\textbf{S}}^{K} & \textbf{0}
       \end{bmatrix}
       \begin{bmatrix}
           \textbf{V}^{n,\top}\\\widetilde{\textbf{V}}^{n,\top}
       \end{bmatrix}\\
       &= -\textbf{X}^{n}\textbf{X}^{n,\top}(\boldsymbol{\Sigma}^{t})^{-1}\left( \mathbf{d}^{0}_{x}B(\textbf{\textit{T}})\mathds{1}^{\top}\textbf{Q}_{x} + \mathbf{d}^{0}_{y}B(\textbf{\textit{T}})\mathds{1}^{\top}\textbf{Q}_{y} \right)\widehat{\textbf{V}}^{n}\widehat{\textbf{V}}^{n,\top}\\
       % &\quad -\textbf{X}^{n}\textbf{X}^{n,\top}(\boldsymbol{\Sigma}^{t})^{-1}\left( \mathbf{d}^{0}_{x}B(\textbf{\textit{T}})\mathds{1}^{\top}\textbf{Q}_{x} + \mathbf{d}^{0}_{y}B(\textbf{\textit{T}})\mathds{1}^{\top}\textbf{Q}_{y} \right)\widetilde{\textbf{V}}^{n+1}\widetilde{\textbf{V}}^{n+1,\top}\\
       &\quad -\widetilde{\textbf{X}}^{n+1}\widetilde{\textbf{X}}^{n+1,\top}(\boldsymbol{\Sigma}^{t})^{-1}\left( \mathbf{d}^{0}_{x}B(\textbf{\textit{T}})\mathds{1}^{\top}\textbf{Q}_{x} + \mathbf{d}^{0}_{y}B(\textbf{\textit{T}})\mathds{1}^{\top}\textbf{Q}_{y} \right)\textbf{V}^{n}\textbf{V}^{n,\top}.
   \end{align*}
   Since $\textbf{Q}_{x}\mathds{1},\textbf{Q}_{y}\mathds{1} \in \mathrm{range}(\textbf{V}^{n})$ and $\textbf{Q}_{x}\mathds{1},\textbf{Q}_{y}\mathds{1} \in \mathrm{range}(\widehat{\textbf{V}})$, we get
   \begin{align*}
        \widehat{\textbf{X}}\widehat{\textbf{S}}\widehat{\textbf{V}}^{\top} &= -\textbf{X}^{n}\textbf{X}^{n,\top}(\boldsymbol{\Sigma}^{t})^{-1}\left( \mathbf{d}^{0}_{x}B(\textbf{\textit{T}})\mathds{1}^{\top}\textbf{Q}_{x} + \mathbf{d}^{0}_{y}B(\textbf{\textit{T}})\mathds{1}^{\top}\textbf{Q}_{y} \right)\\
       &\quad -\widetilde{\textbf{X}}^{n+1}\widetilde{\textbf{X}}^{n+1,\top}(\boldsymbol{\Sigma}^{t})^{-1}\left( \mathbf{d}^{0}_{x}B(\textbf{\textit{T}})\mathds{1}^{\top}\textbf{Q}_{x} + \mathbf{d}^{0}_{y}B(\textbf{\textit{T}})\mathds{1}^{\top}\textbf{Q}_{y} \right)\\
       &= -(\boldsymbol{\Sigma}^{t})^{-1}\left( \mathbf{d}^{0}_{x}B(\textbf{\textit{T}})\mathds{1}^{\top}\textbf{Q}_{x} + \mathbf{d}^{0}_{y}B(\textbf{\textit{T}})\mathds{1}^{\top}\textbf{Q}_{y} \right),
   \end{align*}
   where we get the last equality since $(\boldsymbol{\Sigma}^{t})^{-1}\mathbf{d}^{0}_{x}B(\textbf{\textit{T}}),(\boldsymbol{\Sigma}^{t})^{-1}\mathbf{d}^{0}_{y}B(\textbf{\textit{T}})\in\mathrm{range}(\widehat{\textbf{X}}) $.

   It was shown in \cite[Theorem~4.5]{doi:10.1137/24M1646303} that the conservative truncation preserves the range space of the augmented spatial and angular basis. Thus, as $\varepsilon\to 0$ we get 
   \begin{displaymath}
       \textbf{g}^{n+1} = \textbf{X}^{n+1}\textbf{S}^{n+1}\textbf{V}^{n+1,\top} = -(\boldsymbol{\Sigma}^{t})^{-1}\left( \mathbf{d}^{0}_{x}B(\textbf{\textit{T}})\mathds{1}^{\top}\textbf{Q}_{x} + \mathbf{d}^{0}_{y}B(\textbf{\textit{T}})\mathds{1}^{\top}\textbf{Q}_{y} \right).
   \end{displaymath}
   The rest of the proof follows on the lines of \Cref{theorem:APFull} with $g_{\kappa,\ell}^{n+1} = (\textbf{g}^{n+1})_{\varrho(\kappa),\ell}$, for $\kappa\in\mathcal{K}^{I}$ and $\ell=1,\ldots,N_{q}$.
\end{proof}

%%%%%%%%%%%%%%%%%%%%%%%%%%%%% Sub-sub-section %%%%%%%%%%%%%%%%%%%%%%%%%%%%%%%%
% \subsubsection{Enforcing zero density condition}
%%%%%%%%%%%%%%%%%%%%%%%%%%%%%%%%%%%%%%%%%%%%%%%%%%%%%%%%%%%%%%%%%%%%%%%%%%%%%%

%%%%%%%%%%%%%%%%%%%%%%%%%%%%% Sub-section %%%%%%%%%%%%%%%%%%%%%%%%%%%%%%%%%%%%
\subsection{Energy stability}
%%%%%%%%%%%%%%%%%%%%%%%%%%%%%%%%%%%%%%%%%%%%%%%%%%%%%%%%%%%%%%%%%%%%%%%%%%%%%%

\begin{theorem}
    For a given spatial grid size $\Delta x, \Delta y$, let the step size $\Delta t$ satisfy the CFL condition \eqref{eq:CFLcond} for all $\Omega_{x}^{\ell},\Omega_{y}^{\ell}\neq 0$. Then, the proposed macro--micro parallel BUG scheme is energy stable, i.e.
    \begin{equation*}    
        e^{n+1} \leq e^{n},
    \end{equation*}
    where the energy is defined as
        \begin{displaymath}
        e^{n} = \norm{\frac{1}{c}B^{n} + \frac{\varepsilon^{2}}{c}h^{n}}^{2} + \frac{1}{2\pi}\norm{\frac{\varepsilon}{c}\normalfont{\textbf{X}^{n}\textbf{S}^{n}\textbf{V}^{n,\top}}}^{2} + \frac{2}{5}\norm{\frac{\sqrt{ac_{\nu}}}{2\pi}(T^{n})^{5/2}}^{2}.
    \end{displaymath}
\end{theorem}

\begin{remark}
    The proof of this theorem is based on the proof of Theorem~\ref{thm:StabNMM} and the energy stability for the AP augmented BUG integrator for thermal radiative transfer equations in slab geometry in \cite[Theorem~4.6]{doi:10.1137/24M1646303}. Note that energy stability in \cite[Theorem~4.6]{doi:10.1137/24M1646303} is for an associated linearized problem and thus the energy stability (Theorem~\ref{thm:StabNMM}) of the fully nonlinear problem~\eqref{eq:RTE} proved in this paper is a key ingredient. The extension to the higher dimensional setting \eqref{eq:RTE} and the change of angular discretization don't affect the stability of the augmented BUG integrator. Thus, the macro--micro augmented BUG scheme is energy stable for the nonlinear closure given by the Stefan-Boltzmann law in the sense of \Cref{thm:StabNMM}.

    The discussion in \cite[Section~3.2]{ceruti2023parallel} shows that if $\textbf{g}^{n+1}_{p} = \textbf{X}^{n+1}_{p}\textbf{S}^{n+1}_{p}\textbf{V}^{n+1,\top}_{p} $ is the solution of the parallel BUG at time $t_{n}$ and $\textbf{g}^{n+1}_{a} = \textbf{X}^{n+1}_{a}\textbf{S}^{n+1}_{a}\textbf{V}^{n+1,\top}_{a} $ is the solution obtained from the augmented BUG integrator \cite{ceruti2024robust}, then $\norm{\textbf{g}^{n+1}_{p}} \leq \norm{\textbf{g}^{n+1}_{a}}$. 
    
    Thus, combining the above two arguments, it is straightforward to show that the macro--micro parallel BUG scheme is energy stable for the nonlinear closure given by the Stefan-Boltzmann law. Since the arguments have already been introduced in previous works, the details of re-writing the proof have been omitted and the readers are referred to \cite{doi:10.1137/24M1646303} and \cite{ceruti2023parallel} for further details.
\end{remark}

\section{Numerical experiments} \label{section:Numericalresults}
%%%%%%%%%%%%%%%%%%%%%%%%%%%%%%%%%%%%%%%%%%%%%%%%%%%%%%%%%%%%%%%%%%%%%%%%%%%%%%
This section presents numerical examples to verify the theoretical properties of the integrators and demonstrate their efficiency. In all examples, the initial condition is specified for the material temperature while the particle density is assumed to be at an equilibrium. That is, the initial particle density is set as $f_{I}(\textbf{\textit{x}},\boldsymbol{\Omega}) = \frac{ac}{2\pi}T_{I}(\textbf{\textit{x}})^{4}$. The truncation tolerance, $\vartheta$, of the macro--micro parallel BUG scheme is set as $\vartheta = 10^{-2}\lVert \boldsymbol{\Sigma} \rVert_{2}$ in all experiments, where $\boldsymbol{\Sigma}$ contains the singular values of the coefficient matrix. We are mainly interested in the case where the absorption effects are dominant and thus we set $\sigma^{s} = 0$ in all the test cases. We define the radiation temperature as $T_{\mathrm{rad}} \coloneqq (2\pi(B(T)+\varepsilon^{2}h)/ac)^{1/4}$. All codes used to generate the results are openly available in \cite{CodesForPaper}.

Throughout this section, we refer to the full-rank macro--micro scheme \eqref{eq:FDNMM} as the full solver/ integrator, and the corresponding material temperature, scalar flux, and radiation temperatures are denoted by $T^{\mathrm{Full}}$, $\phi^{\mathrm{Full}}$, and $T_{\mathrm{rad}}^{\mathrm{Full}}$, respectively. The macro--micro parallel BUG scheme described in \Cref{section:ParallelDLRAforNMM} is referred to as the parallel BUG solver/ integrator and $T^{\vartheta}$, $\phi^{\vartheta}$, and $T_{\mathrm{rad}}^{\vartheta}$ denote the corresponding material temperature, scalar flux, and radiation temperature, respectively, for a given tolerance $\vartheta$. To study the behavior of our scheme in the asymptotic limit $\varepsilon\to 0$, we compare our schemes to the solution of the Rosseland equation \eqref{eq:RosselandApprox} which is denoted by $T^{R}$. 

%%%%%%%%%%%%%%%%%%%%%%%%%%%%% Sub-section %%%%%%%%%%%%%%%%%%%%%%%%%%%%%%%%%%%%
\paragraph{AP property}
%%%%%%%%%%%%%%%%%%%%%%%%%%%%%%%%%%%%%%%%%%%%%%%%%%%%%%%%%%%%%%%%%%%%%%%%%%%%%%

The first test case is set up to study the AP property and energy decay of the proposed integrators and uses the parameters described in \cite{10.5445/IR/1000154134}. The details are summarized in \Cref{tab:settingsHeatedWall}. The test case is defined for the spatial domain $\textbf{\textit{x}}\in[0,0.002]^{2}$ with material density $\rho=0.01$ g cm$^{-3}$ and the temperature is initially distributed as a Gaussian centered at $ \textbf{\textit{x}}_{0} = (0.001,0.001) $, i.e.
\begin{displaymath}
    T_{I}(\textbf{\textit{x}}) = \frac{1}{2\pi\sigma^{2}}\cdot\exp{\left(-\frac{\norm{\textbf{\textit{x}} - \textbf{\textit{x}}_{0}}_{2}^{2}}{2\sigma^2}\right)},
\end{displaymath}
 where $\sigma = 10^{-4}$. Furthermore, $ T_{I}(\textbf{\textit{x}}) $ is re-scaled such that the maximum temperature is $80$ eV and the cut-off minimum is $0.02~\mathrm{eV}$. The initially distributed particles move in all directions, and as time progresses, they heat the background material. A temperature heat front, traveling outwards from the center of the domain, develops in the material resulting in further emission of particles. The material temperature at the boundary is kept at a constant temperature of $0.02 ~\mathrm{eV}$. 

\begin{table}[htbp!]
    \centering
    \begin{tabular}{|c|c|}
    \hline
         % spatial domain, $D$& $[0,0.002]\times[0,0.002]$ in cm\\
         number of spatial cells, $N_{x},N_{y}$ & 52,52\\
         quadrature order, $q$& 30\\
         absorption coefficient, $\sigma^{a}$& $10,799.13607$ cm$^{-1}$\\
         speed of light, $c$ & $2.99792458\cdot 10^{10}$ cm s$^{-1}$ \\
         radiation constant, $a$& $7.565766\cdot 10^{-15}$ erg cm$^{-3}$K$^{-4}$ \\
         specific heat, $c_{\nu}$& $0.831\cdot 10^{5} $ J g$^{-1}$K$^{-1}$\\
         \hline
         
    \end{tabular}
    \caption{Material constants and settings for the Gaussian and Marshak wave test case as given in \cite{10.5445/IR/1000154134}.}
    \label{tab:settingsHeatedWall}
\end{table}

To demonstrate the AP property of the full-rank and parallel BUG integrator we compute the solution at $t_{\mathrm{end}} = 5~\mathrm{ps}$ for Knudsen numbers $\varepsilon \in \{1,5\cdot 10^{-1}, 10^{-1},5\cdot 10^{-2},10^{-2},10^{-3},10^{-4}\}$. Note that simulations with $\varepsilon = 1$ correspond to the kinetic regime while $\varepsilon=0.0001$ correspond to the diffusive regime. The solution of the integrators is compared to the solution of the Rosseland equation at $t_{\mathrm{end}}$. The relative error between the material temperature of both the integrators and the Rosseland equation at $t_{\mathrm{end}}$ is plotted in Figure~\ref{fig:APproperty}. We see that as $\varepsilon\to 0$, the solution of the full solver and the parallel BUG solver converge to the Rosseland equation. Additionally, from Figure~\ref{fig:Energydecay} we see that energy decays for both the full solver and the parallel BUG solver in the absence of a source term.

\begin{figure}[htbp!]
    \centering
    \begin{subfigure}[t]{0.45\linewidth}
    \centering
        \includegraphics[width=\linewidth]{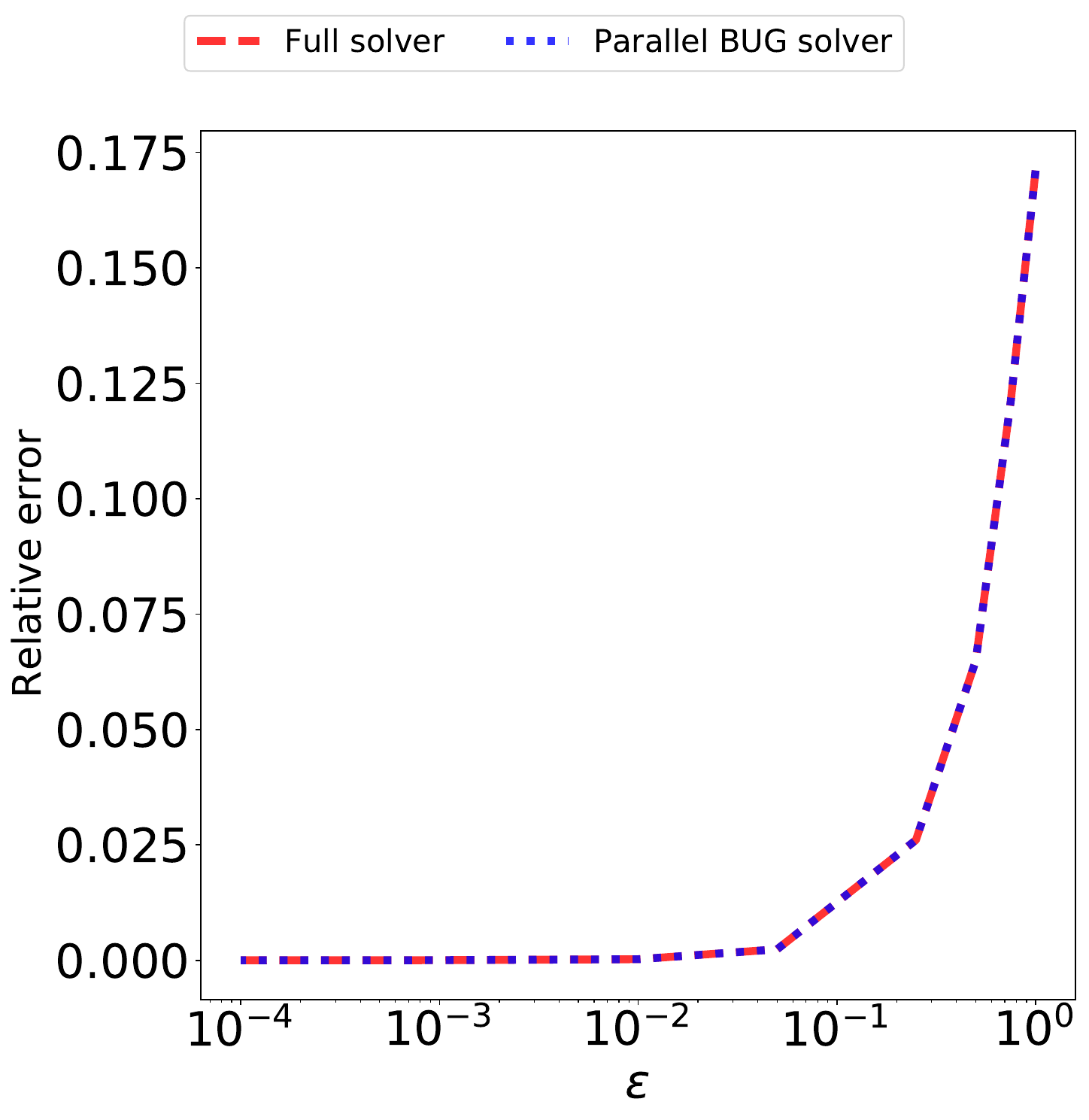}
        \caption{AP property}
        \label{fig:APproperty}
    \end{subfigure}
    \begin{subfigure}[t]{0.45\linewidth}
    \centering
        \includegraphics[width=\linewidth]{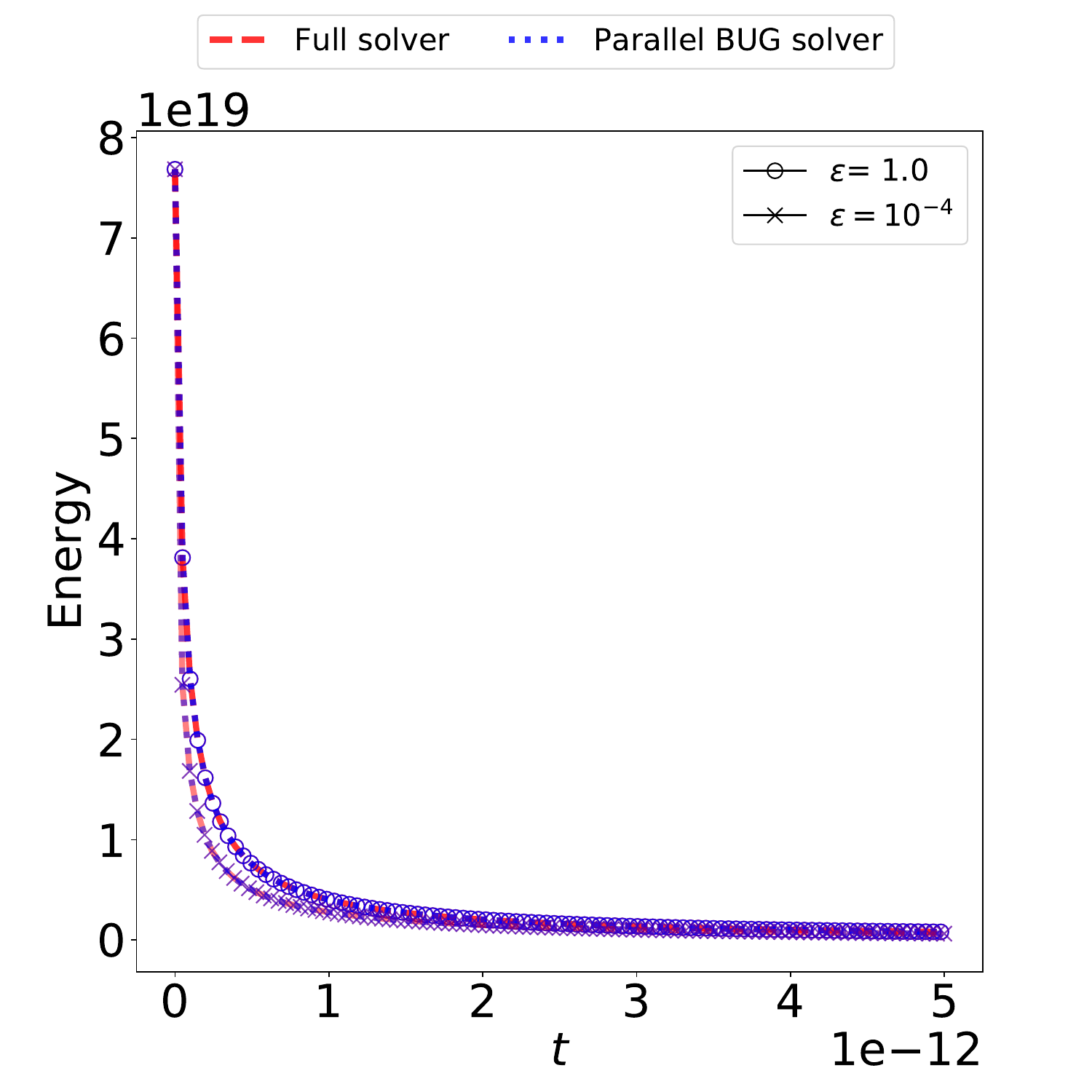}
        \caption{Energy decay over time}
        \label{fig:Energydecay}
    \end{subfigure}
    \caption{Left: Relative error of $T^{\mathrm{Full}}$ and $T^{\vartheta}$ compared to $T^{R}$ for $\varepsilon\in \{1,5\cdot 10^{-1}, 10^{-1},5\cdot 10^{-2},10^{-2},10^{-3},10^{-4}\}$. Right: Energy decay of the full solver and parallel BUG solver for $\varepsilon\in\{1.0,10^{-4}\}$.}
    \label{fig:GaussianTestcase}
\end{figure}

%%%%%%%%%%%%%%%%%%%%%%%%%%%%% Sub-section %%%%%%%%%%%%%%%%%%%%%%%%%%%%%%%%%%%%
\paragraph{Marshak wave}
%%%%%%%%%%%%%%%%%%%%%%%%%%%%%%%%%%%%%%%%%%%%%%%%%%%%%%%%%%%%%%%%%%%%%%%%%%%%%%

The Marshak wave test case is a two-dimensional extension of the test case presented in \cite{10.5445/IR/1000154134} for the one-dimensional thermal radiative transfer equations. The spatial domain and other parameters are the same as in the previous test case and are given in \Cref{tab:settingsHeatedWall}. The initial temperature of the material is $ 0.02~\mathrm{eV}$ and a constant temperature source of $80~\mathrm{eV}$ is applied to the left wall which is switched on at the initial time while the remaining boundaries are maintained at $0.02~\mathrm{eV}$.

As time progresses, particles stream into the domain from the left wall and a temperature heat front traveling to the right wall develops. For $\varepsilon = 1.0$, the cross-section of the material temperature and scalar flux through $y = 0.001$ at times $1,2,3,4,$ and $5~\mathrm{ps}$  are plotted in \Cref{fig:HeatedWall}. Additionally, for $\varepsilon=10^{-4}$ we plot the material temperature and scalar flux through $y = 0.001$ at $5~\mathrm{ps}$ for the full solver, parallel BUG solver and the Rosseland equation in Figure~\ref{fig:DiffLimit}. We see that the full solver and parallel BUG solver both accurately capture the Rosseland diffusion limit.

\begin{figure}[htbp!]
    \centering
    \begin{subfigure}[t]{0.45\linewidth}
    \centering
        \includegraphics[width=0.9\linewidth]{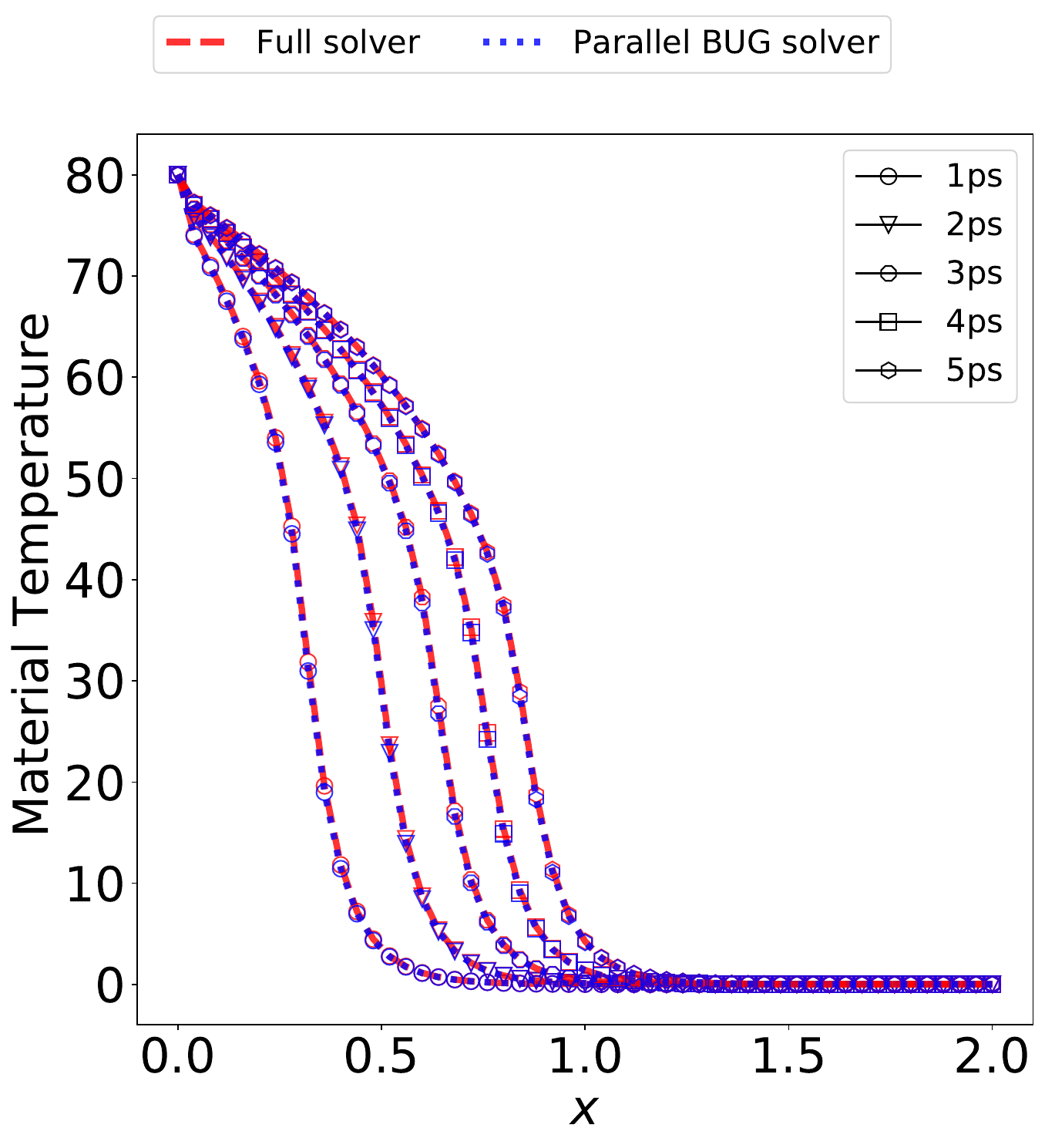}
    \end{subfigure}
    \begin{subfigure}[t]{0.45\linewidth}
    \centering
        \includegraphics[width=0.9\linewidth]{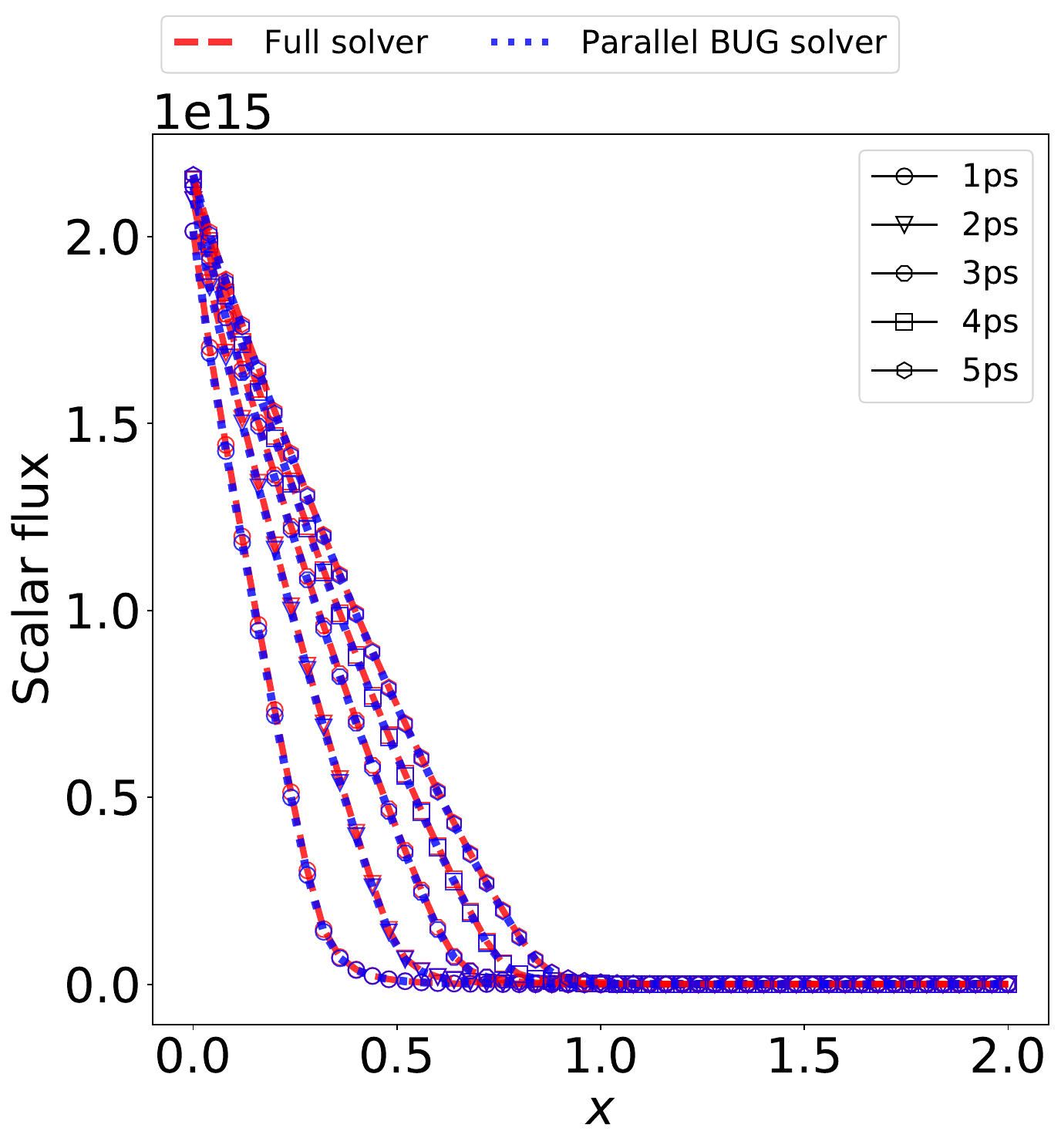}
    \end{subfigure}
    \caption{Cross-section of the material and scalar flux for the Marshak wave test case at times $1$, $2$, $3$, $4$, and $5$ ps through $y=0.001$ for $\varepsilon=1.0$.}
    \label{fig:HeatedWall}
\end{figure}

\begin{figure}[htbp!]
    \centering
    \begin{subfigure}[t]{0.45\linewidth}\centering
        \includegraphics[width=0.9\linewidth]{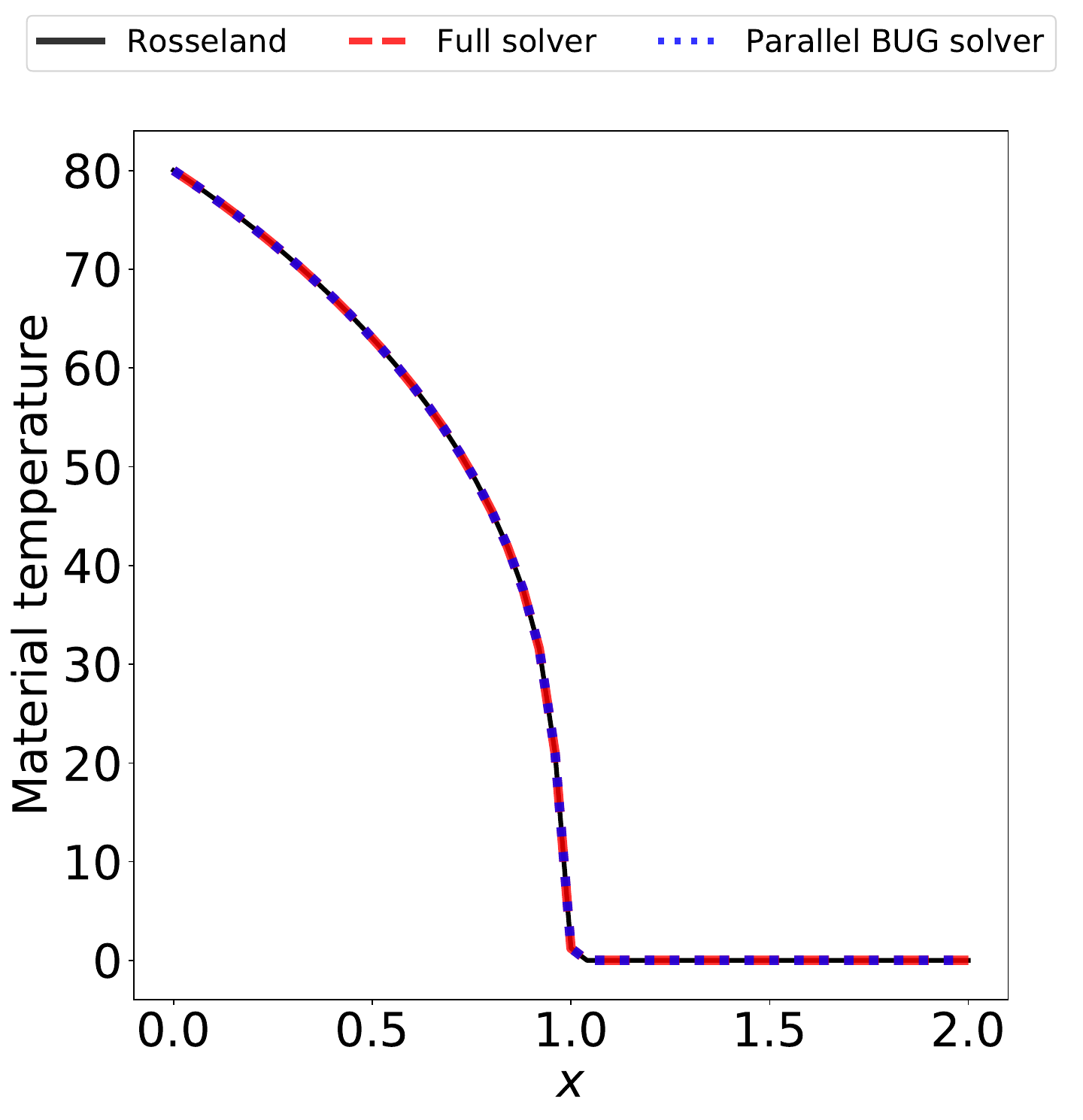}
    \end{subfigure}
    \begin{subfigure}[t]{0.45\linewidth}\centering
        \includegraphics[width=0.9\linewidth]{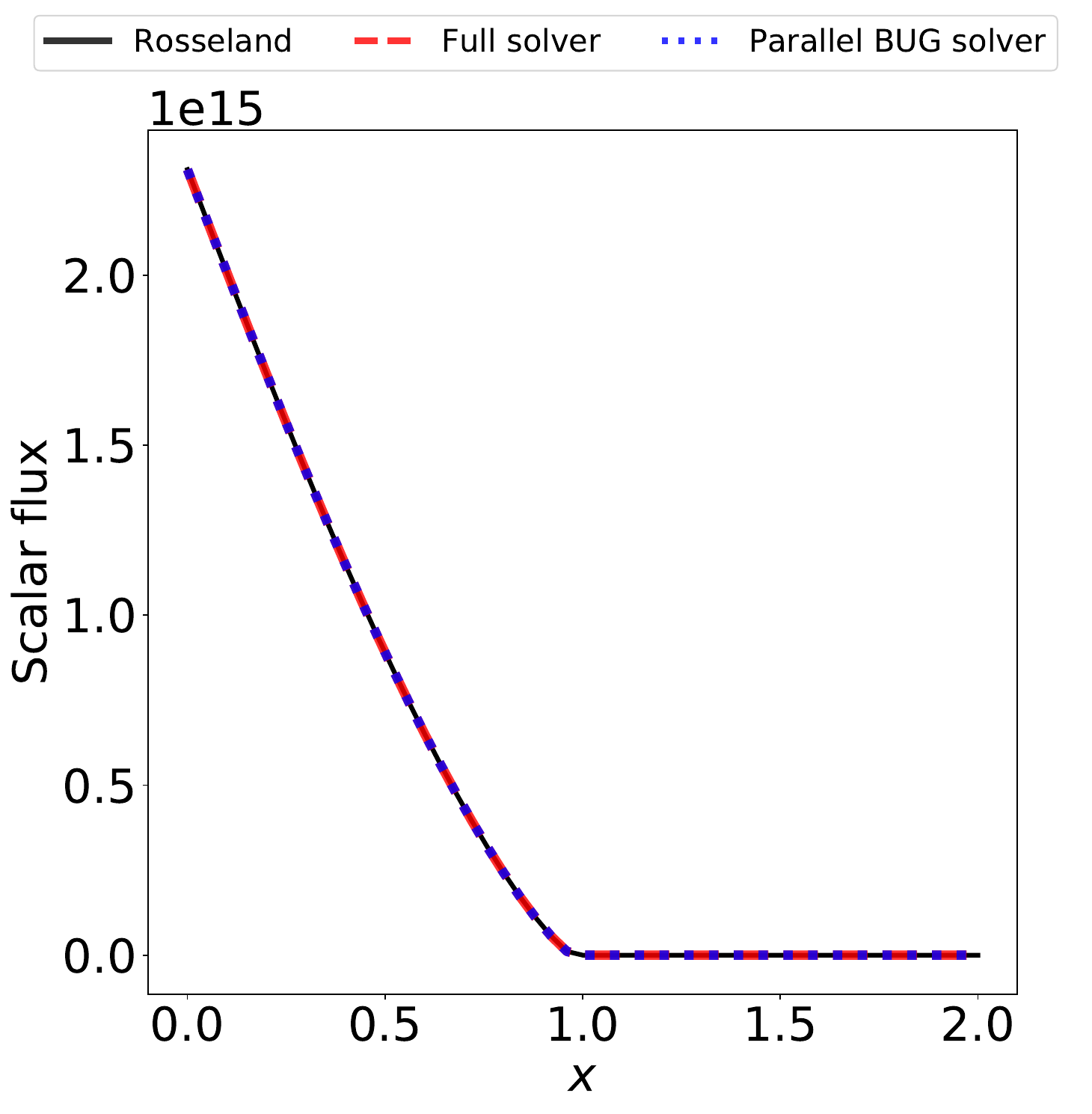}
    \end{subfigure}
    \caption{Cross-section of the material and scalar flux for the Marshak wave test case at time $5$ ps through $y=0.001$ for the diffusive limit with $\varepsilon=10^{-4}$.}
    \label{fig:DiffLimit}
\end{figure}

%%%%%%%%%%%%%%%%%%%%%%%%%%%%% Sub-section %%%%%%%%%%%%%%%%%%%%%%%%%%%%%%%%%%%%
\paragraph{A Hohlraum problem}
%%%%%%%%%%%%%%%%%%%%%%%%%%%%%%%%%%%%%%%%%%%%%%%%%%%%%%%%%%%%%%%%%%%%%%%%%%%%%%
Next, we consider a version of the hohlraum test case described in \cite{MR4673300,MR2657865,osti_800993}. The test case models a hohlraum for testing inertial confinement fusion in Cartesian coordinates. The problem's geometry and parameters are given in \Cref{fig:geometryHohlraum} and \Cref{tab:settingsHohlraum}, respectively. The blue regions in \Cref{fig:geometryHohlraum} are pure absorbers with $\sigma^{a} = 100 $ cm$^{-1}$ and $\rho c_{\nu} = 5.0 \cdot 10^{5}~\mathrm{GJ}~\mathrm{cm}^{-3}~\mathrm{keV}^{-1}$ and the white region is a pure vacuum with $\sigma^{a} = 0$ cm$^{-1}$ and $\rho c_{\nu} = 1.0 \cdot 10^{99} \mathrm{GJ}~\mathrm{cm}^{-3}~\mathrm{keV}^{-1}$. Note that the absorption coefficient $\sigma^{a}$ is kept constant following \cite{osti_800993} while in \cite{MR4673300,MR2657865} the absorption coefficient depends on temperature.

The initial temperature of the entire hohlraum is $10^{-3}~\mathrm{keV}$. The left boundary is heated at a constant temperature of $1~\mathrm{keV}$ and the right boundary is maintained at $10^{-3}~\mathrm{keV}$. There is an inflow of particles from the left and right boundaries into the hohlraum at equilibrium with the temperature while the top and bottom boundaries are reflective. These boundary conditions for the hohlraum have been described in \cite{MR4673300} and no analytical solution exists for this problem.

 \begin{table}[htbp!]
        \centering
        \begin{tabular}{|c|c|}
        \hline
             $N_{x},N_{y}$ & 102,102\\
             quadrature order, $q$& 30\\
             speed of light, $c$ & $29.98$ $\mathrm{m}~\mathrm{ns}^{-1}$ \\
             radiation constant, $a$& $0.01372$  $\mathrm{GJ}~\mathrm{cm}^{-3}~\mathrm{keV}^{-4}$ \\
             \hline
        \end{tabular}
        \caption{Material constants and settings for the hohlraum test case~\cite{MR2657865}.}
        \label{tab:settingsHohlraum}
    \end{table}
    
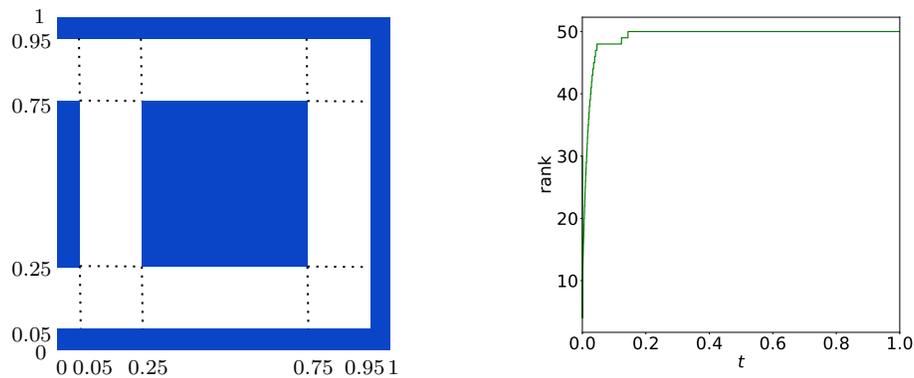
\begin{figure}[htbp!]
\centering
\tikzset{every picture/.style={line width=0.75pt}} %set default line width to 0.75pt        
\begin{subfigure}[t]{0.45\linewidth}\centering
\begin{tikzpicture}[x=0.75pt,y=0.75pt,yscale=-0.7,xscale=0.7]

%Shape: Rectangle [id:dp5772571866309608] 
\draw  [draw opacity=0][fill={rgb, 255:red, 10; green, 69; blue, 199 }  ,fill opacity=1 ] (240,90) -- (256.22,90) -- (256.22,210.17) -- (240,210.17) -- cycle ;
%Shape: Rectangle [id:dp5461518848362614] 
\draw  [draw opacity=0][fill={rgb, 255:red, 10; green, 69; blue, 199 }  ,fill opacity=1 ] (239.83,254.89) -- (480,254.89) -- (480,270.17) -- (239.83,270.17) -- cycle ;
%Shape: Rectangle [id:dp4349012986144877] 
\draw  [draw opacity=0][fill={rgb, 255:red, 10; green, 69; blue, 199 }  ,fill opacity=1 ] (480,30) -- (480,270.17) -- (465.94,270.17) -- (465.94,30) -- cycle ;
%Shape: Rectangle [id:dp6536635013850788] 
\draw  [draw opacity=0][fill={rgb, 255:red, 10; green, 69; blue, 199 }  ,fill opacity=1 ] (239.83,30) -- (480,30) -- (480,45.28) -- (239.83,45.28) -- cycle ;
%Shape: Square [id:dp32597393359368354] 
\draw  [draw opacity=0][fill={rgb, 255:red, 10; green, 69; blue, 199 }  ,fill opacity=1 ] (300.89,90.33) -- (420.22,90.33) -- (420.22,209.67) -- (300.89,209.67) -- cycle ;
%Straight Lines [id:da5791473241274435] 
\draw  [dash pattern={on 0.84pt off 2.51pt}]  (256.22,90) -- (300.89,90.33) ;
%Straight Lines [id:da4749256856510594] 
\draw  [dash pattern={on 0.84pt off 2.51pt}]  (256.22,209.33) -- (300.89,209.67) ;
%Straight Lines [id:da7598574367186455] 
\draw  [dash pattern={on 0.84pt off 2.51pt}]  (420.22,90.33) -- (464.89,90.67) ;
%Straight Lines [id:da03699942883961416] 
\draw  [dash pattern={on 0.84pt off 2.51pt}]  (420.22,209.67) -- (464.89,210) ;
%Straight Lines [id:da47358640676551966] 
\draw  [dash pattern={on 0.84pt off 2.51pt}]  (300.22,45.33) -- (300.89,90.33) ;
%Straight Lines [id:da7379778724538079] 
\draw  [dash pattern={on 0.84pt off 2.51pt}]  (419.56,45.33) -- (420.22,90.33) ;
%Straight Lines [id:da7788942946663048] 
\draw  [dash pattern={on 0.84pt off 2.51pt}]  (300.89,209.67) -- (301.56,254.67) ;
%Straight Lines [id:da21516617137638683] 
\draw  [dash pattern={on 0.84pt off 2.51pt}]  (420.22,209.67) -- (420.89,254.67) ;
%Straight Lines [id:da44753340734612146] 
\draw  [dash pattern={on 0.84pt off 2.51pt}]  (255.56,45) -- (256.22,90) ;
%Straight Lines [id:da7450819663672112] 
\draw  [dash pattern={on 0.84pt off 2.51pt}]  (256.22,209.33) -- (256.89,254.33) ;

% Text Node
\draw (221,22.4) node [anchor=north west][inner sep=0.75pt]  [font=\scriptsize]  {$1$};
% Text Node
\draw (222,264.4) node [anchor=north west][inner sep=0.75pt]  [font=\scriptsize]  {$0$};
% Text Node
\draw (205,251.4) node [anchor=north west][inner sep=0.75pt]  [font=\scriptsize]  {$0.05$};
% Text Node
\draw (249,275) node [anchor=north west][inner sep=0.75pt]  [font=\scriptsize]  {$0.05$};
% Text Node
\draw (237,275) node [anchor=north west][inner sep=0.75pt]  [font=\scriptsize]  {$0$};
% Text Node
\draw (205,40.4) node [anchor=north west][inner sep=0.75pt]  [font=\scriptsize]  {$0.95$};
% Text Node
\draw (205,204.4) node [anchor=north west][inner sep=0.75pt]  [font=\scriptsize]  {$0.25$};
% Text Node
\draw (205,86.4) node [anchor=north west][inner sep=0.75pt]  [font=\scriptsize]  {$0.75$};
% Text Node
\draw (289,275) node [anchor=north west][inner sep=0.75pt]  [font=\scriptsize]  {$0.25$};
% Text Node
\draw (408,275) node [anchor=north west][inner sep=0.75pt]  [font=\scriptsize]  {$0.75$};
% Text Node
\draw (445,275) node [anchor=north west][inner sep=0.75pt]  [font=\scriptsize]  {$0.95$};
% Text Node
\draw (476,275) node [anchor=north west][inner sep=0.75pt]  [font=\scriptsize]  {$1$};
\end{tikzpicture}
\end{subfigure}
\begin{subfigure}[b]{0.45\linewidth}
\centering
    \centering\includegraphics[width=0.75\linewidth]{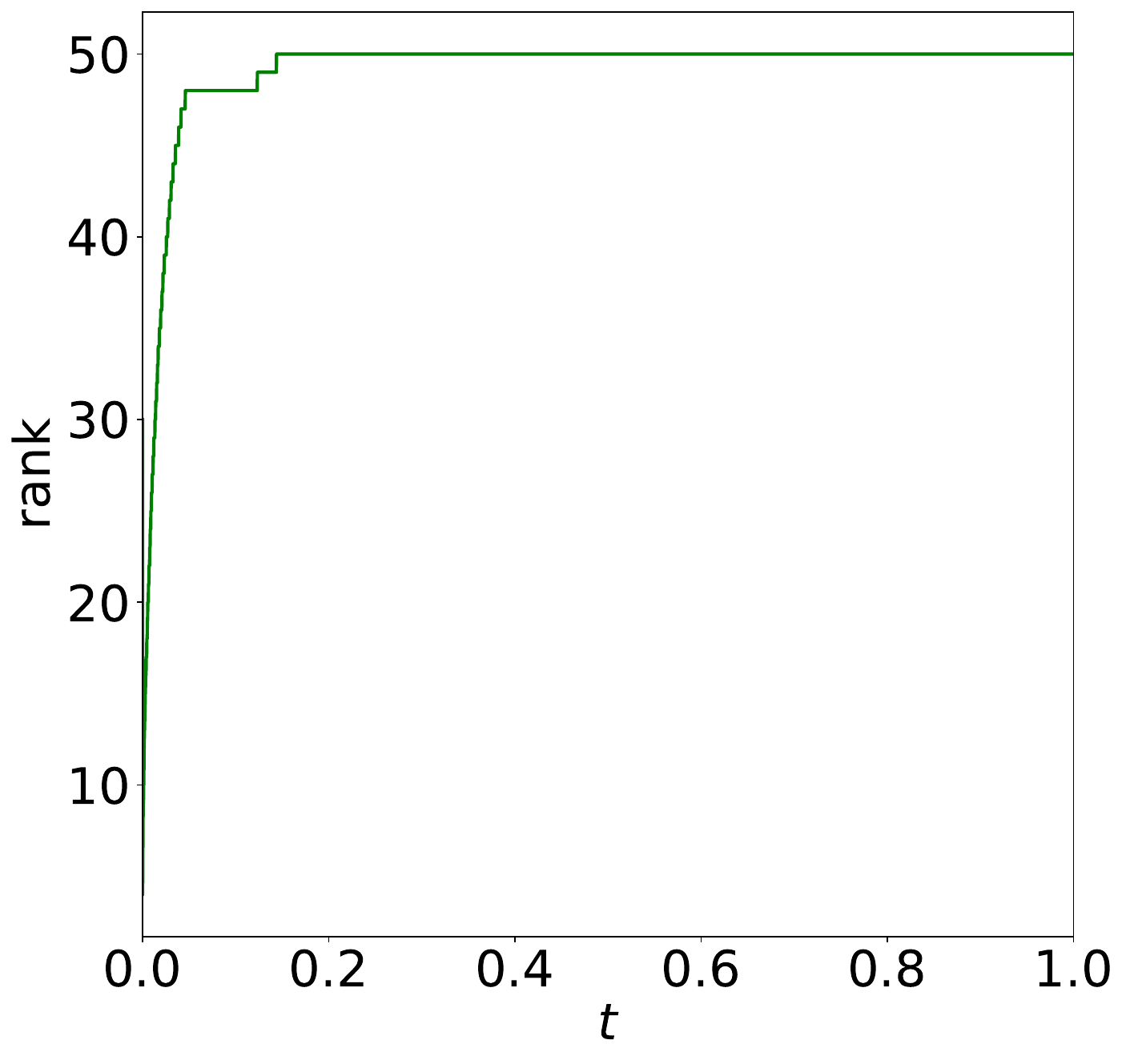}
\end{subfigure}
\caption{Left: Geometry of the hohlraum as described in \cite{MR2657865}. Right: Rank over time for the hohlraum test case with $\vartheta = 10^{-2}$ until $1~\mathrm{ns}$.}
\label{fig:geometryHohlraum}
\end{figure}

We simulate the hohlraum test case for $\varepsilon = 1.0$ until $t_{\mathrm{end}}=1~\mathrm{ns}$. The material temperature and radiation temperature at $t_{\mathrm{end}}$ are plotted in Figure~\ref{fig:hohlraumMatTemp} and Figure~\ref{fig:hohlraumRadTemp}, respectively. We see that the parallel BUG solver approximates the full-rank solution accurately while only requiring 8551 seconds compared to 15,311 seconds for the full solver. Moreover, we see from Figure~\ref{fig:geometryHohlraum} that the rank of the parallel BUG solution does not grow beyond $50$ despite maximal rank being much higher. 

\begin{figure}[htbp!]
    \centering
        \begin{subfigure}[t]{0.32\linewidth}
        \centering
        \includegraphics[width=\linewidth]{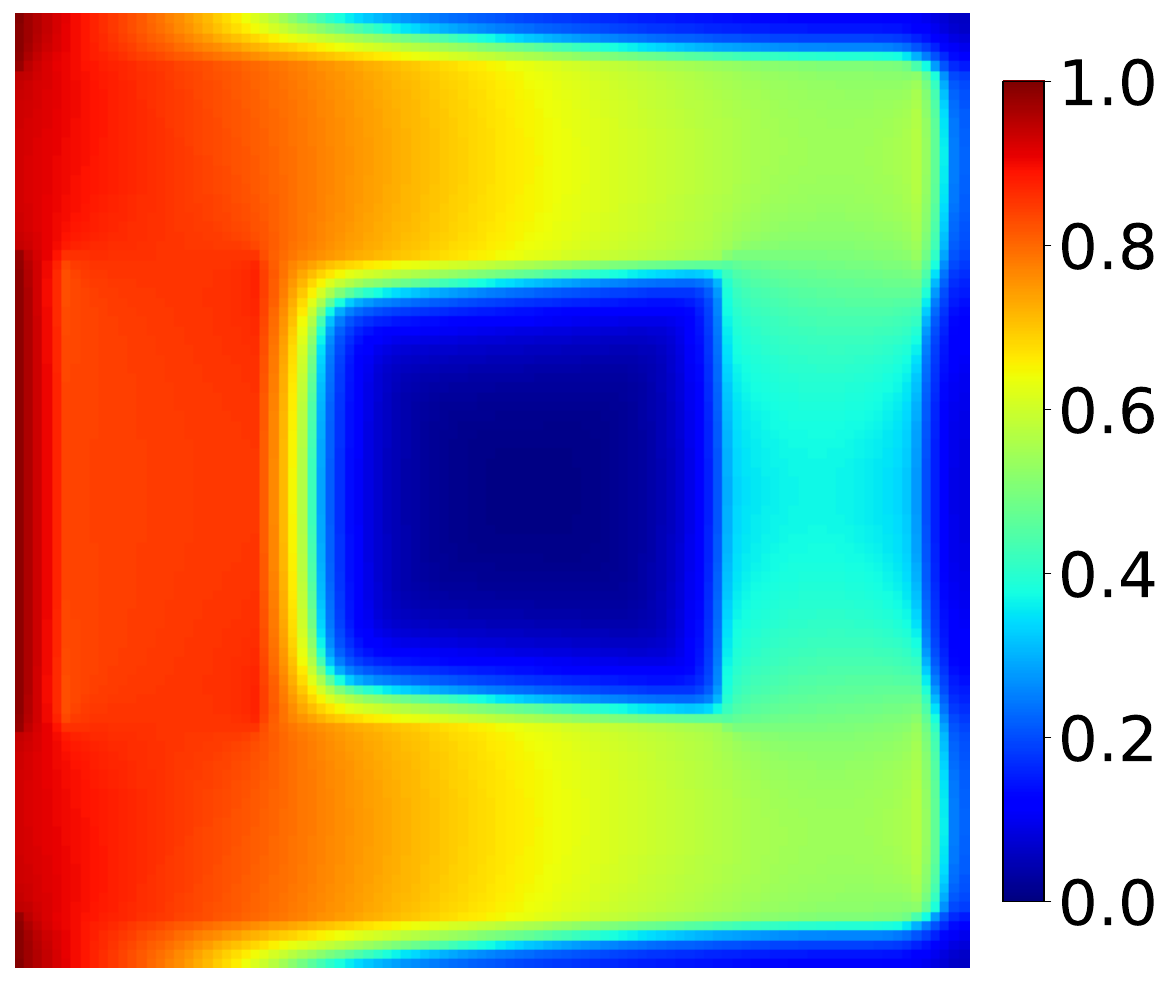}
        \caption{$T_{\mathrm{rad}}^{\mathrm{Full}}$}
    \end{subfigure}
    \begin{subfigure}[t]{0.32\linewidth}
    \centering
        \includegraphics[width=\linewidth]{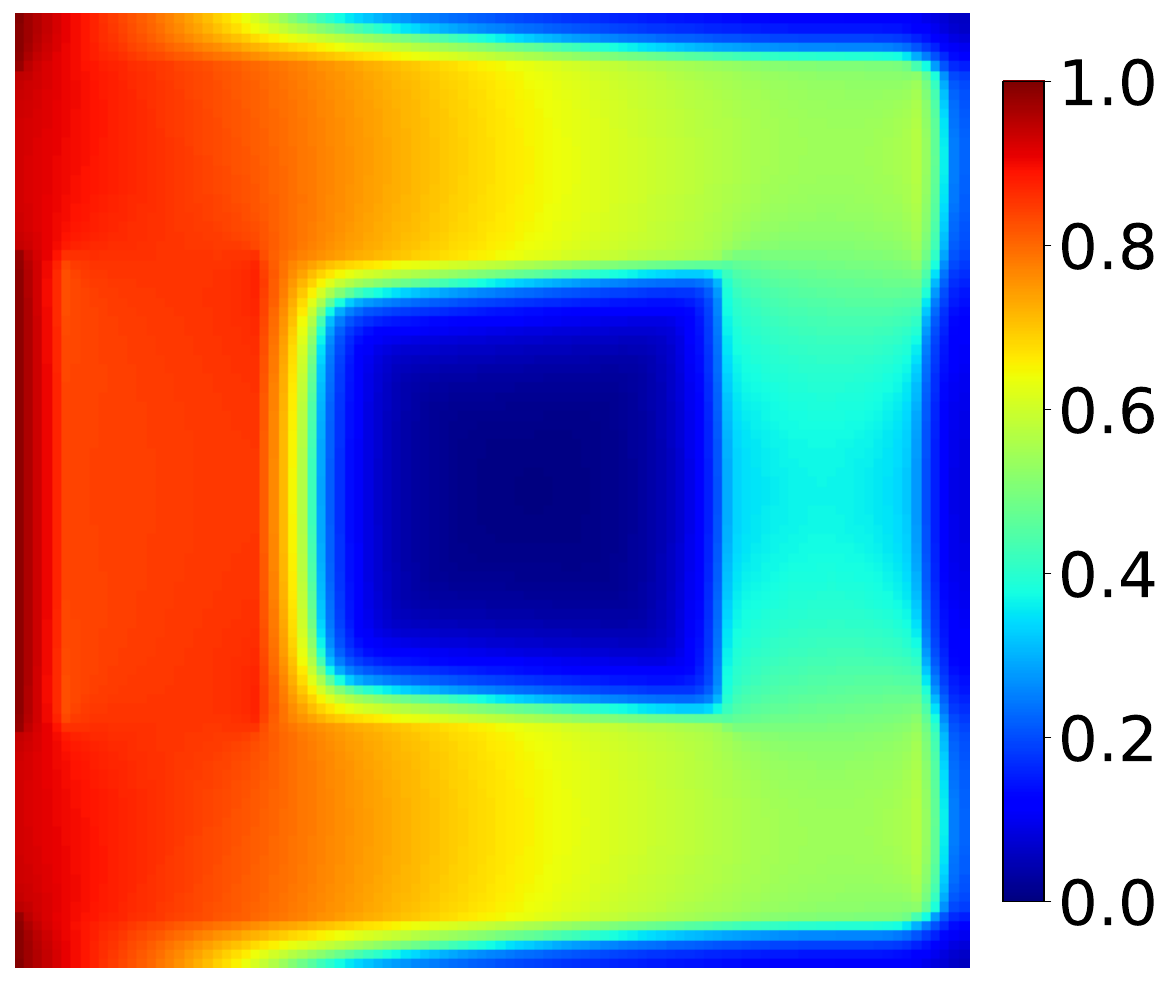}
        \caption{$T_{\mathrm{rad}}^{\vartheta}$}
    \end{subfigure}
    \begin{subfigure}[t]{0.32\linewidth}
    \centering
        \includegraphics[width=\linewidth]{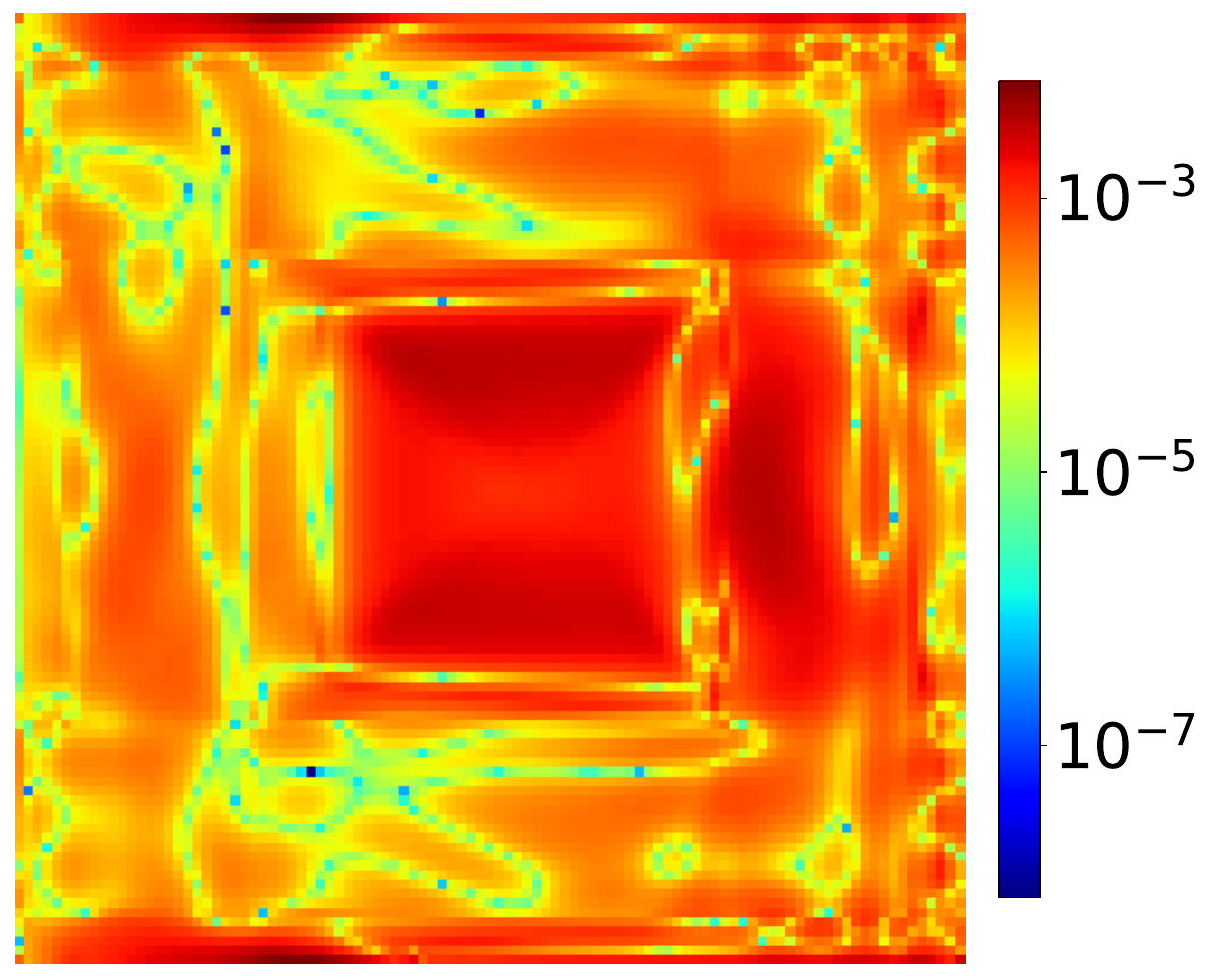}
        \caption{$\lvert T_{\mathrm{rad}}^{\mathrm{Full}} - T_{\mathrm{rad}}^{\vartheta}\rvert$}
    \end{subfigure}
    \caption{Radiation temperature at $t = 1~\mathrm{ns}$  for the hohlraum test case.}
    \label{fig:hohlraumMatTemp}
\end{figure}
\begin{figure}[htbp!]
    \centering
    \begin{subfigure}[t]{0.32\linewidth}
        \centering
        \includegraphics[width=\linewidth]{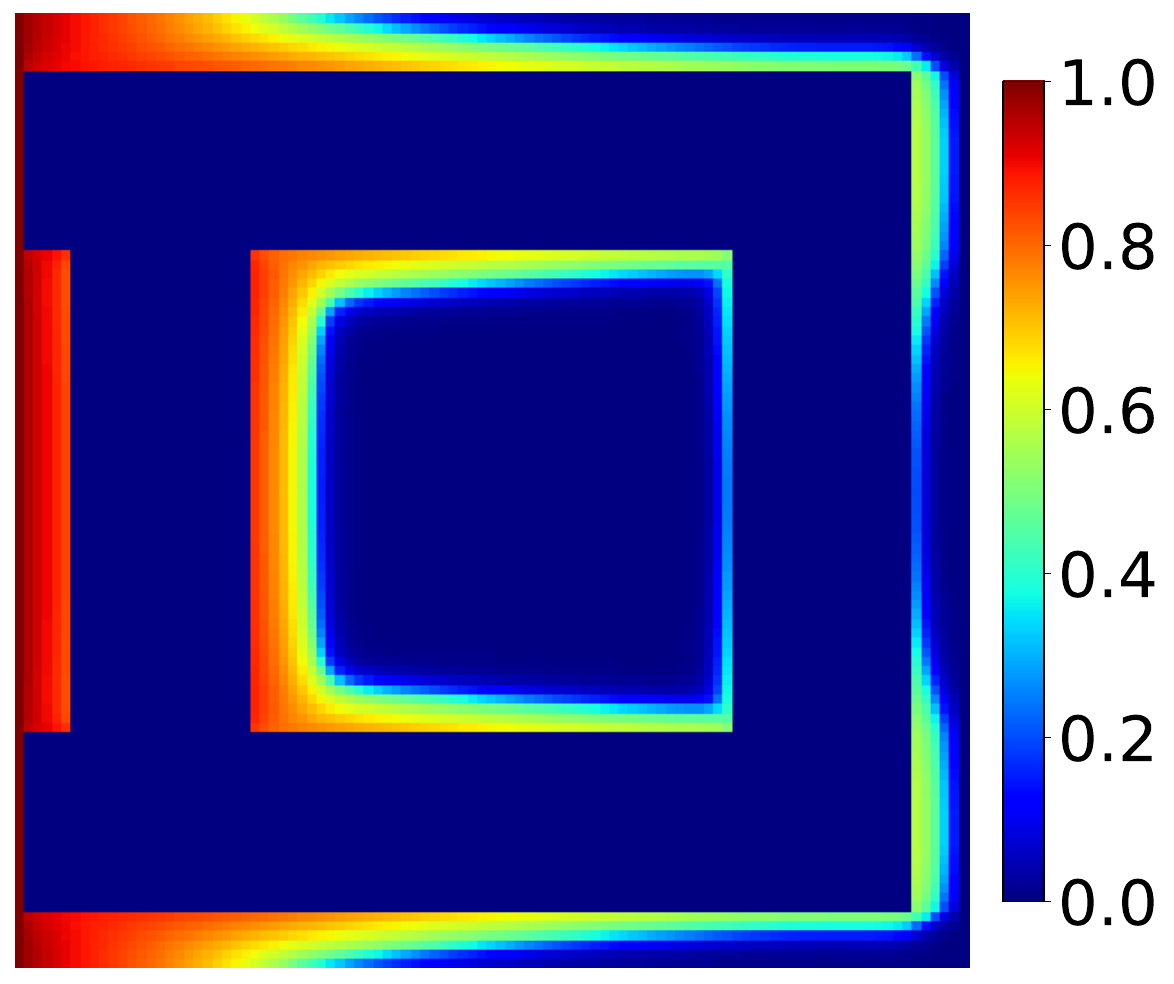}
        \caption{$T^{\mathrm{Full}}$}
    \end{subfigure}
    \begin{subfigure}[t]{0.32\linewidth}
    \centering
        \includegraphics[width=\linewidth]{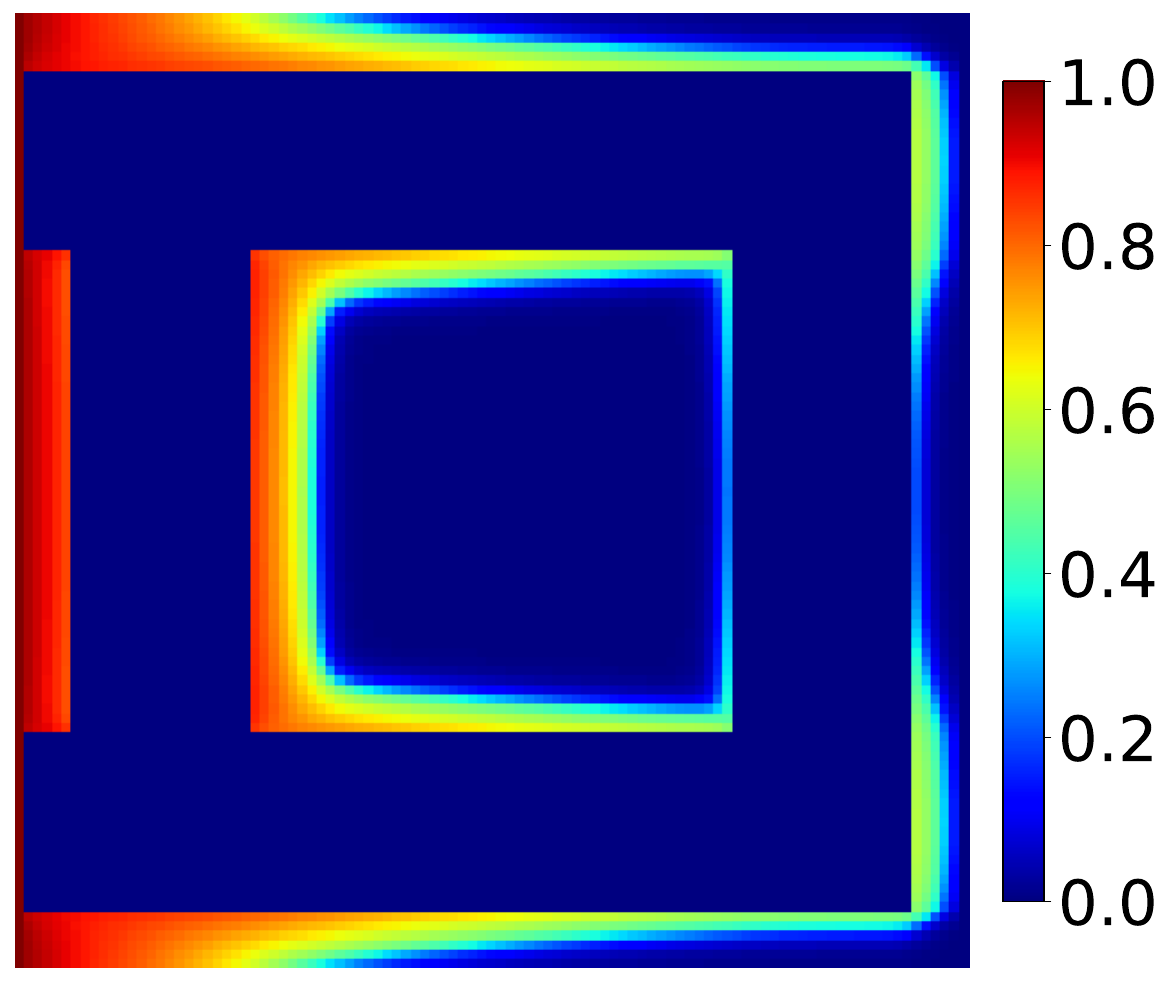}
        \caption{$T^{\vartheta}$}
    \end{subfigure}
    \begin{subfigure}[t]{0.32\linewidth}
    \centering
        \includegraphics[width=\linewidth]{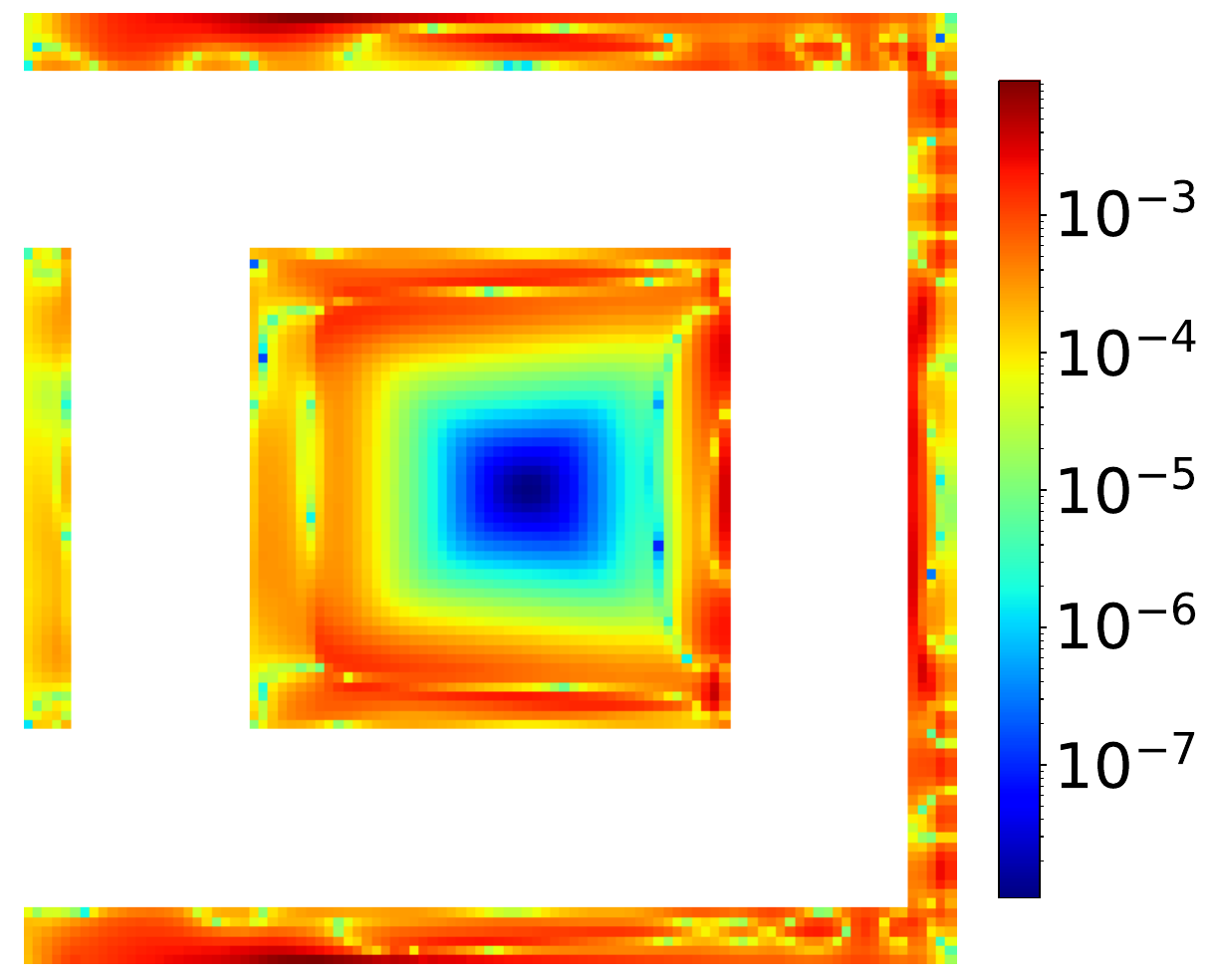}
        \caption{$\lvert T^{\mathrm{Full}} - T^{\vartheta}\rvert$}
    \end{subfigure}
    \caption{Material temperature at $t = 1~\mathrm{ns}$ for the hohlraum test case.}
    \label{fig:hohlraumRadTemp}
\end{figure}

%%%%%%%%%%%%%%%%%%%%%%%%%%%%% Sub-section %%%%%%%%%%%%%%%%%%%%%%%%%%%%%%%%%%%%
\section*{Conclusions and outlook}
%%%%%%%%%%%%%%%%%%%%%%%%%%%%%%%%%%%%%%%%%%%%%%%%%%%%%%%%%%%%%%%%%%%%%%%%%%%%%%
In this work, we have proposed a low-rank scheme for the nonlinear thermal radiative transfer equations based on macro--micro decomposition. Th proposed macro--micro parallel BUG scheme is asymptotic--preserving, mass preserving, rank-adaptive, and energy stable for the nonlinear Stefan-Boltzmann closure under a mixed hyperbolic and parabolic CFL condition. In addition, we also propose an efficient algorithm to implement reflection-transmission type boundary conditions for the macro--micro parallel BUG scheme. Several experiments demonstrating the efficacy of the proposed schemes are presented. It is observed that the macro--micro scheme captures the solution to a high degree of accuracy while being computationally and memory efficient. 

%%%%%%%%%%%%%%%%%%%%%%%%%%%%% Section %%%%%%%%%%%%%%%%%%%%%%%%%%%%%%%%%%%%
\section*{Acknowledgement}
%%%%%%%%%%%%%%%%%%%%%%%%%%%%%%%%%%%%%%%%%%%%%%%%%%%%%%%%%%%%%%%%%%%%%%%%%%%%%%
The work of Chinmay Patwardhan was funded by the Deutsche Forschungsgemeinschaft (DFG, German Research Foundation) – Project-ID 258734477 – SFB 1173.

\printbibliography
\appendix

\begin{appendices}

\section{Proof of Theorem 2}\label{appendix:ProofNMMstab}
%%%%%%%%%%%%%%%%%%%%%%%%%%%%%%%%%%%%%%%%%%%%%%%%%%%%%%%%%%%%%%%%%%%%%%%%%%%%%%
To prove Theorem~\ref{thm:StabNMM} we first state and prove several auxiliary lemmas.

\begin{property}\label{lemma:sumofsquares}
    For any $\{ c_{i}\}_{i=1,\ldots,\Nx}\in\R $ and $\{d_{i}\}_{i=1,\ldots,\Nx}\in\R $ we have
    \begin{equation*}
        \sum_{i}c_{i}d_{i} = \frac{1}{2}\sum_{i}c^{2}_{i} + \frac{1}{2}\sum_{i}d^{2}_{i} - \frac{1}{2}\sum_{i}(c_{i}-d_{i})^{2}. 
    \end{equation*}
\end{property}

\begin{lemma}\label{lemma:IntegByPartsD0}
    Given any $\boldsymbol{\varphi}_{\hindexp{i}{1}j} = (\varphi_{\hindexp{i}{1}j,1},\ldots,\varphi_{\hindexp{i}{1}j,N_{q}})^{\top}\in\R^{N_{q}}$ and $u_{\hindexp{i}{1}j}\in\R$ we have
    \begin{align*}
        \sum_{ij}\left(\boldsymbol{\varphi}_{\hindexp{i}{1}j}\right)^{\top}\textbf{M}^{2}\mathds{1}\delta^{0}_{x}u_{\hindexp{i}{1}j} &= -\sum_{ij}\left(\mathcal{D}^{0}_{x}\boldsymbol{\varphi}_{ij}\right)^{\top}\textbf{M}^{2}\mathds{1}u_{ij},\\
        \sum_{ij}\left(\boldsymbol{\varphi}_{i\hindexp{j}{1}}\right)^{\top}\textbf{M}^{2}\mathds{1}\delta^{0}_{x}u_{i\hindexp{j}{1}} &= -\sum_{ij}\left(\mathcal{D}^{0}_{x}\boldsymbol{\varphi}_{\hindexp{i}{1}\hindexp{j}{1}}\right)^{\top}\textbf{M}^{2}\mathds{1}u_{\hindexp{i}{1}\hindexp{j}{1}},\\
        \sum_{ij}\left(\boldsymbol{\varphi}_{i\hindexp{j}{1}}\right)^{\top}\textbf{M}^{2}\mathds{1}\delta^{0}_{y}u_{i\hindexp{j}{1}} &= -\sum_{ij}\left(\mathcal{D}^{0}_{y}\boldsymbol{\varphi}_{\hindexp{i}{1}\hindexp{j}{1}}\right)^{\top}\textbf{M}^{2}\mathds{1}u_{\hindexp{i}{1}\hindexp{j}{1}},\\
        \sum_{ij}\left(\boldsymbol{\varphi}_{\hindexp{i}{1}j}\right)^{\top}\textbf{M}^{2}\mathds{1}\delta^{0}_{y}u_{\hindexp{i}{1}j} &= -\sum_{ij}\left(\mathcal{D}^{0}_{y}\boldsymbol{\varphi}_{ij}\right)^{\top}\textbf{M}^{2}\mathds{1}u_{ij}.
    \end{align*}
\end{lemma}
\begin{proof}
    Consider 
    \begin{align*}
        \sum_{ij}\left(\boldsymbol{\varphi}_{\hindexp{i}{1}j}\right)^{\top}\textbf{M}^{2}\mathds{1}\delta^{0}_{x}u_{\hindexp{i}{1}j} &= \sum_{\ell}\sum_{j}\sum_{i}w_{\ell}\varphi_{\hindexp{i}{1}j,\ell}\delta^{0}_{x}u_{\hindexp{i}{1}j}\\
        &= \frac{1}{\Delta x}\sum_{\ell}\sum_{j}\sum_{i}w_{\ell}\varphi_{\hindexp{i}{1}j,\ell } \left( u_{i+1,j} - u_{ij} \right)\\
        &= -\frac{1}{\Delta x}\sum_{\ell}\sum_{j}\sum_{i}w_{\ell} \left( \varphi_{\hindexp{i}{1}j,\ell} - \varphi_{\hindexp{i}{1}j,\ell } \right)u_{ij}\\
        &= -\frac{1}{\Delta x}\sum_{\ell}\sum_{j}\sum_{i}w_{\ell}\left(\mathcal{D}^{0}_{x}\varphi_{ij,\ell }\right)u_{ij}\\
        &= -\sum_{ij}\left(\mathcal{D}^{0}_{x}\boldsymbol{\varphi}_{ij}\right)^{\top}\textbf{M}^{2}\mathds{1}u_{ij}.
    \end{align*}
    The other relations can be shown along similar lines.
\end{proof}

\begin{lemma}\label{lemma:IntegByPartsDz}
    Given any $\boldsymbol{\varphi}_{\hindexp{i}{1}j} = (\varphi_{\hindexp{i}{1}j,1},\ldots,\varphi_{\hindexp{i}{1}j,N_{q}})^{\top}\in\R^{N_{q}}$ and $\boldsymbol{\psi}_{\hindexp{i}{1}j} = (\psi_{\hindexp{i}{1}j,1},\ldots,\psi_{\hindexp{i}{1}j,N_{q}})^{\top}\in\R^{N_{q}}$ the following is satisfied for:
    \begin{align*}
        \sum_{ij}\left(\mathcal{D}^{\pm}_{v}\boldsymbol{\varphi}_{\hindexp{i}{1}j}\right)^{\top}{\normalfont\textbf{S}}\boldsymbol{\psi}_{\hindexp{i}{1}j}\Deltaxy = -\sum_{ij}\left(\boldsymbol{\varphi}_{\hindexp{i}{1}j}\right)^{\top}{\normalfont\textbf{S}}\mathcal{D}^{\mp}_{v}\boldsymbol{\psi}_{\hindexp{i}{1}j}\Deltaxy,
    \end{align*}
    where ${\normalfont\textbf{S}}\in\R^{N_{q}\times N_{q}}$ is any symmetric matrix and $v \in \mathcal{J}_{\textbf{\textit{x}}}$. The same results hold at the spatial point $(x_{i},y_{\hindexp{j}{1}})$.
\end{lemma}
\begin{proof}
    The proof of the lemma follows along the lines of \Cref{lemma:IntegByPartsD0}.
\end{proof}

\begin{lemma} \label{lemma:AuxLemma1}
    For $v \in\mathcal{J}_{\textbf{\textit{x}}}$, we define
    \begin{equation*}
        \begin{aligned}
            \mathcal{E}_{v}  &= \frac{1}{2\pi}\left[ \sum_{ij}\left( \mathcal{D}^{0}_{v}\textbf{\textit{g}}^{n+1}_{ij} \right)^{\top}\textbf{M}^{2}\textbf{Q}_{v}\mathds{1}\left( \frac{1}{c}B^{n}_{ij} + \frac{\varepsilon^{2}}{c}h^{n}_{ij} - \frac{1}{c}B^{n+1}_{ij} - \frac{\varepsilon^{2}}{c}h^{n+1}_{ij}  \right) \right.\\
            &\quad\left.+  \sum_{ij}\left( \mathcal{D}^{0}_{v}\textbf{\textit{g}}^{n+1}_{\hindexp{i}{1}\hindexp{j}{1}} \right)^{\top}\textbf{M}^{2}\textbf{Q}_{v}\mathds{1}\left( \frac{1}{c}B^{n}_{\hindexp{i}{1}\hindexp{j}{1}} + \frac{\varepsilon^{2}}{c}h^{n}_{\hindexp{i}{1}\hindexp{j}{1}} - \frac{1}{c}B^{n+1}_{\hindexp{i}{1}\hindexp{j}{1}} - \frac{\varepsilon^{2}}{c}h^{n+1}_{\hindexp{i}{1}\hindexp{j}{1}}  \right) \right]\Deltaxy.
        \end{aligned}
    \end{equation*}
    Then, the following inequality holds
    \begin{equation*}
    \begin{aligned}
        \mathcal{E}_{v} &\leq \frac{\Delta t}{(2\pi)^{2}}\left[ \sum_{ij}\left[\left(\mathcal{D}^{0}_{v}\textbf{\textit{g}}^{n+1}_{ij} \right)^{\top}\textbf{M}^{2}\textbf{Q}_{v}\mathds{1} \right]^{2} + \sum_{ij}\left[\left(\mathcal{D}^{0}_{v}\textbf{\textit{g}}^{n+1}_{\hindexp{i}{1}\hindexp{j}{1}} \right)^{\top}\textbf{M}^{2}\textbf{Q}_{v}\mathds{1} \right]^{2}  \right]\Deltaxy\\
        &\qquad+ \frac{1}{4\Delta t}\norm{\frac{1}{c}B^{n} + \frac{\varepsilon^{2}}{c}h^{n} - \frac{1}{c}B^{n+1} - \frac{\varepsilon^{2}}{c}h^{n+1}}^{2}.
    \end{aligned}
    \end{equation*}
\end{lemma}
\begin{proof}
    Let 
    \begin{align*}
        a_{ij}^{v} &= \frac{1}{2\pi}\left( \mathcal{D}^{0}_{v}\textbf{\textit{g}}^{n+1}_{ij} \right)^{\top}\textbf{M}^{2}\textbf{Q}_{v}\mathds{1}\\
        b_{ij} &= \left( \frac{1}{c}B^{n}_{ij} + \frac{\varepsilon^{2}}{c}h^{n}_{ij} - \frac{1}{c}B^{n+1}_{ij} - \frac{\varepsilon^{2}}{c}h^{n+1}_{ij}  \right),
    \end{align*}
    then we can write $\mathcal{E}_{v}$ as
    \begin{equation*}
        \mathcal{E}_{v} = \sum_{ij}a_{ij}^{v}b_{ij}\Deltaxy + \sum_{ij}a_{\hindexp{i}{1}\hindexp{j}{1}}^{v}b_{\hindexp{i}{1}\hindexp{j}{1}}\Deltaxy.
    \end{equation*}
    Using Young's inequality  we get
    \begin{equation*}
        \mathcal{E}_{v} = \frac{1}{4\alpha}\sum_{ij}\left[(a_{ij}^{v})^{2} + (a_{\hindexp{i}{1}\hindexp{j}{1}}^{v})^{2}\right]\Deltaxy + \alpha\sum_{ij}\left[(b_{ij})^{2} + (b_{\hindexp{i}{1}\hindexp{j}{1}})^{2} \right]\Deltaxy.
    \end{equation*}
    Since 
    \begin{align*}
        \sum_{ij}\left[(b_{ij})^{2} + (b_{\hindexp{i}{1}\hindexp{j}{1}})^{2} \right]\Deltaxy &= \left[\sum_{ij}\left( \frac{1}{c}B^{n}_{ij} + \frac{\varepsilon^{2}}{c}h^{n}_{ij} - \frac{1}{c}B^{n+1}_{ij} - \frac{\varepsilon^{2}}{c}h^{n+1}_{ij}  \right)^{2} \right.\\
        &\qquad+ \left.\sum_{ij}\left( \frac{1}{c}B^{n}_{\hindexp{i}{1}\hindexp{j}{1}} + \frac{\varepsilon^{2}}{c}h^{n}_{\hindexp{i}{1}\hindexp{j}{1}} - \frac{1}{c}B^{n+1}_{\hindexp{i}{1}\hindexp{j}{1}} - \frac{\varepsilon^{2}}{c}h^{n+1}_{\hindexp{i}{1}\hindexp{j}{1}}  \right)^{2} \right]\Deltaxy \\
        &= \norm{ \frac{1}{c}B^{n} + \frac{\varepsilon^{2}}{c}h^{n} - \frac{1}{c}B^{n+1} - \frac{\varepsilon^{2}}{c}h^{n+1}}^{2},
    \end{align*}
    setting $\alpha = \frac{1}{4\Delta t}$ yields the required result.
\end{proof}
\begin{lemma}\label{lemma:AuxLemma2}
    If $\boldsymbol{\varphi}_{\hindexp{i}{1}j} \coloneqq (\varphi_{\hindexp{i}{1}j,1},\ldots,\varphi_{\hindexp{i}{1}j,N_{q}})^{\top}$ and $\boldsymbol{\varphi}_{i\hindexp{j}{1}} \coloneqq (\varphi_{i\hindexp{j}{1},1},\ldots,\varphi_{i\hindexp{j}{1},N_{q}})^{\top}$ in $\R^{N_{q}}$, then the following relations hold for the advection operator:
    \begin{align*}
        \mathcal{L}_{v}\boldsymbol{\varphi}_{\hindexp{i}{1}j} &= \textbf{Q}_{v}\mathcal{D}^{c}_{v}\boldsymbol{\varphi}_{\hindexp{i}{1}j} - \frac{\Delta v}{2}\abs{\textbf{Q}_{v}}\mathcal{D}^{-}_{v}\mathcal{D}^{+}_{v}\boldsymbol{\varphi}_{\hindexp{i}{1}j},\\
        \mathcal{L}_{v}\boldsymbol{\varphi}_{i\hindexp{j}{1}} &= \textbf{Q}_{v}\mathcal{D}^{c}_{v}\boldsymbol{\varphi}_{i\hindexp{j}{1}} - \frac{\Delta v}{2}\abs{\textbf{Q}_{v}}\mathcal{D}^{-}_{v}\mathcal{D}^{+}_{v}\boldsymbol{\varphi}_{i\hindexp{j}{1}},
    \end{align*}
   where 
    \begin{align*}
         \mathcal{D}^{c}_{x}\boldsymbol{\varphi}_{\hindexp{i}{1}j} &\coloneqq \frac{\boldsymbol{\varphi}_{\hindexp{i}{3}j} - \boldsymbol{\varphi}_{\hindexm{i}{1}j}}{2\Delta x}, & \mathcal{D}^{c}_{x}\boldsymbol{\varphi}_{i\hindexp{j}{1}} &\coloneqq \frac{\boldsymbol{\varphi}_{i+1,\hindexp{j}{1}} - \boldsymbol{\varphi}_{i-1,\hindexp{j}{1}}}{2\Delta x},
    \end{align*}
    and $\mathcal{D}^{c}_{y}$ is defined in a similar manner. 
\end{lemma}
\begin{proof}
    The relations can be shown by using the definition of the advection operator and $\mathcal{D}^{\pm}_{v}, \textbf{Q}^{\pm}_{v}$.   
\end{proof}

\begin{lemma}\label{lemma:AuxLemma4}
Let $\boldsymbol{\varphi}_{\hindexp{i}{1}j} = (\varphi_{\hindexp{i}{1}j,1},\ldots,\varphi_{\hindexp{i}{1}j,N_{q}})^{\top}\in\R^{N_{q}}$ and $\boldsymbol{\psi}_{\hindexp{i}{1}j} = (\psi_{\hindexp{i}{1}j,1},\ldots,\psi_{\hindexp{i}{1}j,N_{q}})^{\top}\in\R^{N_{q}}$, then we have the following inequality
    \begin{equation}
        \begin{aligned}
             \left\vert \sum_{ij}\left[\textbf{M}(\textbf{Q}^{+}_{v}\mathcal{D}^{+}_{v} \right.\right. &+ \left.\left. \textbf{Q}^{-}_{v}\mathcal{D}^{-}_{v})\boldsymbol{\varphi}_{\hindexp{i}{1}j} \right]^{\top}(\textbf{M}\boldsymbol{\psi}_{\hindexp{i}{1}j})\Deltaxy  \right\vert\\ &\leq \alpha\sum_{ij}(\boldsymbol{\psi}_{\hindexp{i}{1}j})^{\top}\textbf{M}^{2}\boldsymbol{\psi}_{\hindexp{i}{1}j}\Deltaxy \\ &\qquad+ \frac{1}{4\alpha}\sum_{ij}\left( \abs{\textbf{Q}_{v}}\mathcal{D}^{+}_{v}\boldsymbol{\varphi}_{\hindexp{i}{1}j} \right)^{\top}\textbf{M}^{2}\left( \abs{\textbf{Q}_{v}}\mathcal{D}^{+}_{v}\boldsymbol{\varphi}_{\hindexp{i}{1}j} \right)\Deltaxy
        \end{aligned}
    \end{equation}
\end{lemma}
\begin{proof}
    Using Young's inequality for the left-hand side we get
    \begin{equation*}\label{eq:NMMLem1eq1}
    \begin{aligned}
        \left\vert \sum_{ij}\left[\textbf{M}(\textbf{Q}^{+}_{v}\mathcal{D}^{+}_{v}  \right.\right. &+ \left.\left. \textbf{Q}^{-}_{v}\mathcal{D}^{-}_{v})\boldsymbol{\varphi}_{\hindexp{i}{1}j} \right]^{\top}(\textbf{M}\boldsymbol{\psi}_{\hindexp{i}{1}j})\Deltaxy  \right\vert\\ &\leq\alpha\sum_{ij}(\boldsymbol{\psi}_{\hindexp{i}{1}j})^{\top}\textbf{M}^{2}\boldsymbol{\psi}_{\hindexp{i}{1}j}\Deltaxy\\ &\qquad+ \frac{1}{4\alpha}\sum_{ij}\left[\textbf{M}(\textbf{Q}^{+}_{v}\mathcal{D}^{+}_{v} + \textbf{Q}^{-}_{v}\mathcal{D}^{-}_{v})\boldsymbol{\varphi}_{\hindexp{i}{1}j} \right]^{2}\Deltaxy.
    \end{aligned}
    \end{equation*}
    Now consider the second term on the right-hand side of the above inequality,
    \begin{equation}\label{eq:NMMLem1eq2}
    \begin{aligned}
        \sum_{ij}\left[\textbf{M}(\textbf{Q}^{+}_{v}\mathcal{D}^{+}_{v} + \textbf{Q}^{-}_{v}\mathcal{D}^{-}_{v})\boldsymbol{\varphi}_{\hindexp{i}{1}j} \right]^{2}\Deltaxy  &= \sum_{ij}(\mathcal{D}^{+}_{v}\boldsymbol{\varphi}_{\hindexp{i}{1}j})^{\top}\textbf{Q}^{+}_{v}\textbf{M}^{2}\textbf{Q}^{+}_{v}(\mathcal{D}^{+}_{v}\boldsymbol{\varphi}_{\hindexp{i}{1}j})\Deltaxy \\&+ \sum_{ij}(\mathcal{D}^{-}_{v}\boldsymbol{\varphi}_{\hindexp{i}{1}j})^{\top}\textbf{Q}^{-}_{v}\textbf{M}^{2}\textbf{Q}^{-}_{v}(\mathcal{D}^{-}_{v}\boldsymbol{\varphi}_{\hindexp{i}{1}j})\Deltaxy\\
        &+ 2\sum_{ij}(\mathcal{D}^{+}_{v}\boldsymbol{\varphi}_{\hindexp{i}{1}j})^{\top}\textbf{Q}^{+}_{v}\textbf{M}^{2}\textbf{Q}^{-}_{v}(\mathcal{D}^{-}_{v}\boldsymbol{\varphi}_{\hindexp{i}{1}j})\Deltaxy.
    \end{aligned}
    \end{equation}
    We have that $\textbf{Q}^{+}_{v} = (\textbf{Q}_{v} + \abs{\textbf{Q}_{v}})/2$, $\textbf{Q}^{-}_{v} = (\textbf{Q}_{v} - \abs{\textbf{Q}_{v}})/2$ and since $\textbf{Q}_{v}$ and $\textbf{M}$ are diagonal matrices  $$ \textbf{Q}^{+}_{v}\textbf{M}^{2}\textbf{Q}^{-}_{v} = 0.$$
    Moreover,
    \begin{align*}
        \sum_{ij}(\mathcal{D}^{-}_{v}\boldsymbol{\varphi}_{\hindexp{i}{1}j})^{\top}\textbf{Q}^{-}_{v}\textbf{M}^{2}\textbf{Q}^{-}_{v}(\mathcal{D}^{-}_{v}\boldsymbol{\phi}_{\hindexp{i}{1}j})\Deltaxy 
        % &= -\sum_{i}(\boldsymbol{\phi}_{\hindexp{i}{1}j})^{\top}\textbf{Q}^{-}_{z}\textbf{M}^{2}\textbf{Q}^{-}_{z}(\mathcal{D}^{+}_{z}\mathcal{D}^{-}_{z}\boldsymbol{\phi}_{\hindexp{i}{1}j})\Deltaxy\\
        % &= -\sum_{i}(\boldsymbol{\phi}_{\hindexp{i}{1}j})^{\top}\textbf{Q}^{-}_{z}\textbf{M}^{2}\textbf{Q}^{-}_{z}(\mathcal{D}^{-}_{z}\mathcal{D}^{+}_{z}\boldsymbol{\phi}_{\hindexp{i}{1}j})\Deltaxy\\
        &= \sum_{ij}(\mathcal{D}^{+}_{v}\boldsymbol{\varphi}_{\hindexp{i}{1}j})^{\top}\textbf{Q}^{-}_{v}\textbf{M}^{2}\textbf{Q}^{-}_{v}(\mathcal{D}^{+}_{v}\boldsymbol{\varphi}_{\hindexp{i}{1}j})\Deltaxy.
    \end{align*}
    Thus by substituting $\textbf{Q}^{+}_{v}$ and $\textbf{Q}^{-}_{v}$ in \eqref{eq:NMMLem1eq2}we get
    \begin{align}
        \sum_{ij}\left[\textbf{M}(\textbf{Q}^{+}_{v}\mathcal{D}^{+}_{v} + \textbf{Q}^{-}_{v}\mathcal{D}^{-}_{v})\boldsymbol{\varphi}_{\hindexp{i}{1}j} \right]^{2}\Deltaxy &= \frac{1}{2}\sum_{ij}(\mathcal{D}^{+}_{v}\boldsymbol{\varphi}_{\hindexp{i}{1}j})^{\top}\textbf{Q}_{v}\textbf{M}^{2}\textbf{Q}_{v}(\mathcal{D}^{+}_{v}\boldsymbol{\varphi}_{\hindexp{i}{1}j})\Deltaxy\\
        & + \frac{1}{2}\sum_{ij}(\mathcal{D}^{+}_{v}\boldsymbol{\varphi}_{\hindexp{i}{1}j})^{\top}\abs{\textbf{Q}_{v}}\textbf{M}^{2}\abs{\textbf{Q}_{v}}(\mathcal{D}^{+}_{v}\boldsymbol{\varphi}_{\hindexp{i}{1}j})\Deltaxy\\
        &\leq \sum_{ij}(\mathcal{D}^{+}_{v}\boldsymbol{\varphi}_{\hindexp{i}{1}j})^{\top}\abs{\textbf{Q}_{v}}\textbf{M}^{2}\abs{\textbf{Q}_{v}}(\mathcal{D}^{+}_{v}\boldsymbol{\varphi}_{\hindexp{i}{1}j})\Deltaxy.
    \end{align}
\end{proof}

\begin{lemma}\label{lemma:AuxLemma3}
    For the advection operators $\mathcal{L}_{v}, v \in \mathcal{J}_{\textbf{\textit{x}}}$, $\boldsymbol{\psi}_{\hindexp{i}{1}j} = (\psi_{\hindexp{i}{1}j,1},\ldots,\psi_{\hindexp{i}{1}j,N_{q}})^{\top}\in\R^{N_{q}}$, $\boldsymbol{\psi}_{i\hindexp{j}{1}} = (\psi_{i\hindexp{j}{1},1},\ldots,\psi_{i\hindexp{j}{1},N_{q}})^{\top}\in\R^{N_{q}}$ the following result holds
    \begin{equation}\label{eq:AuxLem3eq1}
        \left[ \sum_{ij}\left(\boldsymbol{\psi}^{n+1}_{\hindexp{i}{1}j} \right)^{\top}\textbf{M}^{2}\mathcal{L}_{v}\boldsymbol{\psi}^{n}_{\hindexp{i}{1}j} + \sum_{ij}\left(\boldsymbol{\psi}^{n+1}_{i\hindexp{j}{1}} \right)^{\top}\textbf{M}^{2}\mathcal{L}_{v}\boldsymbol{\psi}^{n}_{i\hindexp{j}{1}} \right]\Deltaxy  = \mathcal{A}_{v} + \mathcal{B}_{v},
    \end{equation}
    where
    \begin{align*}
        \mathcal{A}_{v} &= \frac{\Delta v}{2}\left[ \sum_{ij}\left(\mathcal{D}^{+}_{v}\boldsymbol{\psi}^{n+1}_{\hindexp{i}{1}j} \right)^{\top}\textbf{M}^{2}\abs{\textbf{Q}_{v}}\mathcal{D}^{+}_{v}\boldsymbol{\psi}^{n+1}_{\hindexp{i}{1}j} + \sum_{ij}\left(\mathcal{D}^{+}_{v}\boldsymbol{\psi}^{n+1}_{i\hindexp{j}{1}} \right)^{\top}\textbf{M}^{2}\abs{\textbf{Q}_{v}}\mathcal{D}^{+}_{v}\boldsymbol{\psi}^{n+1}_{i\hindexp{j}{1}} \right]\Deltaxy,\\
        \mathcal{B}_{v} &= -\left[ \sum_{ij}\left( (\textbf{Q}_{v}^{+}\mathcal{D}^{+}_{v} + \textbf{Q}_{v}^{-}\mathcal{D}^{-}_{v}) \boldsymbol{\psi}^{n+1}_{\hindexp{i}{1}j}\right)^{\top}\textbf{M}^{2}\left( \boldsymbol{\psi}^{n}_{\hindexp{i}{1}j} - \boldsymbol{\psi}^{n+1}_{\hindexp{i}{1}j} \right)\right.\\
        &\left.\qquad\qquad + \sum_{ij}\left( (\textbf{Q}_{v}^{+}\mathcal{D}^{+}_{v} + \textbf{Q}_{v}^{-}\mathcal{D}^{-}_{v}) \boldsymbol{\psi}^{n+1}_{i\hindexp{j}{1}}\right)^{\top}\textbf{M}^{2}\left( \boldsymbol{\psi}^{n}_{i\hindexp{j}{1}} - \boldsymbol{\psi}^{n+1}_{i\hindexp{j}{1}} \right)   \right]\Deltaxy.
    \end{align*}
    Additionally, we have
    \begin{equation*}\normalfont
        \mathcal{B}_{v} \geq -\frac{\varepsilon}{4c\Delta t}\norm{\boldsymbol{\Psi}^{n} - \boldsymbol{\Psi}^{n+1}}^{2} - \frac{c\Delta t}{\varepsilon}\norm{\abs{\textbf{Q}_{v}}\mathcal{D}^{+}_{v}\boldsymbol{\Psi}^{n+1} }^{2},
    \end{equation*}
    where $\norm{\boldsymbol{\Psi}} = \sum_{ij}\sum_{\ell}\psi_{ij,\ell}w_{\ell}$.
\end{lemma}
\begin{proof}
     First, we re-write the first term on the left-hand side of \eqref{eq:AuxLem3eq1} as
    \begin{equation}\label{eq:AuxLem3eq2}
        \begin{split}
            \sum_{ij}\left(\boldsymbol{\psi}^{n+1}_{\hindexp{i}{1}j} \right)^{\top}\textbf{M}^{2}\mathcal{L}_{v}\boldsymbol{\psi}^{n}_{\hindexp{i}{1}j}\Deltaxy &= \sum_{ij}\left(\boldsymbol{\psi}^{n+1}_{\hindexp{i}{1}j} \right)^{\top}\textbf{M}^{2}\mathcal{L}_{v}\boldsymbol{\psi}^{n+1}_{\hindexp{i}{1}j}\Deltaxy\\&+ \sum_{ij}\left(\boldsymbol{\psi}^{n+1}_{\hindexp{i}{1}j} \right)^{\top}\textbf{M}^{2}\mathcal{L}_{v}\left(\boldsymbol{\psi}^{n}_{\hindexp{i}{1}j} - \boldsymbol{\psi}^{n+1}_{\hindexp{i}{1}j}\right)\Deltaxy.
        \end{split}
    \end{equation}
    Using \Cref{lemma:AuxLemma2}, the first term on the right-hand side of \eqref{eq:AuxLem3eq2} is given by
    \begin{align*}
        \sum_{ij}\left(\boldsymbol{\psi}^{n+1}_{\hindexp{i}{1}j} \right)^{\top}\textbf{M}^{2}\mathcal{L}_{v}\boldsymbol{\psi}^{n+1}_{\hindexp{i}{1}j}\Deltaxy &= \sum_{ij}\left(\boldsymbol{\psi}^{n+1}_{\hindexp{i}{1}j} \right)^{\top}\textbf{M}^{2}\textbf{Q}_{v}\mathcal{D}^{c}_{v}\boldsymbol{\psi}^{n+1}_{\hindexp{i}{1}j}\Deltaxy\\& - \frac{\Delta v}{2}\sum_{ij}\left(\boldsymbol{\psi}^{n+1}_{\hindexp{i}{1}j} \right)^{\top}\textbf{M}^{2}\abs{\textbf{Q}_{v}}\mathcal{D}^{-}_{v}\mathcal{D}^{+}_{v}\boldsymbol{\psi}^{n+1}_{\hindexp{i}{1}j}\Deltaxy\\
         &= \frac{\Delta v}{2}\sum_{ij}\left(\mathcal{D}^{+}_{v}\boldsymbol{\psi}^{n+1}_{\hindexp{i}{1}j} \right)^{\top}\textbf{M}^{2}\abs{\textbf{Q}_{v}}\mathcal{D}^{+}_{v}\boldsymbol{\psi}^{n+1}_{\hindexp{i}{1}j}\Deltaxy,
    \end{align*}
    where we use \Cref{lemma:IntegByPartsDz} and
    \begin{align*}
        \sum_{ij}\left(\boldsymbol{\psi}^{n+1}_{\hindexp{i}{1}j} \right)^{\top}\textbf{M}^{2}\textbf{Q}_{v}\mathcal{D}^{c}_{v}\boldsymbol{\psi}^{n+1}_{\hindexp{i}{1}j}\Deltaxy &= \frac{1}{\Delta x}\sum_{ij}\left(\boldsymbol{\psi}^{n+1}_{\hindexp{i}{1}j} \right)^{\top}\textbf{M}^{2}\textbf{Q}_{v}\boldsymbol{\psi}^{n+1}_{\hindexp{i}{3}j}\Deltaxy \\&- \frac{1}{\Delta x}\sum_{ij}\left(\boldsymbol{\psi}^{n+1}_{\hindexp{i}{1}j} \right)^{\top}\textbf{M}^{2}\textbf{Q}_{v}\boldsymbol{\psi}^{n+1}_{\hindexm{i}{1}j}\Deltaxy\\
        &\overset{\text{shifting index}}{=} 0.
    \end{align*}
    Using \Cref{lemma:IntegByPartsDz} for the second term on the right-hand side of \eqref{eq:AuxLem3eq2} we get 
    \begin{equation}\label{eq:AuxLem3eq3}
        \begin{split}
            \sum_{ij}\left(\boldsymbol{\psi}^{n+1}_{\hindexp{i}{1}j} \right)^{\top}\textbf{M}^{2}\mathcal{L}_{v}&\left(\boldsymbol{\psi}^{n}_{\hindexp{i}{1}j} - \boldsymbol{\psi}^{n+1}_{\hindexp{i}{1}j}\right)\Deltaxy\\ &= \sum_{ij}\left( (\textbf{Q}_{v}^{+}\mathcal{D}^{+}_{v} + \textbf{Q}_{v}^{-}\mathcal{D}^{-}_{v})\boldsymbol{\psi}^{n+1}_{\hindexp{i}{1}j}  \right)^{\top}\textbf{M}^{2}\left(\boldsymbol{\psi}^{n}_{\hindexp{i}{1}j} - \boldsymbol{\psi}^{n+1}_{\hindexp{i}{1}j}\right)\Deltaxy.
        \end{split}
    \end{equation}
    We get similar expressions at $(x_{i},y_{\hindexp{j}{1}})$. Combining with the expressions at $(x_{\hindexp{i}{1}},y_{j})$ yields the first part of the lemma. To prove the remaining lemma, we use \Cref{lemma:AuxLemma4} for the right-hand side of \eqref{eq:AuxLem3eq3} which gives
    \begin{equation}\label{eq:AuxLem3eq4}
        \begin{aligned}
             -\sum_{ij}\left( (\textbf{Q}_{v}^{+}\mathcal{D}^{+}_{v} \right. &+ \left. \textbf{Q}_{v}^{-}\mathcal{D}^{-}_{v})\boldsymbol{\psi}^{n+1}_{\hindexp{i}{1}j}  \right)^{\top}\textbf{M}^{2}\left(\boldsymbol{\psi}^{n}_{\hindexp{i}{1}j} - \boldsymbol{\psi}^{n+1}_{\hindexp{i}{1}j}\right)\Deltaxy\\&\geq - \alpha\sum_{ij}\left(\boldsymbol{\psi}^{n}_{\hindexp{i}{1}j} - \boldsymbol{\psi}^{n+1}_{\hindexp{i}{1}j}\right)^{\top}\textbf{M}^{2}\left(\boldsymbol{\psi}^{n}_{\hindexp{i}{1}j} - \boldsymbol{\psi}^{n+1}_{\hindexp{i}{1}j}\right)\Deltaxy \\&- \frac{1}{4\alpha}\sum_{ij}\left( \abs{\textbf{Q}_{v}}\mathcal{D}^{+}_{v}\boldsymbol{\psi}^{n+1}_{\hindexp{i}{1}j}\right)^{\top}\textbf{M}^{2}\left( \abs{\textbf{Q}_{v}}\mathcal{D}^{+}_{v}\boldsymbol{\psi}^{n+1}_{\hindexp{i}{1}j} \right)\Deltaxy.
        \end{aligned}
    \end{equation}
    Similarly, at $(x_{i},y_{\hindexp{j}{1}})$ we get
    \begin{equation}\label{eq:AuxLem3eq5}
        \begin{aligned}
             -\sum_{ij}\left( (\textbf{Q}_{v}^{+}\mathcal{D}^{+}_{v} \right. &+ \left. \textbf{Q}_{v}^{-}\mathcal{D}^{-}_{v})\boldsymbol{\psi}^{n+1}_{i\hindexp{j}{1}}  \right)^{\top}\textbf{M}^{2}\left(\boldsymbol{\psi}^{n}_{i\hindexp{j}{1}} - \boldsymbol{\psi}^{n+1}_{i\hindexp{j}{1}}\right)\Deltaxy \\ &\geq - \alpha\sum_{ij}\left(\boldsymbol{\psi}^{n}_{i\hindexp{j}{1}} - \boldsymbol{\psi}^{n+1}_{i\hindexp{j}{1}}\right)^{\top}\textbf{M}^{2}\left(\boldsymbol{\psi}^{n}_{i\hindexp{j}{1}} - \boldsymbol{\psi}^{n+1}_{i\hindexp{j}{1}}\right)\Deltaxy \\ &- \frac{1}{4\alpha}\sum_{ij}\left( \abs{\textbf{Q}_{v}}\mathcal{D}^{+}_{v}\boldsymbol{\psi}^{n+1}_{i\hindexp{j}{1}}\right)^{\top}\textbf{M}^{2}\left( \abs{\textbf{Q}_{v}}\mathcal{D}^{+}_{v}\boldsymbol{\psi}^{n+1}_{i\hindexp{j}{1}} \right)\Deltaxy.
        \end{aligned}
    \end{equation}
    Adding \eqref{eq:AuxLem3eq4} and \eqref{eq:AuxLem3eq5} and setting $\alpha = \frac{\varepsilon}{4c\Delta t}$ completes the proof of the lemma.
    %(\textbf{Q}_{v}^{-}\mathcal{D}^{+}_{v} + \textbf{Q}_{v}^{+}\mathcal{D}^{-}_{v})
\end{proof}

\begin{lemma}\label{lemma:AuxLemma5}
    Given any $\boldsymbol{\varphi}\in\R^{N_{q}}$ the following inequality holds for $v \in \mathcal{J}_{\textbf{\textit{x}}}$:
    \begin{equation*}
        \boldsymbol{\varphi}^{\top}\textbf{Q}_{v}\textbf{\textit{w}}\textbf{\textit{w}}^{\top}\textbf{Q}_{v}\boldsymbol{\varphi} \leq \pi\boldsymbol{\varphi}^{\top}\abs{\textbf{Q}_{v}}\textbf{M}^{2}\boldsymbol{\varphi}.
    \end{equation*}
\end{lemma}

\begin{lemma}\label{lemma:AuxLemma6}
    The following holds for $\boldsymbol{\varphi}_{\hindexp{i}{1}j},\boldsymbol{\varphi}_{i\hindexp{j}{1}}\in\R^{N_{q}}$ and $v \in \mathcal{J}_{\textbf{\textit{x}}}$,
    \begin{equation*}
        \sum_{ij}(\mathcal{D}^{+}_{v}\boldsymbol{\varphi}_{\hindexp{i}{1}j})^{\top}\textbf{M}^{2}\abs{\textbf{Q}_{v}}(\mathcal{D}^{+}_{v}\boldsymbol{\varphi}_{\hindexp{i}{1}j}) \leq \frac{4}{\Delta v^{2}}\sum_{i}\boldsymbol{\varphi}_{\hindexp{i}{1}j}^{\top}\textbf{M}^{2}\abs{\textbf{Q}_{v}}\boldsymbol{\varphi}_{\hindexp{i}{1}j},
    \end{equation*}
    and
    \begin{equation*}
        \sum_{ij}(\mathcal{D}^{+}_{v}\boldsymbol{\varphi}_{i\hindexp{j}{1}})^{\top}\textbf{M}^{2}\abs{\textbf{Q}_{v}}(\mathcal{D}^{+}_{v}\boldsymbol{\varphi}_{i\hindexp{j}{1}}) \leq \frac{4}{\Delta v^{2}}\sum_{i}\boldsymbol{\varphi}_{i\hindexp{j}{1}}^{\top}\textbf{M}^{2}\abs{\textbf{Q}_{v}}\boldsymbol{\varphi}_{i\hindexp{j}{1}}.
    \end{equation*}
\end{lemma}

\begin{lemma}\label{lemma:AuxLemma7}
    For $a,b > 0$ we have that the following inequality holds:
    \begin{displaymath}
        a^{4}(a-b) - \frac{1}{5}(a^{5} - b^{5}) \geq 0.
    \end{displaymath}
\end{lemma}
\begin{proof}
    For $a = b$ we get that the left-hand side of the inequality is zero and thus we consider the case $a\neq b$. Thus we have two cases $a<b$ or $a>b$ and without loss of generality we consider the case $a<b$ and prove the inequality by contradiction. So assume that the given inequality is not true and thus we have
    \begin{displaymath}
        (a-b)\cdot(4a^{4} - a^{3}b - a^{2}b^{2} - ab^{3} - b^{4}) < 0.
    \end{displaymath}
    Since $a<b$ we have $4a^{4} > a^{3}b + a^{2}b^{2} + ab^{3} + b^{4}$. However, we have
    \begin{displaymath}
        a^{3}b + a^{2}b^{2} + ab^{3} + b^{4} > 4a^{4},
    \end{displaymath}
    which is a contradiction. 
\end{proof}

Finally, we now give the proof of Theorem 2
\begin{proof}
    We denote the scalar flux by $\phi^{n}_{ij} = B^{n}_{ij} + \varepsilon^{2}h^{n}_{ij} $. Then, rearranging \eqref{eq:FDNMM1} yields
    \begin{equation}\label{eq:Thm2eq1}
        \frac{1}{\Delta t}\left[ \frac{1}{c}\phi^{n+1}_{ij} - \frac{1}{c}\phi^{n}_{ij} \right] + \frac{1}{2\pi}\textbf{\textit{w}}^{\top}\textbf{Q}_{x}\mathcal{D}^{0}_{x}\textbf{\textit{g}}^{n+1}_{ij} + \frac{1}{2\pi}\textbf{\textit{w}}^{\top}\textbf{Q}_{y}\mathcal{D}^{0}_{y}\textbf{\textit{g}}^{n+1}_{ij} = -\sigma^{a}_{ij}h^{n+1}_{ij}.
    \end{equation}
    % and similarly substituting \eqref{eq:FDNMM6} in \eqref{eq:FDNMM2} we get
    % \begin{equation}\label{eq:Thm2eq2}
    %     \begin{aligned}
    %         \frac{1}{\Delta t}\left[ \left(\frac{a}{4\pi}T^{n+1}_{\hindexp{i}{1}\hindexp{j}{1}} + \frac{\varepsilon^{2}}{c}h^{n+1}_{\hindexp{i}{1}\hindexp{j}{1}}\right) - \left(\frac{a}{4\pi}T^{n}_{\hindexp{i}{1}\hindexp{j}{1}} + \frac{\varepsilon^{2}}{c}h^{n}_{\hindexp{i}{1}\hindexp{j}{1}}\right) \right] + \frac{1}{4\pi}\textbf{\textit{w}}^{\top}\textbf{Q}_{x}\mathcal{D}^{0}_{x}\textbf{\textit{g}}^{n+1}_{\hindexp{i}{1}\hindexp{j}{1}}\\ + \frac{1}{4\pi}\textbf{\textit{w}}^{\top}\textbf{Q}_{y}\mathcal{D}^{0}_{y}\textbf{\textit{g}}^{n+1}_{\hindexp{i}{1}\hindexp{j}{1}} = -\sigma^{a}_{\hindexp{i}{1}\hindexp{j}{1}}h^{n+1}_{\hindexp{i}{1}\hindexp{j}{1}}.
    %     \end{aligned}
    % \end{equation}
    Multiplying \eqref{eq:Thm2eq1} by $\left( \frac{1}{c}\phi^{n+1}_{ij} \right)\Deltaxy$ and summing over $i,j$ 
    \begin{equation}\label{eq:Thm2eq3}
        \begin{split}
               &\frac{1}{2\Delta t}\sum_{ij}\left[ \left(\frac{1}{c}\phi^{n+1}_{ij}\right)^{2} - \left(\frac{1}{c}\phi^{n}_{ij}\right)^{2} + \left(\frac{1}{c}\phi^{n+1}_{ij} - \frac{1}{c}\phi^{n}_{ij}\right)^{2} \right]\Deltaxy\\& + \frac{1}{2\pi}\sum_{ij}(\mathcal{H}^{1}_{x})^{n+1,n+1}_{ij}\Deltaxy + \frac{1}{2\pi}\sum_{ij}(\mathcal{H}^{1}_{y})^{n+1,n+1}_{ij}\Deltaxy = -\sum_{ij}(\mathcal{H}^{2})^{n+1,n+1}_{ij}\Deltaxy,
        \end{split}
    \end{equation}
    where 
    \begin{displaymath}
        \begin{aligned}
            (\mathcal{H}^{1}_{v})^{n+1,n+1}_{ij} &= \textbf{\textit{w}}^{\top}\textbf{Q}_{v}\mathcal{D}^{0}_{v}\textbf{\textit{g}}^{n+1}_{ij}\left(\frac{1}{c}\phi^{n+1}_{ij}\right), \quad v \in \mathcal{J}_{\textbf{\textit{x}}},\\
            (\mathcal{H}^{2})^{n+1,n+1}_{ij} &= \sigma^{a}_{ij}h^{n+1}_{ij}\left( \frac{1}{c}\phi^{n+1}_{ij} \right).
        \end{aligned}
    \end{displaymath}
    The double temporal index on $(\mathcal{H}^{1}_{v})_{ij}^{n+1,n+1} $ are for $\textbf{\textit{g}}^{n+1}$ and $\phi^{n+1}$, respectively. A similar equality is obtained at $ (x_{\hindexp{i}{1}}, y_{\hindexp{j}{1}}) $; adding this to \eqref{eq:Thm2eq3} and summing over $i,j$ yields

    \begin{equation}\label{eq:Thm2eq5}
        \frac{1}{2\Delta t}\left[ \norm{\frac{1}{c}\boldsymbol{\phi}^{n+1}}^{2} - \norm{\frac{1}{c}\boldsymbol{\phi}^{n}}^{2}  +\norm{\frac{1}{c}\boldsymbol{\phi}^{n+1} - \frac{1}{c}\boldsymbol{\phi}^{n}}^{2} \right] + \frac{1}{2\pi}\left( \mathcal{H}^{1} \right)^{n+1,n+1} = -\left( \mathcal{H}^{2} \right)^{n+1,n+1},
    \end{equation}
        where
    \begin{displaymath}
        \begin{aligned}
            \left( \mathcal{H}^{1}_{v} \right)^{n+1,n+1} &= \left[ \sum_{ij}\left( \mathcal{H}^{1}_{v} \right)^{n+1,n+1}_{ij} + \sum_{ij}\left( \mathcal{H}^{1}_{v} \right)^{n+1,n+1}_{\hindexp{i}{1}\hindexp{j}{1}} \right]\Deltaxy, \quad v \in \mathcal{J}_{\textbf{\textit{x}}},\\
            \left( \mathcal{H}^{1} \right)^{n+1,n+1} &= \left( \mathcal{H}^{1}_{x} \right)^{n+1,n+1} + \left( \mathcal{H}^{1}_{y} \right)^{n+1,n+1},\\
            \left( \mathcal{H}^{2} \right)^{n+1,n+1} &= \left[ \sum_{ij}\left( \mathcal{H}^{2} \right)^{n+1,n+1}_{ij} + \sum_{ij}\left( \mathcal{H}^{2} \right)^{n+1,n+1}_{\hindexp{i}{1}\hindexp{j}{1}} \right]\Deltaxy.
        \end{aligned}
    \end{displaymath}
    Similarly, multiplying \eqref{eq:FDNMM3} by $\left(\textbf{\textit{g}}^{n+1}_{\hindexp{i}{1}j} \right)^{\top}\textbf{M}^{2}\Deltaxy $ and the equivalent equation at $( x_{i},y_{\hindexp{j}{1}} )$ by $\left(\textbf{\textit{g}}^{n+1}_{i\hindexp{j}{1}} \right)^{\top}\textbf{M}^{2}\Deltaxy $, summing over $i,j$ and adding them yields
    \begin{equation}\label{eq:Thm2eq6}
        \begin{aligned}
            \frac{1}{2\Delta t}\left[ \frac{1}{c}\norm{\textbf{g}^{n+1}}^{2} - \frac{1}{c}\norm{\textbf{g}^{n}}^{2} + \frac{1}{c}\norm{\textbf{g}^{n+1} - \textbf{g}^{n}}^{2} \right]\hspace{3in}\\
            + \frac{1}{\varepsilon}\left[ \sum_{ij}\left(\textbf{\textit{g}}^{n+1}_{\hindexp{i}{1}j} \right)^{\top}\textbf{M}^{2}\left( \textbf{I} - \frac{1}{2\pi}\mathds{1}\textbf{\textit{w}}^{\top} \right)\mathcal{L}_{x}\textbf{\textit{g}}^{n+1}_{\hindexp{i}{1}j} + \sum_{ij}\left(\textbf{\textit{g}}^{n+1}_{i\hindexp{j}{1}} \right)^{\top}\textbf{M}^{2}\left( \textbf{I} - \frac{1}{2\pi}\mathds{1}\textbf{\textit{w}}^{\top} \right)\mathcal{L}_{x}\textbf{\textit{g}}^{n+1}_{i\hindexp{j}{1}} \right]\Deltaxy\\
            + \frac{1}{\varepsilon}\left[ \sum_{ij}\left(\textbf{\textit{g}}^{n+1}_{\hindexp{i}{1}j} \right)^{\top}\textbf{M}^{2}\left( \textbf{I} - \frac{1}{2\pi}\mathds{1}\textbf{\textit{w}}^{\top} \right)\mathcal{L}_{y}\textbf{\textit{g}}^{n+1}_{\hindexp{i}{1}j} + \sum_{ij}\left(\textbf{\textit{g}}^{n+1}_{i\hindexp{j}{1}} \right)^{\top}\textbf{M}^{2}\left( \textbf{I} - \frac{1}{2\pi}\mathds{1}\textbf{\textit{w}}^{\top} \right)\mathcal{L}_{y}\textbf{\textit{g}}^{n+1}_{i\hindexp{j}{1}} \right]\Deltaxy\\
            +\frac{c}{\varepsilon^{2}}\left(\mathcal{G}^{1} \right)^{n+1,n}
            = -\frac{1}{\varepsilon^{2}}\left[ \sum_{ij}\sigma^{t}_{\hindexp{i}{1}j}\left(\textbf{\textit{g}}^{n+1}_{\hindexp{i}{1}j} \right)^{\top}\textbf{M}^{2}\textbf{\textit{g}}^{n+1}_{\hindexp{i}{1}j} + \sum_{ij}\sigma^{t}_{i\hindexp{j}{1}}\left(\textbf{\textit{g}}^{n+1}_{i\hindexp{j}{1}} \right)^{\top}\textbf{M}^{2}\textbf{\textit{g}}^{n+1}_{i\hindexp{j}{1}} \right]\Deltaxy,
        \end{aligned}
    \end{equation}
    where 
    \begin{displaymath}
        \begin{aligned}
            \left(\mathcal{G}^{1}_{v} \right)^{n+1,n}_{\hindexp{i}{1}j} &= \left(\textbf{\textit{g}}^{n+1}_{\hindexp{i}{1}j} \right)^{\top}\textbf{M}^{2}\textbf{Q}_{v}\mathds{1}\delta^{0}_{v}\left(\frac{1}{c}\phi^{n}_{\hindexp{i}{1}j}\right),\quad v \in \mathcal{J}_{\textbf{\textit{x}}},\\
         \left(\mathcal{G}^{1}_{v} \right)^{n+1,n} &= \left[ \sum_{ij}\left(\mathcal{G}^{1}_{v} \right)^{n+1,n}_{\hindexp{i}{1}j} + \sum_{ij}\left(\mathcal{G}^{1}_{v} \right)^{n+1,n}_{i\hindexp{j}{1}}\right]\Deltaxy,\quad v \in \mathcal{J}_{\textbf{\textit{x}}},\\
         \left(\mathcal{G}^{1} \right)^{n+1,n} &= \left(\mathcal{G}^{1}_{x} \right)^{n+1,n} + \left(\mathcal{G}^{1}_{y} \right)^{n+1,n}.
        \end{aligned}
    \end{displaymath}
    Since $\sigma^{t}(\textbf{\textit{x}})\geq \sigma^{t}_{0} , \forall \textbf{\textit{x}} $, $\textbf{M}^{2}\mathds{1} = \textbf{\textit{w}} $, and $0 = \langle\textbf{\textit{g}}\rangle = \textbf{\textit{g}}^{\top}_{\hindexp{i}{1}j}\textbf{\textit{w}} = \textbf{\textit{g}}^{\top}_{i\hindexp{j}{1}}\textbf{\textit{w}} $, we get
    \begin{align*}
        -\frac{1}{\varepsilon^{2}}\left[ \sum_{ij}\sigma^{t}_{\hindexp{i}{1}j}\left(\textbf{\textit{g}}^{n+1}_{\hindexp{i}{1}j} \right)^{\top}\textbf{M}^{2}\textbf{\textit{g}}^{n+1}_{\hindexp{i}{1}j} + \sum_{ij}\sigma^{t}_{i\hindexp{j}{1}}\left(\textbf{\textit{g}}^{n+1}_{i\hindexp{j}{1}} \right)^{\top}\textbf{M}^{2}\textbf{\textit{g}}^{n+1}_{i\hindexp{j}{1}} \right]\Deltaxy &\leq -\frac{\sigma^{t}_{0}}{\varepsilon^{2}}\norm{\textbf{g}^{n+1}}^{2}\\
        \left(\textbf{\textit{g}}^{n+1}_{\hindexp{i}{1}j} \right)^{\top}\textbf{M}^{2}\mathds{1} = \left(\textbf{\textit{g}}^{n+1}_{\hindexp{i}{1}j} \right)^{\top}\textbf{\textit{w}} &= 0, \\
        \left(\textbf{\textit{g}}^{n+1}_{i\hindexp{j}{1}} \right)^{\top}\textbf{M}^{2}\mathds{1} = \left(\textbf{\textit{g}}^{n+1}_{i\hindexp{j}{1}} \right)^{\top}\textbf{\textit{w}} &= 0.
    \end{align*}
      Hence \eqref{eq:Thm2eq6} reduces to
    \begin{equation}\label{eq:Thm2eq7}
        \begin{split}
            \frac{1}{2\Delta t}\left[ \frac{1}{c}\norm{\textbf{g}^{n+1}}^{2} - \frac{1}{c}\norm{\textbf{g}^{n}}^{2} + \frac{1}{c}\norm{\textbf{g}^{n+1} - \textbf{g}^{n}}^{2} \right]
            + \frac{1}{\varepsilon}\left(\mathcal{G}^{2}\right)^{n+1,n+1}\\
            +\frac{c}{\varepsilon^{2}}\left(\mathcal{G}^{1} \right)^{n+1,n}\leq -\frac{\sigma^{t}_{0}}{\varepsilon^{2}}\norm{\textbf{g}^{n+1}}^{2},
        \end{split}
    \end{equation}
    where
    \begin{displaymath}
        \begin{aligned}
            \left(\mathcal{G}^{2}_{v}\right)^{n+1,n+1}_{\hindexp{i}{1}j} &= \left(\textbf{\textit{g}}^{n+1}_{\hindexp{i}{1}j} \right)^{\top}\textbf{M}^{2}\mathcal{L}_{v}\textbf{\textit{g}}^{n+1}_{\hindexp{i}{1}j}, \quad v \in \mathcal{J}_{\textbf{\textit{x}}},\\
            \left(\mathcal{G}^{2}_{v}\right)^{n+1,n+1} &= \left[ \sum_{ij}\left(\mathcal{G}^{2}_{v}\right)^{n+1,n+1}_{\hindexp{i}{1}j} + \sum_{ij}\left(\mathcal{G}^{2}_{v}\right)^{n+1,n+1}_{i\hindexp{j}{1}} \right]\Deltaxy,\quad v \in \mathcal{J}_{\textbf{\textit{x}}},\\
            \left(\mathcal{G}^{2}\right)^{n+1,n+1} &= \left(\mathcal{G}^{2}_{x}\right)^{n+1,n+1} + \left(\mathcal{G}^{2}_{y}\right)^{n+1,n+1}.
        \end{aligned}
    \end{displaymath}
    Adding, \eqref{eq:Thm2eq5} + $\frac{\varepsilon^{2}}{2\pi c}$\eqref{eq:Thm2eq7} yields 
    \begin{equation}\label{eq:Thm2eq8}
        \begin{aligned}
            \frac{1}{2\Delta t}\left[ \norm{\frac{1}{c}\boldsymbol{\phi}^{n+1}}^{2} \right.& + \left. \frac{1}{2\pi}\norm{\frac{\varepsilon}{c}\textbf{g}^{n+1}}^{2}- \norm{\frac{1}{c}\boldsymbol{\phi}^{n}}^{2} - \frac{1}{2\pi}\norm{\frac{\varepsilon}{c}\textbf{g}^{n}}^{2} +\right.\\ &\left. \norm{\frac{1}{c}\boldsymbol{\phi}^{n+1} - \frac{1}{c}\boldsymbol{\phi}^{n}}^{2} + \frac{1}{2\pi}\norm{\frac{\varepsilon}{c}\textbf{g}^{n+1} - \frac{\varepsilon}{c}\textbf{g}^{n}}^{2} \right]
            + \frac{\varepsilon}{2\pi c}\left( \mathcal{G}^{2} \right)^{n+1,n+1}\\
            &+ \frac{1}{2\pi}\left( \mathcal{H}^{1} \right)^{n+1,n+1}
            \leq  -\frac{\sigma^{t}_{0}}{2\pi c}\norm{\textbf{g}^{n+1}}^{2} -\left( \mathcal{H}^{2} \right)^{n+1,n+1}
            -\frac{1}{2\pi}\left( \mathcal{G}^{1} \right)^{n+1,n}.
        \end{aligned}
    \end{equation}
    For $\left( \mathcal{H}^{1}\right)^{n+1,n+1}$, $ v \in \mathcal{J}_{\textbf{\textit{x}}} $, the following relation holds
    \begin{equation*}
        \sum_{ij}\left( \mathcal{H}^{1}_{v} \right)^{n+1,n+1}_{ij} = \sum_{ij}\left(\mathcal{D}^{0}_{v}\textbf{\textit{g}}^{n+1}_{ij} \right)^{\top}\textbf{M}^{2}\textbf{Q}_{v}\mathds{1}\left(\frac{1}{c}\phi^{n+1}_{ij} \right) =: \sum_{ij}\left( \widetilde{\mathcal{H}}^{1}_{v} \right)^{n+1,n+1}_{ij}
    \end{equation*}
     Using \Cref{lemma:IntegByPartsD0} for $\left( \mathcal{G}^{1} \right)^{n+1,n}$
     \begin{displaymath}
        \sum_{ij}\left( \mathcal{G}^{1}_{v} \right)^{n+1,n}_{\hindexp{i}{1}j} = -\sum_{ij}\left(\mathcal{D}^{0}_{v}\textbf{\textit{g}}^{n+1}_{ij} \right)^{\top}\textbf{M}^{2}\textbf{Q}_{v}\mathds{1}\left( \frac{1}{c}\phi^{n}_{ij} \right) = -\sum_{ij}\left( \widetilde{\mathcal{H}}^{1}_{v} \right)^{n+1,n}_{ij}.
    \end{displaymath}
    Then, we get
    \begin{equation}\label{eq:Thm2eq9}
        \begin{aligned}
              \frac{1}{2\Delta t}&\left[ \norm{\frac{1}{c}\boldsymbol{\phi}^{n+1}}^{2} + \frac{1}{2\pi}\norm{\frac{\varepsilon}{c}\textbf{g}^{n+1}}^{2}- \norm{\frac{1}{c}\boldsymbol{\phi}^{n}}^{2} - \frac{1}{2\pi}\norm{\frac{\varepsilon}{c}\textbf{g}^{n}}^{2} +\right.\\ &\left. \norm{\frac{1}{c}\boldsymbol{\phi}^{n+1} - \frac{1}{c}\boldsymbol{\phi}^{n}}^{2} + \frac{1}{2\pi}\norm{\frac{\varepsilon}{c}\textbf{g}^{n+1} - \frac{\varepsilon}{c}\textbf{g}^{n}}^{2} \right]\\
            &+ \frac{\varepsilon}{2\pi c}\left( \mathcal{G}^{2} \right)^{n+1,n+1}
            + \frac{1}{2\pi}\left( \widetilde{\mathcal{H}}^{1} \right)^{n+1,n+1}
            \leq  -\frac{\sigma^{t}_{0}}{2\pi c}\norm{\textbf{g}^{n+1}}^{2}\\ &\hspace{2.5in}-\left( \mathcal{H}^{2} \right)^{n+1,n+1} +\frac{1}{2\pi}\left( \widetilde{\mathcal{H}}^{1} \right)^{n+1,n}.
        \end{aligned}    
    \end{equation}
    Subtracting $ \left( \widetilde{\mathcal{H}}^{1} \right)^{n+1,n+1} $ from $ \left( \widetilde{\mathcal{H}}^{1} \right)^{n+1,n} $ and using \Cref{lemma:AuxLemma1} we get
    \begin{displaymath}
    \begin{split}
        \frac{1}{2\pi}\left[ \left( \widetilde{\mathcal{H}}^{1} \right)^{n+1,n} - \left( \widetilde{\mathcal{H}}^{1} \right)^{n+1,n+1} \right]  &= \frac{\Delta t}{(2\pi)^{2}}\left(\mathcal{G}^{3} \right)^{n+1,n+1}\\ &+ \frac{1}{2\Delta t}\norm{\frac{1}{c}\phi^{n} - \frac{1}{c}\phi^{n+1}}^{2},
    \end{split}
    \end{displaymath}
    where
    \begin{displaymath}
    \begin{split}
        \left(\mathcal{G}^{3} \right)^{n+1,n+1} = \left[ \sum_{ij}\left[\left(\mathcal{D}^{0}_{x}\textbf{\textit{g}}^{n+1}_{ij} \right)^{\top}\textbf{M}^{2}\textbf{Q}_{x}\mathds{1} \right]^{2} + \sum_{ij}\left[\left(\mathcal{D}^{0}_{x}\textbf{\textit{g}}^{n+1}_{\hindexp{i}{1}\hindexp{j}{1}} \right)^{\top}\textbf{M}^{2}\textbf{Q}_{x}\mathds{1} \right]^{2}  \right]\Deltaxy \\
            +\left[ \sum_{ij}\left[\left(\mathcal{D}^{0}_{y}\textbf{\textit{g}}^{n+1}_{ij} \right)^{\top}\textbf{M}^{2}\textbf{Q}_{y}\mathds{1} \right]^{2} + \sum_{ij}\left[\left(\mathcal{D}^{0}_{y}\textbf{\textit{g}}^{n+1}_{\hindexp{i}{1}\hindexp{j}{1}} \right)^{\top}\textbf{M}^{2}\textbf{Q}_{y}\mathds{1} \right]^{2}  \right]\Deltaxy.
    \end{split}
    \end{displaymath}
     Using \Cref{lemma:AuxLemma5} and the definition of the differential operators $\mathcal{D}^{0}_{v}$, we get for the last two terms on the right-hand side
    \begin{equation}\label{eq:Thm2eq17}
        \begin{aligned}
            \frac{\Delta t}{(2\pi)^{2}}\left(\mathcal{G}^{3} \right)^{n+1,n+1} \leq \frac{\Delta t}{4\pi}\left(\widetilde{\mathcal{G}}^{3}\right)^{n+1,n+1}. 
        \end{aligned}
    \end{equation}
    where 
    \begin{align*}
           \left(\widetilde{\mathcal{G}}^{3}_{v}\right)^{n+1,n+1}_{{\hindexp{i}{1}j}} &= \left( \mathcal{D}^{+}_{v}\textbf{\textit{g}}^{n+1}_{\hindexp{i}{1}j} \right)^{\top}\abs{\textbf{Q}_{v}}\textbf{M}^{2}\mathcal{D}^{+}_{x}\textbf{\textit{g}}^{n+1}_{\hindexp{i}{1}j},\quad v \in \mathcal{J}_{\textbf{\textit{x}}},\\
            \left(\widetilde{\mathcal{G}}^{3}_{v}\right)^{n+1,n+1} &= \left[ \sum_{ij} \left(\widetilde{\mathcal{G}}^{3}_{v}\right)^{n+1,n+1}_{\hindexp{i}{1}j} + \sum_{ij} \left(\widetilde{\mathcal{G}}^{3}_{v}\right)^{n+1,n+1}_{i\hindexp{j}{1}} \right]\Deltaxy,\quad v \in \mathcal{J}_{\textbf{\textit{x}}},\\
         \left(\widetilde{\mathcal{G}}^{3}\right)^{n+1,n+1} &= \left(\widetilde{\mathcal{G}}^{3}_{x}\right)^{n+1,n+1} + \left(\widetilde{\mathcal{G}}^{3}_{y}\right)^{n+1,n+1}.
    \end{align*}
    Putting it all together we get
     \begin{equation}\label{eq:Thm2eq10}
        \begin{aligned}
              \frac{1}{2\Delta t}\left[ \norm{\frac{1}{c}\boldsymbol{\phi}^{n+1}}^{2} + \frac{1}{2\pi}\norm{\frac{\varepsilon}{c}\textbf{g}^{n+1}}^{2} - \norm{\frac{1}{c}\boldsymbol{\phi}^{n}}^{2} - \frac{1}{2\pi}\norm{\frac{\varepsilon}{c}\textbf{g}^{n}}^{2} + \frac{1}{2\pi}\norm{\frac{\varepsilon}{c}\textbf{g}^{n+1} - \frac{\varepsilon}{c}\textbf{g}^{n}}^{2} \right]\\
              +  \frac{\varepsilon}{2\pi c}\left( \mathcal{G}^{2} \right)^{n+1,n+1}
             \leq  -\frac{\sigma^{t}_{0}}{2\pi c}\norm{\textbf{g}^{n+1}}^{2} -\left( \mathcal{H}^{2} \right)^{n+1,n+1} +\frac{\Delta t}{4\pi}\left(\widetilde{\mathcal{G}}^{3}\right)^{n+1,n+1}.
        \end{aligned}    
    \end{equation}
    Multiplying \eqref{eq:FDNMM5} by $\left( \frac{1}{c}B^{n+1}_{ij} \right)\Deltaxy$ and its equivalent at $(x_{\hindexp{i}{1}},y_{\hindexp{j}{1}})$ by $\left( \frac{1}{c}B^{n+1}_{\hindexp{i}{1}\hindexp{j}{1}} \right)\Deltaxy$, respectively, and summing over $i,j$ yields
    \begin{displaymath}
        \frac{1}{\Delta t}\sum_{ij}\frac{a c_{\nu}}{(2\pi)^{2}}(T^{n+1}_{ij})^{4}\left( T^{n+1}_{ij} - T^{n}_{ij} \right)\Deltaxy = \sum_{ij}\frac{1}{c}B^{n+1}_{ij}\sigma^{a}_{ij}h^{n+1}_{ij}\Deltaxy,
    \end{displaymath}
    \begin{displaymath}
        \frac{1}{\Delta t}\sum_{ij}\frac{a c_{\nu}}{(2\pi)^{2}}(T^{n+1}_{\hindexp{i}{1}\hindexp{j}{1}})^{4}\left( T^{n+1}_{\hindexp{i}{1}\hindexp{j}{1}} - T^{n}_{\hindexp{i}{1}\hindexp{j}{1}} \right)\Deltaxy = \sum_{ij}\frac{1}{c}B^{n+1}_{\hindexp{i}{1}\hindexp{j}{1}}\sigma^{a}_{\hindexp{i}{1}\hindexp{j}{1}}h^{n+1}_{\hindexp{i}{1}\hindexp{j}{1}}\Deltaxy.
    \end{displaymath}
    Using \Cref{lemma:AuxLemma7} we get 
    \begin{equation}\label{eq:Thm2eq11}
        \frac{1}{\Delta t}\sum_{ij}\frac{a c_{\nu}}{5(2\pi)^{2}}\left( \left(T^{n+1}_{ij}\right)^{5} - \left(T^{n}_{ij}\right)^{5} \right)\Deltaxy \leq \sum_{ij}\frac{1}{c}B^{n+1}_{ij}\sigma^{a}_{ij}h^{n+1}_{ij}\Deltaxy,
    \end{equation}
    \begin{equation}\label{eq:Thm2eq12}
        \frac{1}{\Delta t}\sum_{ij}\frac{a c_{\nu}}{5(2\pi)^{2}}\left( \left(T^{n+1}_{\hindexp{i}{1}\hindexp{j}{1}}\right)^{5} - \left(T^{n}_{\hindexp{i}{1}\hindexp{j}{1}}\right)^{5} \right)\Deltaxy \leq \sum_{ij}\frac{1}{c}B^{n+1}_{\hindexp{i}{1}\hindexp{j}{1}}\sigma^{a}_{\hindexp{i}{1}\hindexp{j}{1}}h^{n+1}_{\hindexp{i}{1}\hindexp{j}{1}}\Deltaxy.
    \end{equation}
    Combining \eqref{eq:Thm2eq11} and \eqref{eq:Thm2eq12}, yields
    \begin{equation}\label{eq:Thm2eq13}
        \begin{aligned}
            \frac{1}{2\Delta t}\left[ \frac{2}{5}\norm{\frac{\sqrt{a c_{\nu}}}{2\pi}(T^{n+1})^{5/2}}^{2} - \frac{2}{5}\norm{\frac{\sqrt{a c_{\nu}}}{2\pi}(T^{n})^{5/2}}^{2} \right]\hspace{1in}\\ \leq \frac{1}{c}\left[\sum_{ij}B^{n+1}_{ij}\sigma^{a}_{ij}h^{n+1}_{ij} + \sum_{ij}B^{n+1}_{\hindexp{i}{1}\hindexp{j}{1}}\sigma^{a}_{\hindexp{i}{1}\hindexp{j}{1}}h^{n+1}_{\hindexp{i}{1}\hindexp{j}{1}} \right]\Deltaxy.
        \end{aligned}
    \end{equation}
    Since $\sigma^{a}(\textbf{\textit{x}}) \geq \sigma^{a}_{0} $
    \begin{equation*}
        -\left[\sum_{ij}\frac{\varepsilon^{2}}{c}h^{n+1}_{ij} \sigma^{a}_{ij}h^{n+1}_{ij} + \sum_{ij}\frac{\varepsilon^{2}}{c}h^{n+1}_{\hindexp{i}{1}\hindexp{j}{1}}\sigma^{a}_{\hindexp{i}{1}\hindexp{j}{1}}h^{n+1}_{\hindexp{i}{1}\hindexp{j}{1}} \right]\Deltaxy \leq -\frac{\sigma^{a}_{0}}{c}\norm{\varepsilon h^{n+1}}^{2}.
    \end{equation*}
    Thus, adding \eqref{eq:Thm2eq10} and \eqref{eq:Thm2eq13} yields
    \begin{equation}\label{eq:Thm2eq14}
        \begin{aligned}
              \frac{1}{2\Delta t}\left[ e^{n+1} - e^{n} + \frac{1}{2\pi}\norm{\frac{\varepsilon}{c}\textbf{g}^{n+1} - \frac{\varepsilon}{c}\textbf{g}^{n}}^{2} \right]
              + \frac{\varepsilon}{2\pi c}\left( \mathcal{G}^{2} \right)^{n+1,n+1}\\
             \leq  -\frac{\sigma^{t}_{0}}{2\pi c}\norm{\textbf{g}^{n+1}}^{2}-\frac{\sigma^{a}_{0}}{c}\norm{\varepsilon h^{n+1}}^{2}
             +\frac{\Delta t}{4\pi}\left(\widetilde{\mathcal{G}}^{3}\right)^{n+1,n+1}.
        \end{aligned}    
    \end{equation}
    Note that we bound $-\frac{\sigma^{a}_{0}}{c}\norm{\varepsilon h^{n+1}}^{2} \leq 0$ and use \Cref{lemma:AuxLemma3} for $\left( \mathcal{G}^{2} \right)^{n+1,n+1}$, on the left-hand side of \eqref{eq:Thm2eq14}, such that $\frac{1}{2\pi}\norm{\frac{\varepsilon}{c}\textbf{g}^{n+1} - \frac{\varepsilon}{c}\textbf{g}^{n}}^{2}$ gets canceled out. This yields the following inequality
     \begin{equation}\label{eq:Thm2eq15}
        \begin{aligned}
              \frac{1}{2\Delta t}(e^{n+1} - e^{n})
              +\frac{\varepsilon\Delta x}{4\pi c}\left(\widetilde{\mathcal{G}}^{3}_{x}\right)^{n+1,n+1}
             +\frac{\varepsilon\Delta y}{4\pi c}\left(\widetilde{\mathcal{G}}^{3}_{y}\right)^{n+1,n+1}
              - \frac{\Delta t}{2\pi}\norm{\abs{\textbf{Q}_{x}}\mathcal{D}^{+}_{x}\textbf{g}^{n+1} }^{2}\\ - 
            \frac{\Delta t}{2\pi}\norm{\abs{\textbf{Q}_{y}}\mathcal{D}^{+}_{y}\textbf{g}^{n+1} }^{2}
             \leq  -\frac{\sigma^{t}_{0}}{2\pi c}\norm{\textbf{g}^{n+1}}^{2}
             +\frac{\Delta t}{4\pi}\left(\widetilde{\mathcal{G}}^{3}\right)^{n+1,n+1}.
        \end{aligned}    
    \end{equation}
   For any $\boldsymbol{\phi}\in\R^{N_{q}}$, since $\abs{\Omega_{v}}\leq 1 $, we have $ \boldsymbol{\phi}^{\top}\abs{\textbf{Q}_{v}}^{\top}\textbf{M}^{2}\abs{\textbf{Q}_{v}}\boldsymbol{\phi} \leq \boldsymbol{\phi}^{\top}\abs{\textbf{Q}_{v}}^{\top}\textbf{M}^{2}\boldsymbol{\phi} $. Hence we obtain
   \begin{equation}\label{eq:Thm2eq16}
        \norm{\abs{\textbf{Q}_{v}}\mathcal{D}^{+}_{v}\textbf{g}^{n+1} }^{2} \leq \left(\widetilde{\mathcal{G}}^{3}_{v}\right)^{n+1,n+1}.
   \end{equation}
    Combining \eqref{eq:Thm2eq16} and \eqref{eq:Thm2eq17} yields
    \begin{displaymath}
        \begin{aligned}
            \frac{1}{2\Delta t}(e^{n+1} - e^{n}) &\leq -\frac{\sigma^{t}_{0}}{2\pi c}\norm{\textbf{\textit{g}}^{n+1}}^{2}
            +\left(\frac{3\Delta t}{4\pi} - \frac{\varepsilon\Delta x}{4\pi c} \right)\left(\widetilde{\mathcal{G}}^{3}_{x}\right)^{n+1,n+1}\\
            &+ \left(\frac{3\Delta t}{4\pi} - \frac{\varepsilon\Delta y}{4\pi c} \right)\left(\widetilde{\mathcal{G}}^{3}_{y}\right)^{n+1,n+1}.
        \end{aligned}
    \end{displaymath}
    Using \Cref{lemma:AuxLemma6} and expanding the sum over the quadrature we get 
    \begin{equation*}
        \begin{aligned}
            \frac{1}{2\Delta t}(e^{n+1} - e^{n}) &\leq -\frac{\sigma^{t}_{0}}{2\pi c}\sum_{\ell}w_{\ell}\left[ \sum_{ij} g_{\ell,\hindexp{i}{1}j}^{2} + \sum_{ij} g_{\ell,i\hindexp{j}{1}}^{2} \right]\Deltaxy\\
            &\quad+\left(\frac{3\Delta t}{4\pi} - \frac{\varepsilon\Delta x}{4\pi c} \right)\frac{4}{\Delta x^{2}}\sum_{\ell}w_{\ell}\abs{\Omega_{x}^{\ell}}\left[ \sum_{ij} g_{\ell,\hindexp{i}{1}j}^{2} + \sum_{ij} g_{\ell,i\hindexp{j}{1}}^{2} \right]\Deltaxy \\
            &\quad+ \left(\frac{3\Delta t}{4\pi} - \frac{\varepsilon\Delta y}{4\pi c} \right)\frac{4}{\Delta y^{2}}\sum_{\ell}w_{\ell}\abs{\Omega_{y}^{\ell}}\left[ \sum_{ij} g_{\ell,\hindexp{i}{1}j}^{2} + \sum_{ij} g_{\ell,i\hindexp{j}{1}}^{2} \right]\Deltaxy,
        \end{aligned}
    \end{equation*}
    which can be written as 
    \begin{equation*}
        \begin{aligned}
            \frac{1}{2\Delta t}(e^{n+1} - e^{n}) &\leq \sum_{\ell}w_{\ell}\left[\left(\frac{3\Delta t}{4\pi} - \frac{\varepsilon\Delta x}{4\pi c} \right)\frac{4\abs{\Omega_{x}^{\ell}}}{\Delta x^{2}}  -\frac{\sigma^{t}_{0}}{4\pi c} \right]\left[ \sum_{ij} g_{\ell,\hindexp{i}{1}j}^{2} + \sum_{ij} g_{\ell,i\hindexp{j}{1}}^{2} \right]\Deltaxy \\
            &\quad+ \sum_{\ell}w_{\ell}\left[\left(\frac{3\Delta t}{4\pi} - \frac{\varepsilon\Delta y}{4\pi \vert c} \right)\frac{4\abs{\Omega_{y}^{\ell}}}{\Delta y^{2}}  -\frac{\sigma^{t}_{0}}{4\pi c} \right]\left[ \sum_{ij} g_{\ell,\hindexp{i}{1}j}^{2} + \sum_{ij} g_{\ell,i\hindexp{j}{1}}^{2} \right]\Deltaxy.
        \end{aligned}
    \end{equation*}
    Since the step size $\Delta t$ satisfies the CFL condition
    \begin{equation*}
        \Delta t \leq \frac{1}{3c}\text{min}\left\{ \varepsilon\Delta x + \frac{\sigma^{t}_{0}\Delta x^{2}}{4\abs{\Omega_{x}^{\ell}}}, \varepsilon\Delta y + \frac{\sigma^{t}_{0}\Delta y^{2}}{4\abs{\Omega_{y}^{\ell}}} \right\},
    \end{equation*}
     for $\abs{\Omega_{x}^{\ell}},\abs{\Omega_{y}^{\ell}}\neq 0$. We get $e^{n+1}\leq e^{n}$.
\end{proof}

% \section{Proof of \Cref{lemma:APraBUGAux1}}\label{appendix:ProofAPraBUGAux1}

\end{appendices}

\end{document}